\documentclass[11pt,leqno]{article}
\usepackage{amsmath}
\usepackage{amsthm}
\usepackage{amstext}
\usepackage{amsopn}
\usepackage{texdraw}
\usepackage{graphicx}
\newtheorem{theorem}{Theorem}[section]
\newtheorem{lemma}[theorem]{Lemma}

\theoremstyle{definition}
\newtheorem{definition}[theorem]{Definition}
\newtheorem{example}[theorem]{Example}

\theoremstyle{remark}
\newtheorem{remark}[theorem]{Remark}

\begin{document}
\title{ A continuum of periodic solutions to the planar four-body problem with various choices of masses }
\author{Tiancheng Ouyang \\
 Department of Mathematics, Brigham Young University\\
 Provo, Utah 84602, USA\\
 Email: ouyang@math.byu.edu\\Zhifu Xie %\footnote{Partially supported by RIG Grant (code 2137)
%from Virginia State University 2008-2009. %To appear in the Transactions of American Mathematical Society %and partially supported by an NSF grant 2009-2011.}
\\
Department of Mathematics and Computer Science\\
 Virginia State University\\
Petersburg, Virginia 23806, USA\\
 Email: zxie@vsu.edu \\
 }
\date{}

\maketitle
\begin{abstract}

In this paper, we apply the variational method with the Structural Prescribed Boundary Conditions (SPBC) to prove the existence of periodic and quasi-periodic solutions for planar four-body problem with $m_1=m_3$ and $m_2=m_4$. A path $q(t)$ in $[0,T]$ satisfies SPBC if the boundaries $q(0)\in \mathbf{A}$ and $q(T)\in \mathbf{B}$, where $\mathbf{A}$ and $\mathbf{B}$ are two structural configuration spaces in $(\mathbf{R}^2)^4$ and they depend on a rotation angle $\theta\in (0,2\pi)$ and the mass ratio $\mu=\frac{m_2}{m_1}\in \mathbf{R}^+$. %where $\mathbf{A}=\{[0 , -a_3; -a_1 , a_2; 0 , -2\mu a_2+a_3; a_1 , a_2 ]  R(\theta)|(a_1,a_2,a_3)\in\mathbf{R}^3\},$ $ \mathbf{B}=\{[ a_4,  a_5;0, -a_6; -a_4, a_5; 0, -\frac{2}{\mu}a_5+a_6]|(a_4,a_5,a_6)\in \mathbf{R}^3 $, and $R(\theta)=[\cos(\theta), -\sin(\theta); $ $\sin(\theta), $ $\cos(\theta)]$.

 We show that  there is a region $\Omega\subseteq (0,2\pi)\times R^+$ such that there exists at least one local minimizer of the Lagrangian action functional on the path space satisfying SPBC $\{q(t)\in H^1([0,T],$ $(\mathbf{R}^2)^4)| $  $q(0)\in $ $\mathbf{A}, q(T)\in  $ $\mathbf{B}\}$ for any $(\theta,\mu)\in \Omega$. %The minimizer has lower action than the action of the homographic solutions satisfying SPBC.
 The corresponding minimizing path of the minimizer can be extended to a non-homographic periodic solution if $\theta$ is  commensurable with $\pi$ or  a quasi-periodic  solution if $\theta$ is not commensurable with $\pi$. In the variational method with SPBC, we only impose constraints on boundary and we do not impose any symmetry constraint on solutions. Instead, we prove that our solutions extended from the initial minimizing pathes have the  symmetries. %But variational methods by minimizing action functional over loop space or by using Palais' principle of symmetric criticality have to impose symmetry constraints on solution space.

 The periodic solutions can be further classified as simple choreographic solutions, double choreographic solutions and non-choreographic solutions. Among the many stable simple choreographic orbits,  the most extraordinary one is the stable star pentagon choreographic solution when $(\theta,\mu)=(\frac{4\pi}{5},1)$. Remarkably the unequal-mass variants of the stable star pentagon are just as stable as the basic equal mass choreography (See figure \ref{fig1}). %The fact that unequal-mass variants of the stable start pentagon seem to be just as stable as the basic equal mass choreography makes the beautiful star pentagon orbit all the more remarkable!\\

\end{abstract}
{\bf AMS classification number:} 37N05, 70F10, 70F15, 37N30, 70H05, 70F17\\

\section{ Introduction}
%ZhifuSearch2012_36.m;
\begin{figure}
\includegraphics[height=5cm,width=.32\textwidth]{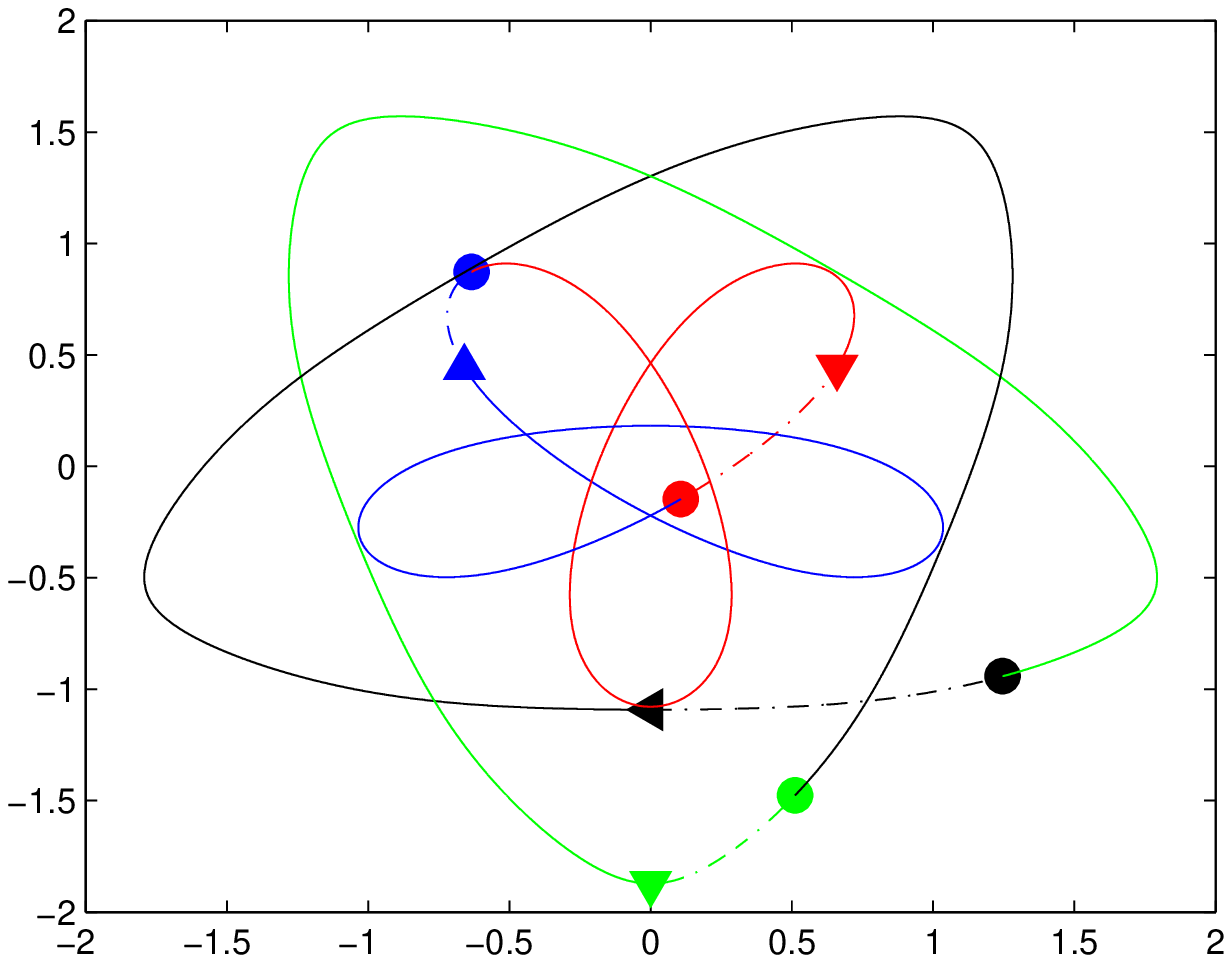}%{Vm03A4f5.eps}
\includegraphics[height=4.5cm,width=.33\textwidth]{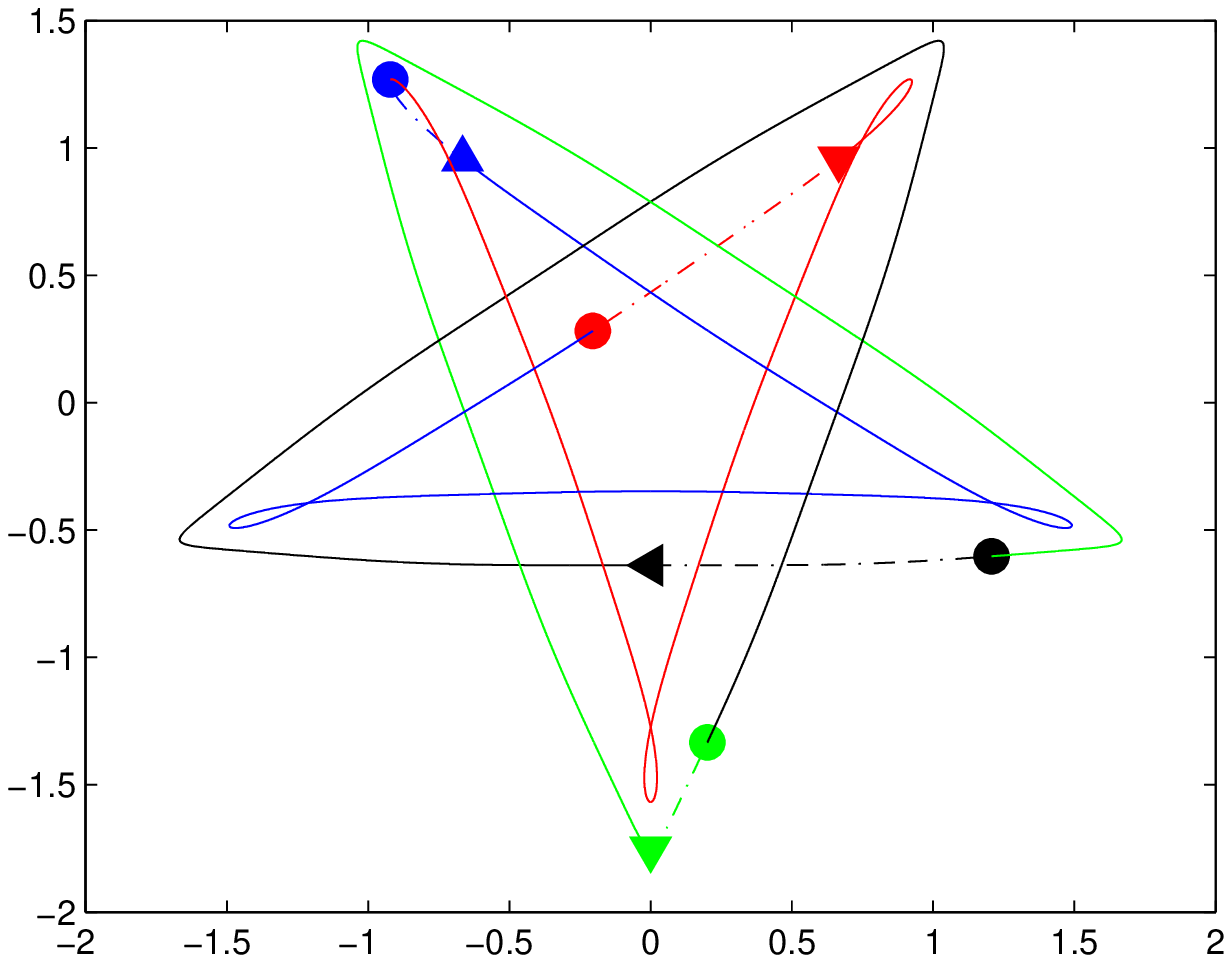}
\includegraphics[height=4.5cm,width=.33\textwidth]{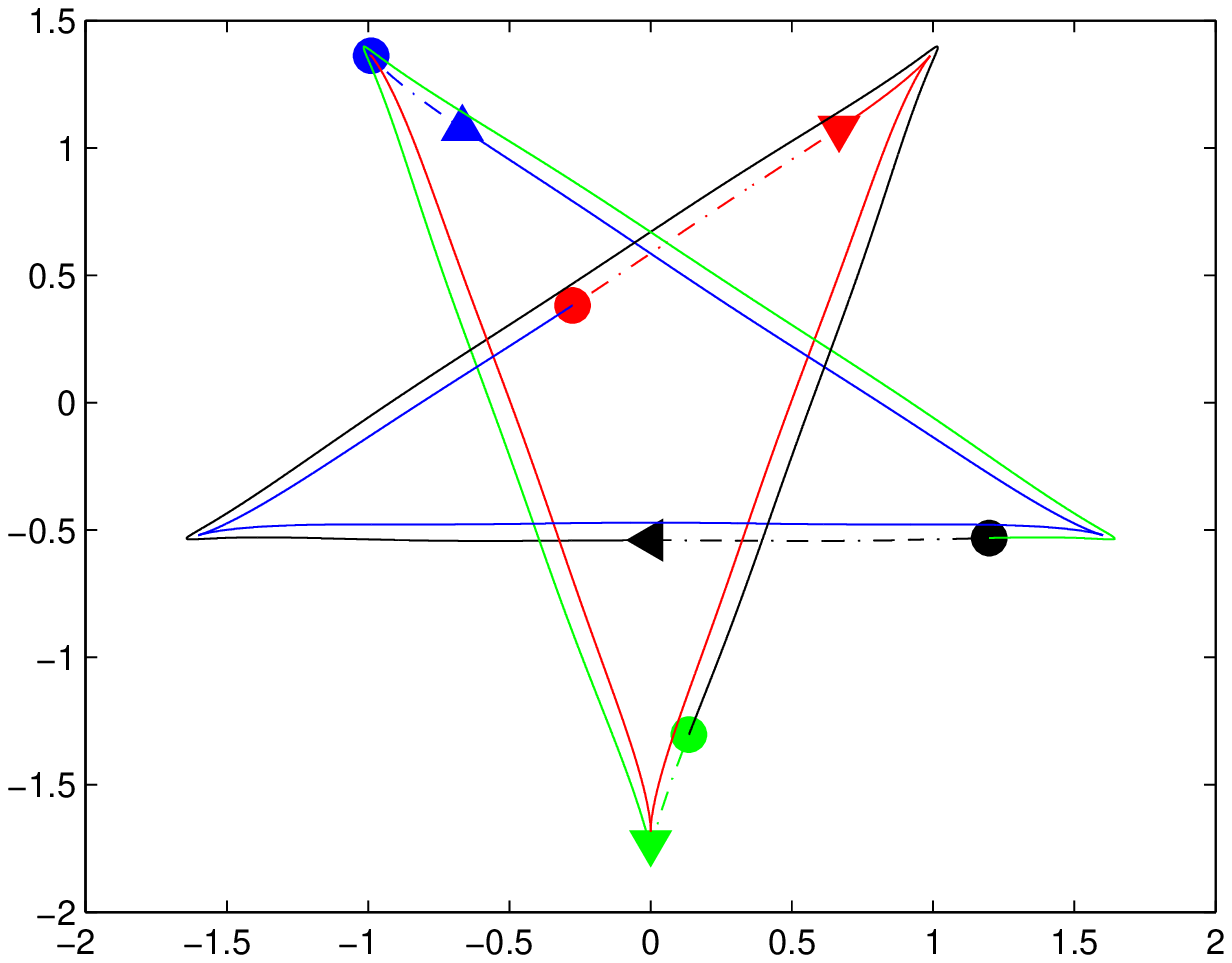}
\includegraphics[height=4.5cm,width=.33\textwidth]{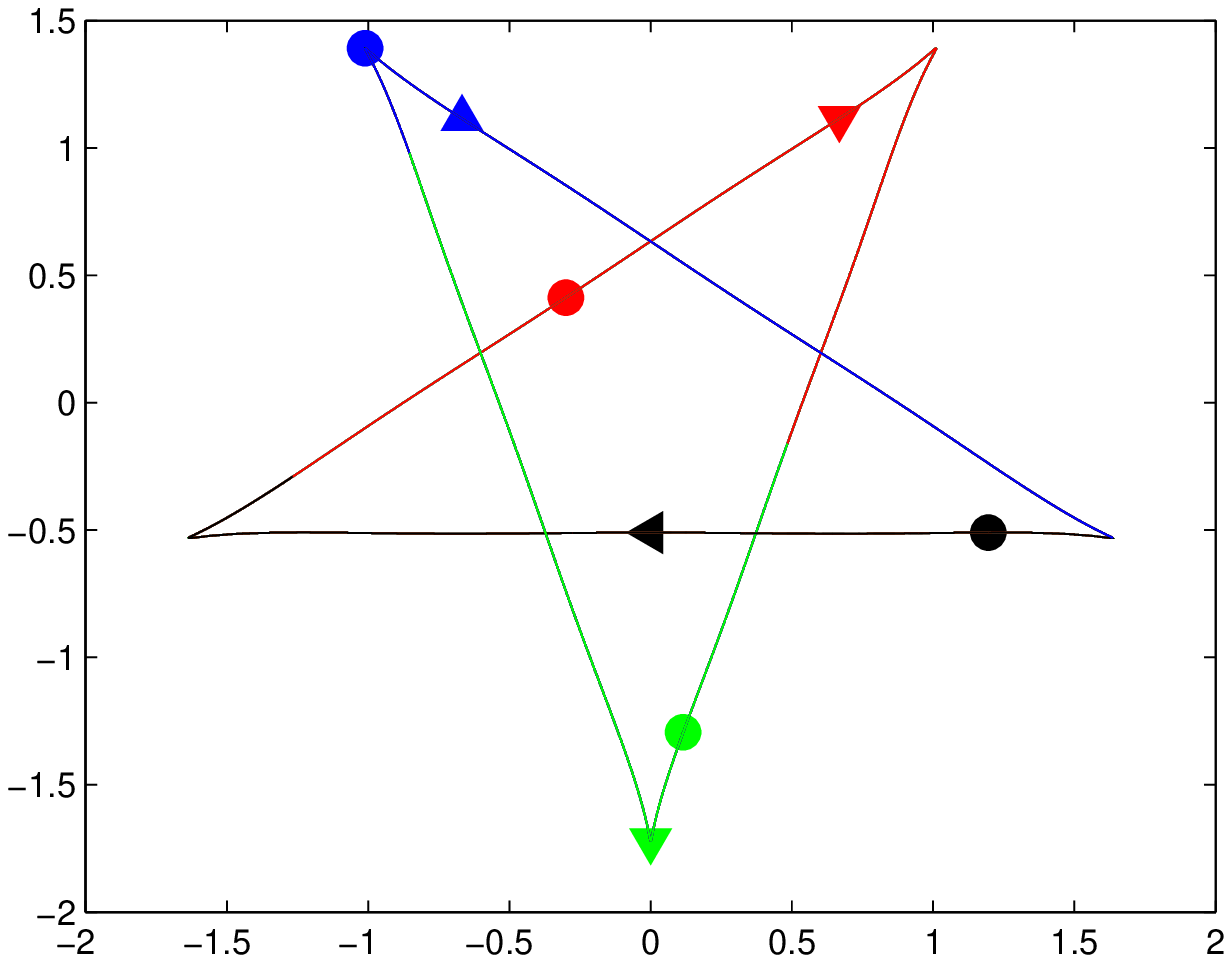}
\includegraphics[height=5cm,width=.32\textwidth]{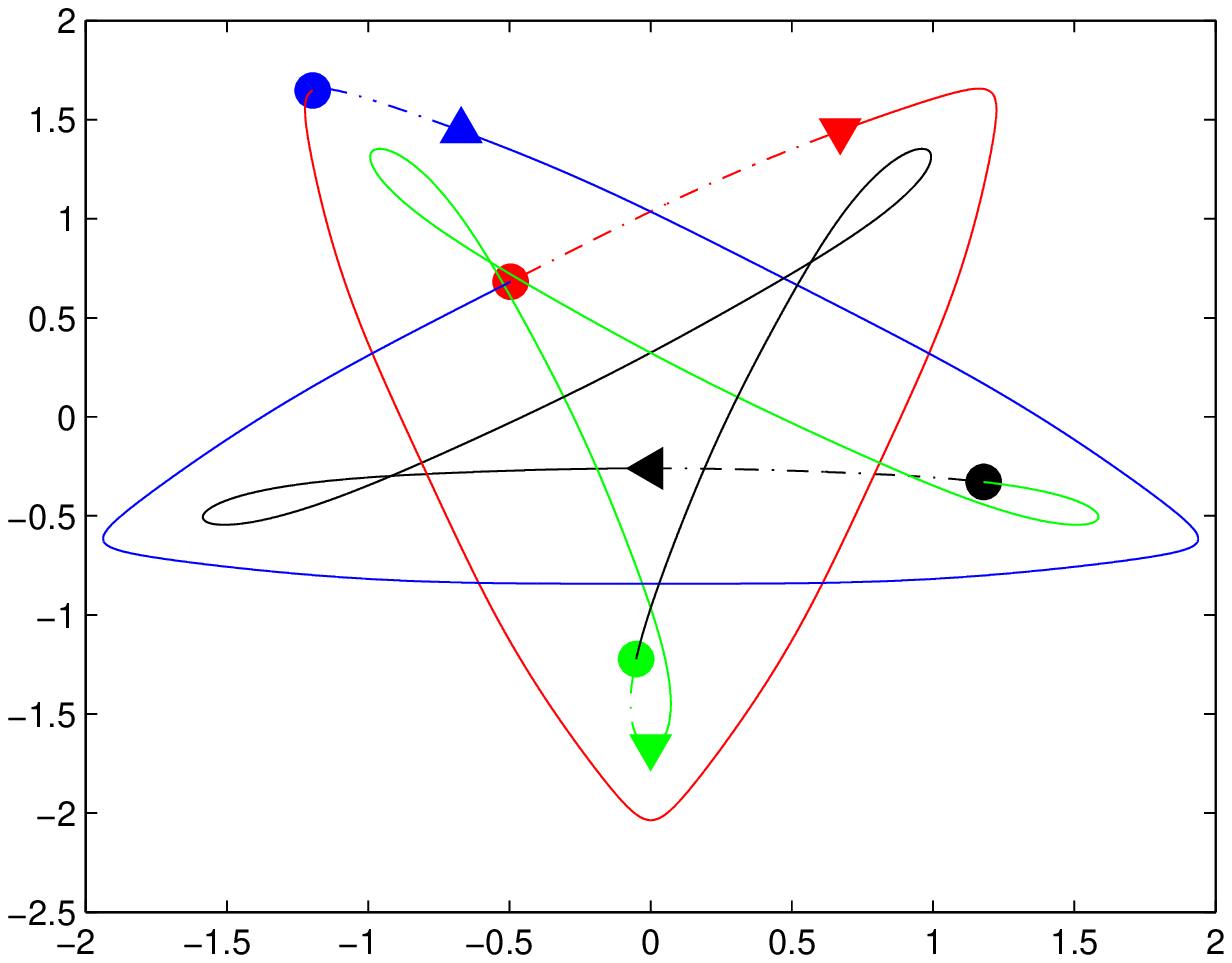}
\includegraphics[height=4.5cm,width=.32\textwidth]{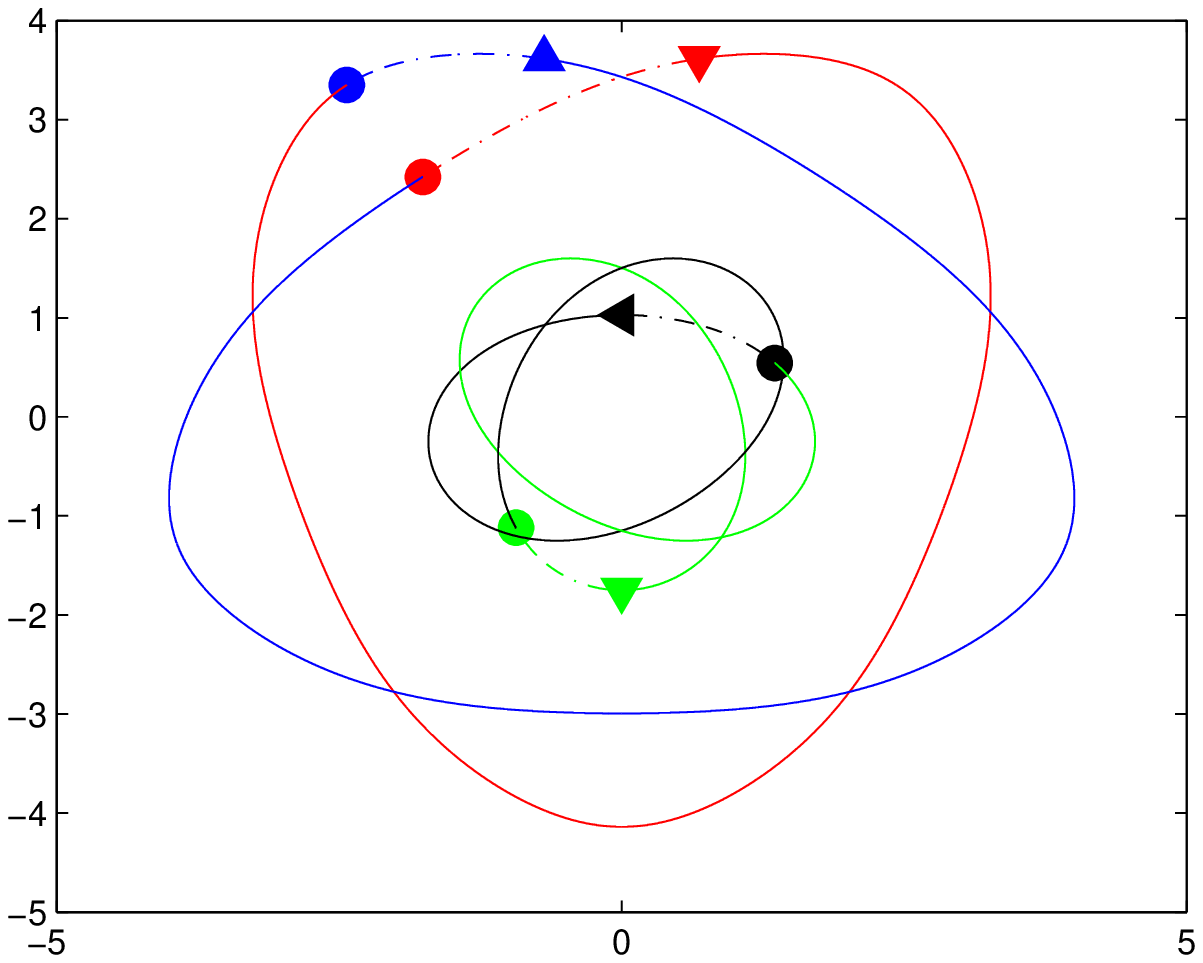}%{Vm5A4f5s.eps}%{Vm1_5A4f5.eps}
\caption{\small $\theta=\frac{4\pi}{5}$, $m_1=m_3=1, m_2=m_4=\mu$. From left to right $\mu=0.3,$ $0.8,$ $0.95 ,$ $ 1,$ $1.5,$ and $10.$ Solutions start from an isosceles triangle $q(0)$ (circular spots) with one in the axis of its symmetry to another isosceles triangle $q(T)$ (triangular spots).    }\label{fig1}\end{figure}
\begin{figure}
\includegraphics[height=5cm,width=.32\textwidth]{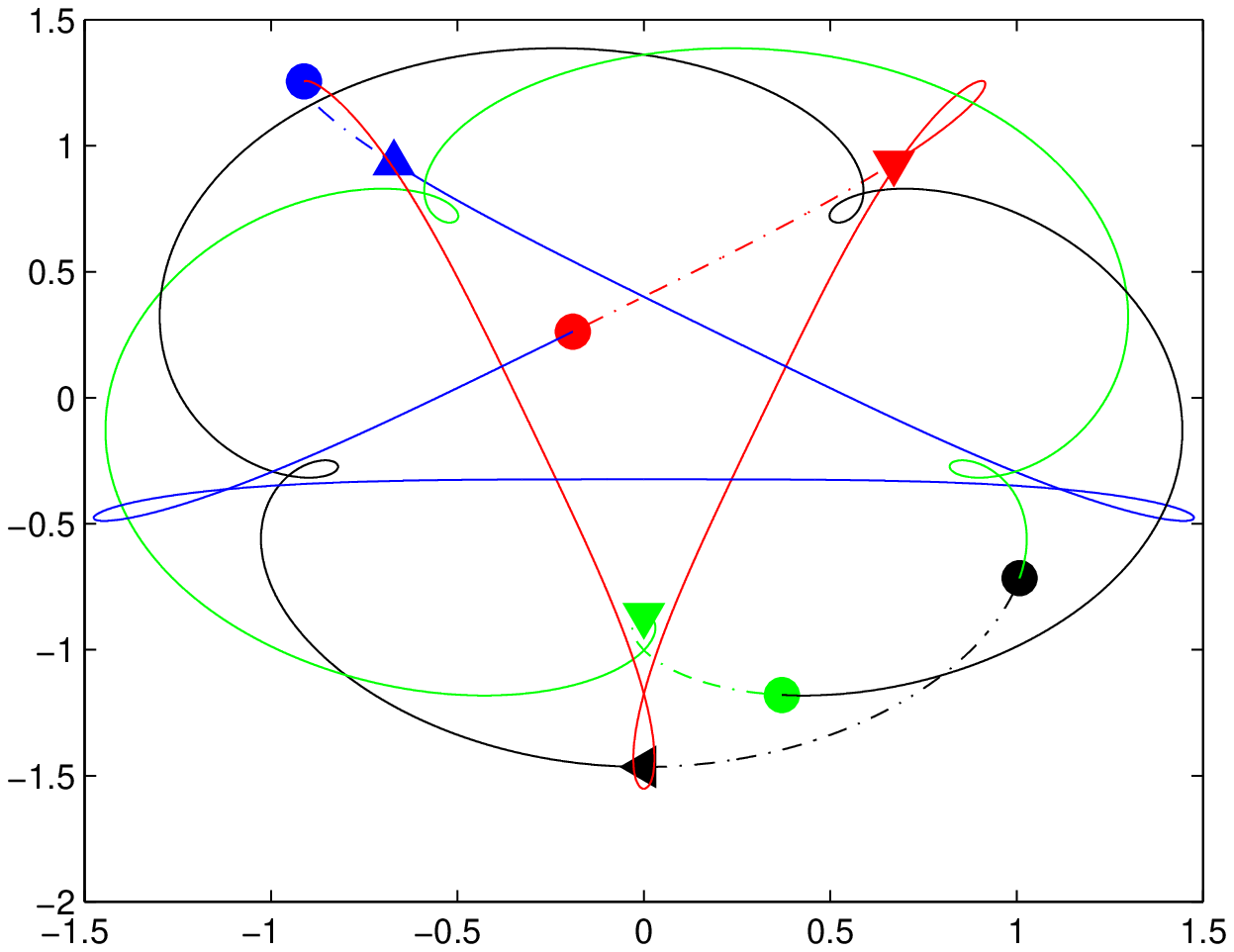}
\includegraphics[height=4.5cm,width=.33\textwidth]{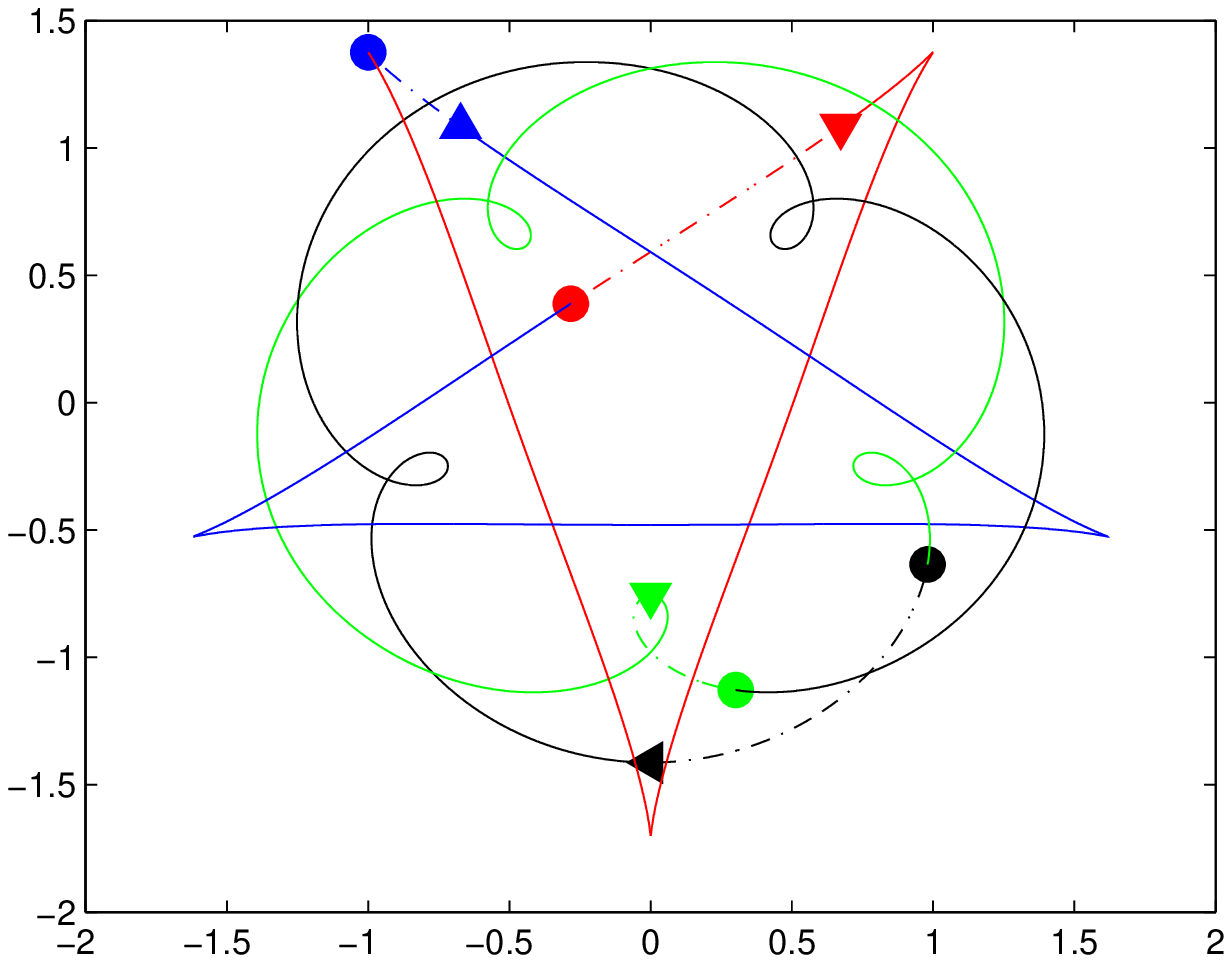}
\includegraphics[height=5cm,width=.32\textwidth]{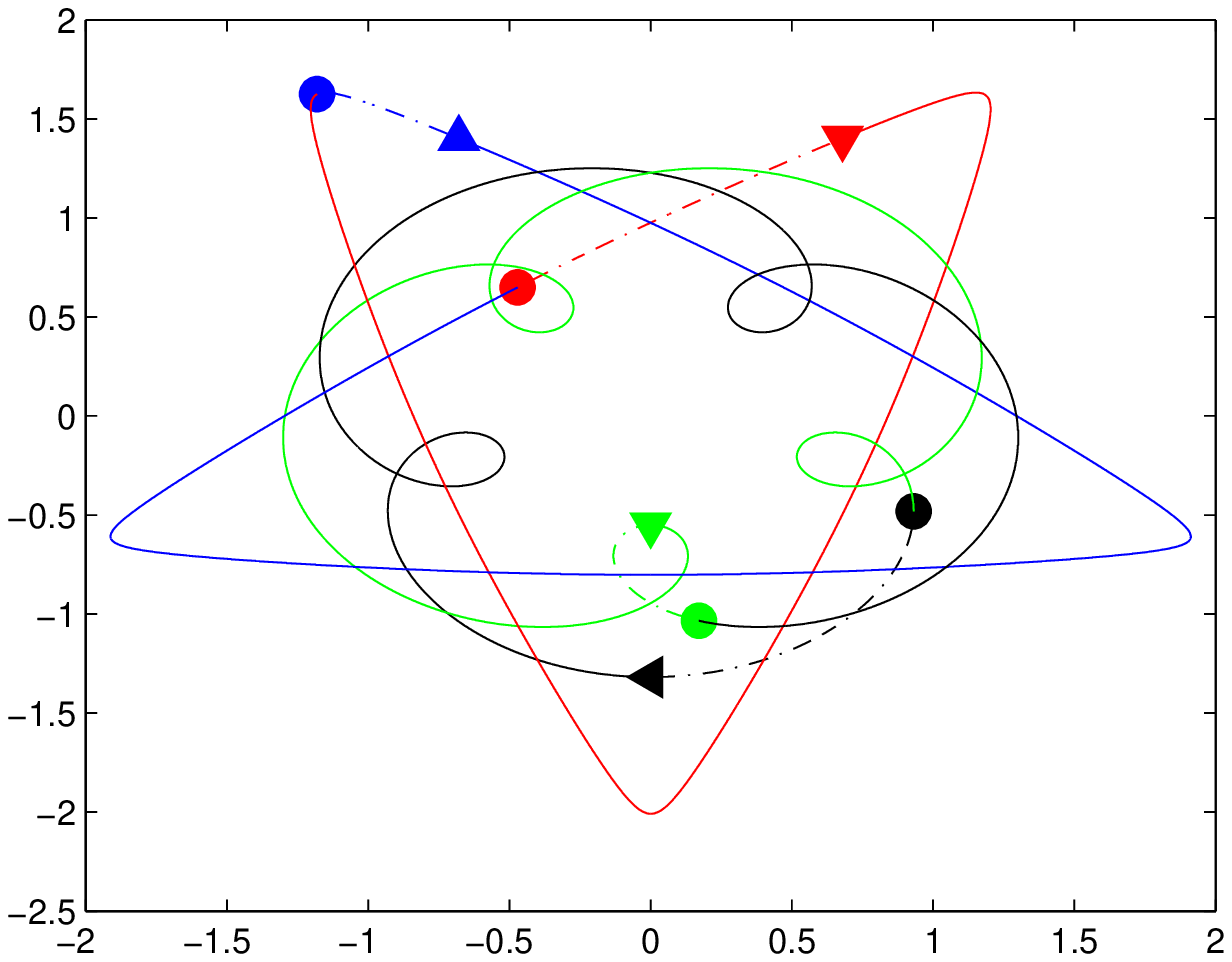}
\caption{\small $\theta=\frac{4\pi}{5}$, $m_1=m_3=1, m_2=m_4=\mu$. From left to right $\mu=0.8,$ $1,$  and $1.5.$ Solutions start from an isosceles triangle $q(0)$ (circular spots) with one in the axis of its symmetry to another isosceles triangle $q(T)$ (triangular spots).    }\label{fig111}\end{figure}

Given $n$ bodies, let $m_i$ denote the mass and $q_i(t)$ denote the position in $\mathbf{R}^d, d\geq 2$ of body $i$ at time $t$ in $d$-dimensional space. The {\it action functional} is a mapping from the space of all trajectories $q_1(t), q_2(t),\cdots, q_n(t)$ into the reals.  It is defined as the integral:% of the kinetic minus the potential energy:
\begin{equation}\label{Action}
\mathcal{A}(q(t))=\int_{0}^{T}K(\dot{q}(t))+U(q(t))dt,
\end{equation}
where $K(\dot{q}(t))=\frac{1}{2}\sum_{i=1}^{n} m_i|\dot{q}_i(t)|^2$ is the kinetic energy and $U$ is the Newtonian
potential function
%\begin{equation}\label{1.1b}
$U=\sum_{1\leq i<j\leq n}\frac{m_i m_j}{|{q}_i-{q}_j|}.$
%\end{equation}
Critical points of the action functional are trajectories that satisfy the equations of motion, i.e. Newton's equations:
\begin{equation}\label{Newton}
m_i\ddot{q}_i=\frac{\partial U}{\partial q_i}=\sum_{j=1,j\not= i}^{n} \frac{m_im_j(q_j-q_i)}{|q_j-q_i|^3} \hspace{1cm} 1\leq i\leq n.
\end{equation}
Without loss of generality, we assume that the center of mass $c=(1/M)\sum_{i=1}^{n}m_iq_i$ is always at the origin, where $M=\sum m_i$ is the total mass. Let $p_i=m_i\dot{q}_i$. Then the Hamiltonian of the Newton's equations is
\begin{equation}\label{Hamilton}
H(q,p)=\sum_{i=1}^{n}\frac{|p_i|^2}{2m_i} -U(q).
\end{equation}

In the past decade, the existence of many new interesting periodic orbits are proved by using variational method for the n-body problem. Most of them are found by minimizing the Lagrangian action on a symmetric loop space with some topological constraints (for example, see \cite{BT, BT2, Chen1, DengZhangZhou, FG, FT, TE, TV}).  %Only a few of them are able to study their dynamical properties such as stability because most of them are unable to provide some essential detail information such as initial positions and initial velocities of the orbits.  %But very few results were obtained on the general periodic function space because of lack of the structural information of solutions on the space. In addition,It is natural and typical to impose suitable symmetries or topological conditions such as equivariant group on the function space. The lack of coercivity of the action functional on the general periodic function space and the fact that collision orbits may give a bounded action, are the main mathematical obstructions in the search of periodic solutions of the $n$-body problem as minimizers of the action.Many new periodic solutions have been exploited

Following the notions in \cite{BT, CGMS},  a {\it simple choreographic solution} (for short, choreographic solution) is a periodic solution that all bodies chase one another along a single closed orbit. If the orbit of a periodic solution consists of two
 closed curves, then it is
called a {\it double-choreographic solution}.  If the orbit of a periodic solution consists of three
 closed curves, then it is
called a {\it triple-choreographic solution}. If the orbit of a periodic solution consists of different closed curves, each of which is the trajectory of exact one body,  it is called {\it non-choreographic solution}. Many relative equilibria give rise to simple choreographic solutions and they are called trivial choreographic solutions (circular motions). After the discovery of the  first remarkable non-trivial choreographic solution -- the figure eight of the three body problem  by Moore (1993 \cite{Moor}) and Chenciner and Montgomery (2000, \cite{CM}), many expertise attempt to study choreographic solutions and a large number of simple choreographic solutions have been discovered numerically  but very few of them have rigorous existence proofs. More results can be found in \cite{ABT,  BCPS, Br, CGMS, CV, Chen2, Ouyang1,DengZhangZhou} and the reference therein.

  In this paper, we are interested in the action minimizing solutions in path space satisfying the SPBC for planar four-body problem with unequal masses. We prove the existence of periodic and quasi-periodic  orbits for two pairs of unequal masses. We look for the continuum of the periodic and quasi-periodic solutions for equal masses discovered by the  variational method with SPBC developed in the recent paper \cite{OuXie}. Among the many stable simple choreographic orbits in \cite{OuXie},  the most extraordinary one is the stable star pentagon choreographic solution. The variants of star pentagon from equal mass to unequal mass is continuously deforming from  simple choreographic orbit to double choreographic orbits (see figure \ref{fig1}).

Figure eight is a remarkably  non-trivial simple choreographic solution, but more importantly, it is stable and the stability was proved in (\cite{KS,RG}).  It seems very hard to find a stable simple choreographic solution (C. Sim\'o \cite{SM} and R. Vanderbei \cite{Va}). To the best knowledge of the authors, all of the above known simple choreographic solutions are unstable except the figure eight \cite{ CM, Moor} and the family of star pentagon \cite{OuXie}. The journal {\it Science} had two articles \cite{MaD,Sc} on the figure-eight orbit. %with the titles \lqlq Triple Star Systems May Do Crazy Eights\rq\rq and \lq\lq Planetary ballet\rq\rq.
They deal with the idea that there could exist a planet system of equal masses. But the window of stability of figure-eight orbit is very small with slightly change of masses. Hence it seems unlikely that any real stars follow such an orbit.
Significantly different from the remarkable figure-eight orbit, the unequal-mass variants of the stable star pentagon seem to be just as stable as the basic equal mass choreography. This fact makes the beautiful star pentagon orbit all the more remarkable because such periodic solutions actually have more chance to be seen in some quadruple star system.

%The simulation of the solutions for the n-body problem can be found at http://sest.vsu.edu/$\sim$zxie/N\underline{\hbox{ }}body\underline{\hbox{ }}Simulation.htm.
In order to get a possible preassigned periodic orbit, we have to find an appropriate SPBC. Throughout the paper, we assume $m_1=m_3$ and $m_2=m_4$ and let $\mu=\frac{m_2}{m_1}$. Let $\Gamma=\mathbf{R}^6$ and the rotation matrix $R(\theta)= \left( \begin{array}{ll} \cos(\theta) & -\sin(\theta)\\
\sin(\theta) & \cos(\theta)\end{array} \right)$.

\begin{quote} {\bf \Large Our Settings on SPBC:} \\
Given $\vec{a}=(a_1,a_2,\cdots,a_6)\in \Gamma$, two fixed configurations are defined by $Qstart=   \left( \begin{array}{cc} 0 & -a_3\\
-a_1 & a_2\\ 0 & \frac{-m_2a_2-m_4a_2+m_1a_3}{m_3}\\
a_1 & a_2 \end{array} \right) R(\theta),$
and
$Qend=\left( \begin{array}{cc} a_4 & a_5\\
0 & -a_6\\ -a_4 & a_5\\
0 & \frac{-m_1a_5-m_3a_5+m_2 a_6}{m_4} \end{array} \right).$ Let
\begin{equation}\label{Fixpath}
 \mathcal{P}(Qstart,Qend):=\{q(t)\in H^1([0,T], (\mathbf{R}^2)^4) {\big | } q(0)=Qstart, q(T)= Qend\}.
 \end{equation}
 Then the set $S(\vec{a})$ of minimizers is defined by $$S(\vec{a})=\{ q(t) =(q_1, q_2, q_3, q_4)(t) \in C^2((0,T),(\mathbf{R^2})^4)\quad {\big | }\quad   q(0)=Qstart, q(T)=Qend,$$ $$ q(t) \hbox{ is a minimizer of the action functional } \mathcal{A} \hbox{ over } \mathcal{P}(Qstart,Qend) \}.$$
So the configuration of the bodies changes from an isosceles triangle with one on the axis of symmetry  to another isosceles triangle for some positive $\vec{a}$.\end{quote}

For any given $\vec{a}\in \Gamma$, the minimizers of $\mathcal{A}$ that connect $Qstart$ and $Qend$ are classical collision-free solutions in the interval $(0, T)$.  The motion starting from $Qstart$ to $Qend$ will continue beyond the $Qend$ under the universal gravitation but the continuation is hard to predict in general. The second minimizing process can find an appropriate $\vec{a}$ such that the motion can be extended in the way we expected.

The real value function $\tilde{\mathcal{A}}(\vec{a}): \Gamma\rightarrow \mathbf{R}$ is well defined by
 \begin{equation}\label{VAR1}
 \tilde{\mathcal{A}}(\vec{a}) = \int_{0}^{T}  \frac{1}{2}\sum_{i=1}^{n} m_i\|\dot{q}_i(t;\vec{a})\|^2+U(q(t;\vec{a}))dt,
 \end{equation}
where $q(t;\vec{a})\in S(\vec{a})$  is a minimizer of the action functional $\mathcal{A}$ over $\mathcal{P}(Qstart,Qend)$ for the given $\vec{a}\in \Gamma$. If it is clear that $q(t;\vec{a})$ is a minimizer for the given $\vec{a}$ from context, we still use $q(t)$ for $q(t;\vec{a})$ for convenience.

Let $\vec{a}_0=(a_{10},a_{20},\cdots,a_{60})\in \Gamma$ be a minimizer of $\tilde{\mathcal{A}}(\vec{a})$ over the space $\Gamma$ and the corresponding path $q^*(t)=q^*(t;\vec{a}_0)\in S(\vec{a}_0)$, i.e.
  \begin{equation}\label{VAR}
  \begin{array}{ll}
  \tilde{\mathcal{A}}(\vec{a}_0) &= \min_{\vec{a}\in \Gamma} \tilde{\mathcal{A}}(\vec{a}) =\min_{\vec{a}\in\Gamma}\left\{\inf_{q(t)\in \mathcal{P}(Q_{start},Q_{end})} \mathcal{A}(q(t))\right\}\\
  \\
  &=\min_{\vec{a}\in\Gamma}\left\{\inf_{q(t)\in \mathcal{P}(Q_{start},Q_{end})}\int_{0}^{T}  \frac{1}{2}\sum_{i=1}^{n} m_i\|\dot{q}_i(t)\|^2+U(q(t))dt\right\}. \end{array}
 \end{equation}
   Then the path $q^*$  is the solution we want.

   \begin{theorem}[Existence and Extension Formula]\label{main1}
Assume  $m_1=m_3>0$, $m_2=m_4>0$ and $\mu=\frac{m_2}{m_1}$.  For any $\theta\in(0,2\pi)$ and $\theta \notin \{\frac{\pi}{2},\pi, \frac{3\pi}{2}\}$, there exists at least one local  minimizer of $\tilde{\mathcal{A}}(\vec{a})$ over the space $\Gamma$. For any minimizer $\vec{a}\in \Gamma$  of $\tilde{\mathcal{A}}$ over the space $\Gamma$,
the corresponding minimizing path $q^*(t)$  on $[0, T]$  connecting $q(0)$ and $q(T)$ can be extended to a classical solution $q(t)=(q_1(t), $ $ q_2(t),$ $ q_3(t),$ $q_4(t))$ of the Newton's equation \eqref{Newton} by the reflection $B=\left( \begin{array}{ll} -1 & 0\\
0 & 1\end{array} \right)$, the permutation $\sigma$ and the rotation $R(\theta)$ as follows: $q(t)=q^*(t)$ on $[0, T]$,
 $$q(t)=(q^*_3(2T-t),q^*_2(2T-t),q^*_1(2T-t),q^*_4(2T-t))B \quad \hbox{ on } \quad (T, 2T],$$
and
 \begin{equation}\label{qet}
 q(t)=\sigma^{k}(q(t-2kT))R(-2k\theta) \hbox{ for } t\in (2kT,(2k+2)T]
  \hbox{ and } k\in \mathbf{Z}^+,
 \end{equation}
where $\sigma=[3, 4, 1, 2]$ is a permutation such that  $\sigma(q(t-2T))=( q_3(t-2T), q_4(t-2T), q_1(t-2T), q_2(t-2T)).$

\end{theorem}
\begin{remark}
For given $(\theta,\mu)$, there may exist more than one local minimizer other than homographic solution. But all the corresponding minimizing pathes can be extended by the same extension formula \eqref{qet}.  For example, solutions in figure \ref{fig111} are different from those solutions in figure \ref{fig1} for  $(\theta,\mu)=(\frac{4\pi}{5},0.8)$, $(\frac{4\pi}{5},1)$ and $(\frac{4\pi}{5},1.5)$. The  actions of solutions in figure \ref{fig111} are larger than the corresponding actions of those solutions in figure \ref{fig1}. By our numerical computation,  solutions with smaller action seem more likely stable.  % Solutions in figure \ref{fig1} and figure \ref{fig111} are two different type for same $\theta=\frac{4\pi}{5}$.
\end{remark}

\begin{figure}
\includegraphics[height=5cm,width=.80\textwidth]{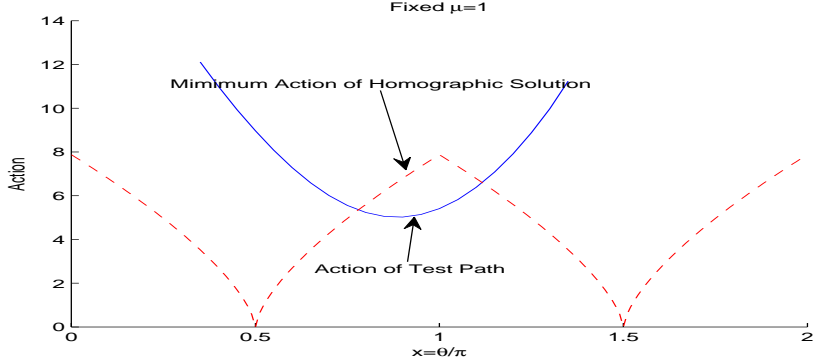}
\caption{Fixed $\mu=1$. The fixed SPBC $\vec{a}_{test}= [ 0.6676542303,$ $ 1.11499232,$ $ 0.5099504088,$ $
        0.6676542314, $ $1.11499232, $ $0.5099504078]$.  Test path has lower action than homographic solution when $0.78\pi<\theta<1.11\pi$.  }\label{ActionMu1}\end{figure}

There is no loss of generality in assuming $m_1=1$, $\mu=\frac{m_2}{m_1}$ and $T=1$ in numerical computation. But we still use $m_1$ and $T$ for the purpose of clarity. Define $\mathcal{A}_{\diamondsuit}, \mathcal{A}_{tpath}: (0,2\pi)\times \mathbf{R}^+\mapsto\mathbf{R}^+$ by
\begin{equation}\label{Ad}
\mathcal{A}_{\diamondsuit}(\theta,\mu)=3\omega^2(m_1r_1^2+m_2r_2^2)T,
\end{equation}
where
\begin{equation}
\omega=\left\{ \begin{array}{ll}|\frac{\pi}{2}-\theta|/T, &\hspace{1cm} \hbox{ if } 0<\theta<\pi, \\
|\frac{3\pi}{2}-\theta|/T, &\hspace{1cm} \hbox{ if } \pi\leq\theta<2\pi.
\end{array}\right.
\end{equation}
  $r_1,$ $r_2$ are uniquely determined by $(\theta,\mu)$ in equations \eqref{CC1} and \eqref{CC2} in section \ref{sec31}. $\mathcal{A}_{\diamondsuit}(\theta,\mu)$ is the minimum value of the action functional over the homographic solution satisfying the SPBC for $(\theta,\mu)$.   Now given a SPBC $\vec{a}_{test}$, the test path $\bar{q}(t)$ with constant velocity connecting the structural prescribed boundaries $Qstart$ and $Qend$ is given by
\begin{equation}\label{Tpath}
\bar{q}(t)=Qstart+\frac{t(Qend-Qstart)}{T}, t\in[0,T].
\end{equation}
Then the action of the test path is computed as
\begin{equation}\label{At}
\mathcal{A}_{tpath}(\theta,\mu)= \left\{ \begin{array}{ll}
& \sum_{k=1}^4\frac{1}{2T}m_k\|Qend_k-Qstart_k\|^2 \\
\\
  &+\int_0^T  \sum_{1\leq k<j\leq 4} \frac{m_k m_j}{\|(Qstart_k-Qstart_j)(1-\frac{t}{T})+(Qend_k-Qend_j)\frac{t}{T}\|}dt,
 \end{array} \right.
\end{equation}
which is an explicit function of $\theta$ and $\mu$. For example, fixed $\mu=1$, figure \ref{ActionMu1} shows the graph of $\mathcal{A}_{tpath}(\theta,1)$ and $\mathcal{A}_{\diamondsuit}(\theta,1)$. There exists an interval of $\theta$ such that test path has lower action than homographic solution has.\\
The  set $\Omega_{\vec{a}_{test}}$ for $\vec{a}_{test}$ is defined as
\begin{equation}%\label{thmu}
\Omega_{\vec{a}_{test}}= \{ (\theta,\mu)\in (0,2\pi)\times\mathbf{R}^+| \mathcal{A}_{\diamondsuit}(\theta,\mu)>\mathcal{A}_{tpath}(\theta,\mu) \hbox{ and } \theta\not=\pi \}.
 \end{equation}
 The size of the set $\Omega_{\vec{a}_{test}}$ strongly depends on the choice of $\vec{a}_{test}$. Figure \ref{Omega1} shows an example of the nonempty region $\Omega_{\vec{a}_{test}}$ on which the test path has lower action than homographic solution.\\
The admissible set $\Omega$ is defined as the union of all the set $\Omega_{\vec{a}_{test}}$.
\begin{equation}\label{thmu}
\Omega=\bigcup_{\vec{a}_{test}} \Omega_{\vec{a}_{test}}.\end{equation}
 \begin{figure}
\includegraphics[height=5cm,width=.68\textwidth]{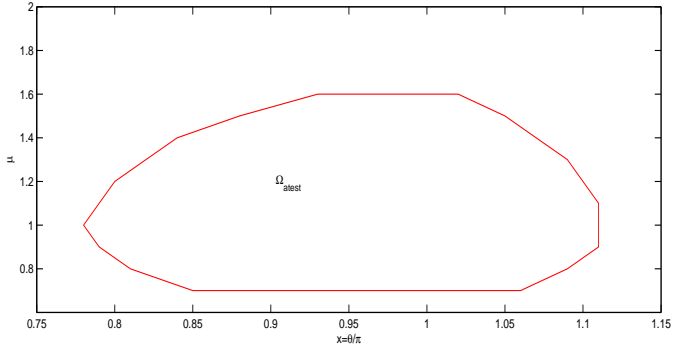}
\caption{ Test path has lower action than homographic solution when $(\theta,\mu)$ is in the region $\Omega_{\vec{a}_{test}}$. The fixed SPBC $\vec{a}_{test}= [ 0.6676542303,$ $ 1.11499232,$ $ 0.5099504088,$ $
        0.6676542314, $ $1.11499232, $ $0.5099504078]$. The $x$-axis is $\frac{\theta}{\pi}$ and the $y$-axis is $\mu$. }\label{Omega1}\end{figure}

\begin{theorem}[Classifications of Non-homographic Solutions]\label{main2}

%There exist $\theta_0$ and $\theta_1$ such that for  any $\theta\in (\theta_0,\theta_1)$ there exist $\mu_{0\theta}$ and $\mu_{1\theta}$. For any fixed $\theta \in (\theta_0,\theta_1)$ and any fixed $\mu\in (\mu_{0\theta},\mu_{1\theta})$, There exists a SPBC $\vec{a}_{test}$ such that $\Omega$ is nonempty.
 For any given $(\theta, \mu)\in \Omega$ and $\theta\not=\pi$,
 there exists at least one minimizer $\vec{a}_0\in \Gamma$  of $\tilde{\mathcal{A}}$ over the space $\Gamma$, such that,
the corresponding minimizing path $q^*(t)$  on $[0, T]$  connecting $q(0)$ and $q(T)$ can be extended to a non-homographic solution $q(t;\theta,\mu,\vec{a}_0)$ (for short $q(t)$) of the Newton's equation \eqref{Newton} by the extension formula \eqref{qet}. %The action of the solution is smaller than the action of its homographic solution.
Each curve $q_i(t), t\in[4kT,(4k+4)T]$ is called a side of the orbit since the orbit of the solution is assembled out the sides $q_i(t), t\in[0,4T]$ by rotation only. The non-homographic solution $q(t;\theta,\mu,\vec{a}_0)$ can be classified as follows (see figures \ref{qua1} to \ref{case4}).\\
 (1) {\bf [Quasi-Periodic Solutions]}  $q(t;\theta,\mu,\vec{a}_0)$ is a quasi-periodic solution if $\theta$ is not commensurable with $\pi$. \\
 (2) {\bf [Periodic Solutions]} $q(t;\theta,\mu,\vec{a}_0)$is a periodic solution if $\theta=\frac{P}{Q}\pi$, where the positive integers $P$ and $Q$ are relatively prime.
\begin{itemize}
\item  When $Q$ is even, the periodic solution $q(t;\theta,\mu,\vec{a}_0)$ is a non-choreographic solution. Each closed curve has $\frac{Q}{2}$ sides. The minimum period is $\mathcal{T}=2QT$.
%\item When $Q$ is odd and $P$ is odd, the periodic solution $q(t;\theta,\mu,\vec{a}_0)$ is a double-choreographic solution. Each closed curve has $Q$ sides. The minimum period is $\mathcal{T}=4QT$. Body $q_1$ chase body $q_3$ on a closed curve and body $q_2$ chase body $q_4$ on another closed curve. $q_1(t+2QT)=q_3(t)$ and $q_3(t+2QT)=q_1(t).$ $q_4(t+2QT)=q_2(t)$ and $q_2(t+2QT)=q_4(t).$
\item When $Q$ is odd, there are four cases.\\
 {\bf Case 1:} If $\mu\not=1$, the periodic solution $q(t;\theta,\mu,\vec{a}_0)$ is a double-choreographic solution. Each closed curve has $Q$ sides. The minimum period is $\mathcal{T}=4QT$. Body $q_1$ chases body $q_3$ on a closed curve and body $q_2$ chases body $q_4$ on another closed curve. $q_1(t+2QT)=q_3(t)$ and $q_3(t+2QT)=q_1(t).$ $q_4(t+2QT)=q_2(t)$ and $q_2(t+2QT)=q_4(t).$\\
 {\bf Case 2:}  If $\mu=1$ and  $P$ is odd , the periodic solution $q(t;\theta,\mu,\vec{a}_0)$ is a double-choreographic solution with minimum period $\mathcal{T}=4QT$. Body $q_1$ chases body $q_3$ on a closed curve and body $q_2$ chases body $q_4$ on another closed curve.\\
   {\bf Case 3:} If $\mu=1$, $P$ is even and the initial configuration $q(0)$ is geometrically same to the ending configuration $q(T)$, i.e. $(a_{10}, a_{20}, a_{30})= (a_{40}, a_{50}, a_{60})$,  then the periodic solution is a choreographic solution. The closed curve has $Q$ sides. The minimum period is $\mathcal{T}=4QT$.\\ (A) If $\frac{Q-1}{2}$ is odd, then the four bodies chase each other on the closed curve in the order of $q_1, q_2, q_3, q_4,$ and then $q_1$, i.e. $q_1(t+QT)=q_2(t),$ $q_2(t+QT)=q_3(t),$ $q_3(t+QT)=q_4(t),$ and $q_4(t+QT)=q_1(t).$\\  (B) If $\frac{Q-1}{2}$ is even, then the four bodies chase each other on the closed curve in the order of $q_1, q_4, q_3, q_2,$ and then $q_1$, i.e. $q_1(t+QT)=q_4(t),$ $q_4(t+QT)=q_3(t),$ $q_3(t+QT)=q_2(t),$ and $q_2(t+QT)=q_1(t).$\\
  {\bf Case 4: } If $\mu=1$, $P$ is even and the initial configuration $q(0)$ is not geometrically same to the ending configuration $q(T)$, i.e. $(a_{10}, a_{20}, a_{30})\not= (a_{40}, a_{50}, a_{60})$,  then the periodic solution is a double choreographic solution. Each closed curve has $Q$ sides. The minimum period is $\mathcal{T}=4QT$. Body $q_1$ chases body $q_3$ on a closed curve and body $q_2$ chases body $q_4$ on another closed curve. $q_1(t+2QT)=q_3(t)$ and $q_3(t+2QT)=q_1(t).$ $q_4(t+2QT)=q_2(t)$ and $q_2(t+2QT)=q_4(t).$\\

\end{itemize}

\end{theorem}
  By using canonical transformation, we reduce the dimension of the Hamiltonian system to eliminate the trivial $+1$ multipliers for the periodic solutions. Then we prove that the periodic solutions are linearly stable in the reduced system by computing the remaining  multipliers of monodromy matrix. The proof is computer-assisted and it is computed one by one.
\begin{theorem}\label{main3}(Linear Stability). Consider the solutions in theorem \ref{main2}.
\begin{itemize}
\item    If $\theta=\frac{2P-1}{2P}\pi$ and $\mu=0.5, 1, 1.5$, the  non-choreographic solutions $q(t)$ are linearly stable for    $P=3,4,5,\cdots, 15$.
\item If $\theta=\frac{2P}{2P+1}$ and $\mu=0.5, 1.5$, the double choreographic solutions $q(t)$ are linearly stable for $P=2,3,\cdots,15$.
 \item If $\theta=\frac{2P-1}{2P+1}$ and $\mu=1$, the  double choreographic solutions $q(t)$ are linearly stable for $P=4,5, 6,\cdots, 15$.
 \item If $\theta=\frac{2P}{2P+1}$ and $\mu=1$, the  choreographic solutions $q(t)$ are linearly stable for $P=2,3, 4,\cdots, 15$.
\end{itemize}
\end{theorem}
\begin{remark}\label{remark1}
(1) $\Omega_{\vec{a}_{test}}$  strongly depends on the choice of SPBC $\vec{a}_{test}$. The union $\Omega$ of such regions   provides the range of $(\theta,\mu)$ where the minimizers have lower action than the action of homographic solutions. Then new periodic or quasi-periodic solutions can be generated from these minimizers. Most solutions for $(\theta, 1)$ in theorem \ref{main2} have been studied in \cite{OuXie} but solutions for $(\theta,1)$ in case 4 do not belongs to the family of solutions in \cite{OuXie}.  \\ %The union contains the region $(0.78\pi,1.1\pi)\times (0.01, 10)$ from our numerical computations but larger range can be expected.\\
(2) Although theorem \ref{main2} only proves the existence of new periodic solutions for $(\theta,\mu)\in \Omega$, there exist new periodic solutions for $(\theta,\mu)\notin \Omega$. There also exist periodic solutions which have larger action than their homographic solutions have. Periodic solution with larger actions are likely unstable from our numerical simulation.    \\
(3) We give a rigorous analytical proof for theorem \ref{main1} and theorem \ref{main2}. The proof of theorem \ref{main3} is computer-assisted. Our theorem \ref{main3} and numerical simulation support the following conjecture. But the proof of the conjecture would be a quite difficult matter and it would involve some new techniques. %To prove the conjecture, some new techniques may be involved to deal with a family of orbits such as index theory (see \cite{HLS}, \cite{YL}, \cite{DO}).\\
%Some cases of the linear stability in theorem \ref{main3} have also been numerically checked \cite{VaW} by Dr. Robert Vanderbei (Princeton University). \\

 {\bf \Large Conjecture:} For every $(\theta,\mu)\in \Omega$ in theorem \ref{main2}, if $\theta$ is commeasurable with $\pi$, there is a linear stable periodic solution.
 %\end{conjecture}
\end{remark}
The rest of the paper is organized as follows. In section 2, we prove the existence and noncollision of minimizing pathes. The existence of the minimizers of the functional $\tilde{\mathcal{A}}$ over the space $\Gamma$ is due to the structure of boundary conditions. Due to the collision free theorem of boundary value problem, it is not hard to prove that the corresponding path of a minimizer is collision free for all time. To prove that the initial minimizing path $q^*(t)$ in $[0,T]$ can be extended to a full solution, we have to check whether the orbits fit well at time $t=kT$. The major difficulty to construct periodic solutions in this variational method with SPBC is to find appropriate SPBC and extension formula. In section 3, we prove that the minimizer generate new periodic solutions which are not homographic orbits.   A special class of homographic orbits satisfying SPBC have their configurations remaining rhomboid for all time. We study orbits of this type in section 3 and we compare them with the orbits we found. This finishes the proof of the existence of new periodic solutions other than homographic solutions. The properties of the new periodic solution are easy to prove by the extension formula. Linear stability is studied in section 4. In the last section, we list some other interesting planar 4-body SPBC and their solutions without detail proof.

\section{Existence, collision free, and extension of minimizing path for boundary value problem}

The minimizer is founded by a two-step minimizing process \eqref{VAR} with appropriate SPBC. In the first step, minimizers are obtained in the full space \eqref{Fixpath} with fixed boundary condition.  For any fixed $\vec{a}\in \Gamma$, the minimizers of $\mathcal{A}$ that connect $Qstart$ and $Qend$ are classical collision-free solutions in the interval $(0, T)$. The existence of minimizers in the Sobolev space is classic and standard. But the assertion of collision free for the boundary value problem is proved by Chenciner \cite{CA2} and Marchal \cite{Ma1} in 2002. They proved that minimizers of $\mathcal{A}$ on the space $\mathcal{P}(Qstart,Qend)$ are collision-free on the interval $(0,T)$ for any given  $Qstart$ and $Qend$ including collision boundary. It is easy to know that $\tilde{\mathcal{A}}$ is lower semicontinuous on $\Gamma$.
 Then the existence of minimizers in the finite dimension space $\Gamma$ is due to the following theorem.
 \begin{theorem}\label{Thm:Ex}%existence
 For $\theta\in (0,2\pi)\backslash\{\frac{\pi}{2},\pi, \frac{3\pi}{2} \}$, $\tilde{\mathcal{A}}(\vec{a})\rightarrow +\infty$ if $|\vec{a}|\rightarrow +\infty$.
  \end{theorem}

\begin{proof}
For any $\vec{a}\in \Gamma$,
  $$\tilde{\mathcal{A}}(\vec{a}) \geq \sum_{i=1}^{n} \int_{0}^{T}  \frac{1}{2} m_i\|\dot{q}_i(t,\vec{a})\|^2 dt\geq \sum_{i=1}^{n}\frac{1}{2}m_i \left\|\int_{0}^{T} \dot{q}_i(t,\vec{a})dt \right\|^2=\sum_{i=1}^{n}\frac{1}{2}m_i \left\| q_i(T)-q_i(0)\right\|^2.$$
  If $|\vec{a}|\rightarrow +\infty$, then at least one $a_i\rightarrow \infty$ for $i=1,2,\cdots,6$. By the structural prescribed boundary conditions, $\|q_i(T)-q_i(0)\|$ can not remain finite for all $i$ if $|\vec{a}|\rightarrow +\infty$.\\
  In fact, if $a_1\rightarrow \infty$ or $a_2\rightarrow \infty$, $\tilde{\mathcal{A}}(\vec{a}) \geq \frac{1}{2}m_2 \left\| q_2(T)-q_2(0)\right\|^2 +\frac{1}{2}m_4 \left\| q_4(T)-q_4(0)\right\|^2 > \frac{1}{2}m_2 | a_1 \cos (\theta)-a_2\sin(\theta)| +\frac{1}{2}m_4 |a_1 \cos (\theta)+a_2\sin(\theta)| \rightarrow \infty$ since $\theta\not=\frac{\pi}{2}$, $\theta\not=\pi$ and $\theta\not=\frac{3\pi}{2}$. \\
  If $a_3\rightarrow \infty$, $\tilde{\mathcal{A}}(\vec{a}) \geq \frac{1}{2}m_1 \left\| q_1(T)-q_1(0)\right\|^2 +\frac{1}{2}m_3 \left\| q_3(T)-q_3(0)\right\|^2 > \frac{1}{2}m_1 | a_3 \sin (\theta)+a_4| +\frac{1}{2}m_3 |\frac{-m_2a_2-m_4a_2+m_1a_3}{m_3} \sin (\theta)-a_4| \rightarrow \infty$ for any choice of $a_4$. Other cases can be easily obtained by similar arguments.

 \end{proof}

\begin{theorem}[Collision-free]\label{Thm:NC}%No collision.
 For $\theta\in (0,2\pi)\backslash\{\frac{\pi}{2},\pi,  \frac{3\pi}{2}\}$, let $\vec{a}_0$ be a minimizer of $\tilde{\mathcal{A}}(\vec{a})$ over the space $\Gamma$ and the corresponding path $q^*(t)\in S(\vec{a}_0)$. Then $q^*$ satisfying SPBC  is a classical collision-free solution of Newton's equation \eqref{Newton} in the whole interval $[0, T]$.% and it satisfies the structural prescribed boundary conditions (SPBC).
\end{theorem}

\begin{proof}
If $\vec{a}_0$ is a minimizer of $\tilde{\mathcal{A}}(\vec{a})$ over the space $\Gamma$, it is well known that the corresponding path $q^*(t)$ is collision-free in the open interval $(0,T)$. To prove $q^*$ is a classical solution of Newton's equation in the whole interval $[0,T]$, we only need to prove that $Qstart(a_{10},a_{20},a_{30})$ and $Qend(a_{40},a_{50},a_{60})$ have no collision. In fact, there are six cases corresponding to initial collision boundary. (1) $a_{10}\not=0$ and $a_{20}=\mu^{-1}a_{30}$ binary collision ($m_1$ and $m_3$ collide). (2) $a_{10}=0$, $a_{20}\not=-a_{30}$,$a_{20}\not= \mu^{-1}a_{30}$ and $a_{20}\not=\frac{1}{1+2\mu}a_{30}$, binary collision ($m_2$ and $m_4$ collide). (3) $a_{10}=0,$ and $a_{20}=\mu^{-1}a_{30}\not=0$ simultaneous binary collision ($m_1$ and $m_3$ collide and $m_2$ and $m_4$ collide). (4) $a_{10}=a_{20}=a_{30}=0$ total collision. (5) $a_{10}=0,a_{20}=-a_{30}\not=0$ triple collision ($m_1$, $m_2$, and $m_4$ collide). (6) $a_{10}=0,a_{20}=\frac{1}{1+2\mu}a_{30}\not=0$ triple collision ($m_2$, $m_3$, and $m_4$ collide). Similarly, there are six cases corresponding to ending collision boundary.

Since $q$ has no collision in the open interval $(0,T)$, we will then analyze the motion during the closed time interval $[0,\epsilon]$ or $[\epsilon,T]$ and prove the existence of sufficiently small values of $\epsilon$ such that a local deformation has lower action and satisfy the SPBC. The contradiction proves that $q$ can not have this collision. Local deformation argument has appeared in a number of papers such as Chenciner \cite{CA2}, Chen \cite{Chen3}, Ferrario-Terracini \cite{FT}, Marchal \cite{Ma1},  and Terracini-Venturelli \cite{TV} etc. Here we only study the collisions at $t=0$ and similar arguments can be applied for collisions at $t=T$. By the nature of SPBC and the construction of the local deformation, we will prove it in two cases: collision with two bodies and collision with three or more bodies. The proof is almost the same as the proof  in the paper by Ouyang-Xie \cite{OuXie} except the perturbation on the deformation due to the differences of SPBC. We include here for the sake of completeness. \\
{\bf CASE ONE:  Collision with two bodies. }\\
  Suppose that $q$ is a local minimizer of $\mathcal{A}$ satisfying the SPBC for $\vec{a}_0$. Let the collision subset $\mathcal{C}=\{\tau_1,\tau_2 \}\subseteq \{1,2,3,4\}$. At time $t=0$, the bodies $m_{\tau_1}$ and $ m_{\tau_2}$ start at the collision point $q_{\tau_1}(0)=q_{\tau_2}(0)$ while the other bodies are away. By the structual of SPBC, the collsion set $\mathcal{C}$ must be either $\{1, 3\}$ or $\{2,4\}$ which is corresponding to the binary collisions (1), (2) and (3) at $t=0$.

 We will build the two following pathes $S_2$ (Kepler ejection orbits at the starting point) and $S_3$ (the deformation of $S_2$) with: (A) Exactly the same motion of all bodies in the interval $[\epsilon, T)$. (B) At the time interval $[0,\epsilon]$, the ejection orbits are replaced by a collision free orbits with boundary conditions satisfying SPBC. The corresponding actions will be $A_1=\mathcal{A}(q)$, $A_2=\mathcal{A}(S_2)$, $A_3=\mathcal{A}(S_3)$. We want to prove that $A_1>A_3$ for sufficiently small time $\epsilon$. Since (A), the actions are different only in the time interval $[0,\epsilon]$.\\
First, consider the ejection orbits in the starting time interval $[0, \epsilon]$ in $S_2$. Let $r$ be the simple radial two-body motion leading from $0$ to $r_\epsilon$ in the time interval $[0,\epsilon]$. By Sundman and Sperling's estimates near collisions \cite{SH,SK}, there exists a positive constant $\gamma$ such that  $r(t)=(\gamma t^{\frac{2}{3}})\vec{\alpha}$ where $\vec{\alpha}$ is a unit vector.  Let $\xi(t)=\frac{m_{\tau_1}q_{\tau_1}(t)+m_{\tau_2}q_{\tau_2}(t)}{m_{\tau_1}+m_{\tau_2}}$ be the center of mass of the $\tau_1$-th and $\tau_2$-th bodies.
$$  q_{\tau_1S_2}(t)=\xi(t)+\frac{m_{\tau_2}}{m_{\tau_1}+m_{\tau_2}}r(t),q_{\tau_2S_2}(t)= \xi(t)-\frac{m_{\tau_1}}{m_{\tau_1}+m_{\tau_2}}r(t);$$
$$ q_{jS_2}(t)=q_j(t), j\notin\mathcal{C}.$$
%It is well known that $A_{1[0,\epsilon]}\geq A_{2[0,\epsilon]}$  in $[0, \epsilon]$ (see 4.3.2 in \cite{Ma1}). We only need to prove $A_{2[0,\epsilon]}>A_{3[0,\epsilon]}$   in order to prove $A_{1[0,\epsilon]}>A_{3[0,\epsilon]}$ in $[0, \epsilon]$. \\
We consider the deformation of $r(t)$ as \begin{equation}
r_{\delta}(t)=r(t)+\delta \phi(t)\vec{s},
\end{equation}
where $\vec{s}$ is an appropriate unit vector, $\delta =\frac{\epsilon}{N}$ with $N\geq 2\max\{ K_{in}/U_{in}, 4\}$, and
$$\phi(t)=\left\{ \begin{array}{ll} 1, & 0\leq t\leq \delta,\\
\frac{\delta+\tilde{N}\delta-t}{\tilde{N}\delta}, & \delta<t\leq \delta+\tilde{N}\delta,\\
0, & \delta+\tilde{N}\delta<t\leq \epsilon,\end{array}\right.$$
where $K_{in}/U_{in}<\tilde{N}<N-1$. The positive $K_{in}$ and $U_{in}$ are given in the equations \eqref{Kin} and \eqref{Uin} respectively, which are independent of $\epsilon$.

The collision-free motion $S_3$ is denoted by
$$ q_{\tau_1S_3}(t)=\xi(t)+\frac{m_{\tau_2}}{m_{\tau_1}+m_{\tau_2}}r_\delta(t), q_{\tau_2S_3}(t)=\xi(t)-\frac{m_{\tau_1}}{m_{\tau_1}+m_{\tau_2}}r_\delta(t);$$
$$ q_{jS_3}(t)=q_j(t), j\notin \mathcal{C}.$$
We choose $\vec{s}$  to be the unit vector of $(0,\pm 1)R(\theta)$ when $\{\tau_1,\tau_2\}=\{2,4\}$ and  we choose $\vec{s}$  to be the unit vector of $(\pm 1,0)R(\theta)$  when $\{\tau_1,\tau_2\}=\{1,3\}$. The sign will be determined later. So the initial condition of $S_3$ satisfies the SPBC.

Now consider the expression of the actions for each path in the time interval $[0,\epsilon]$. They will be decomposed into two parts: the first part $A_{in}$ is to compute the action of the relative motion of the colliding bodies  $m_{\tau_1}$ and $m_{\tau_2}$; the second part $A_{out}$ is to compute the action of the remainder. It is easy to know that $A_{1in}\geq A_{2in}$   since the homothetic collision-ejection orbit is a minimizer. We only need to prove $A_{2in}-A_{3in}>A_{3out}-A_{1out}$   in order to prove $A_{1}>A_{3}$ in $[0, \epsilon]$.
 We first note that
$$m_{\tau_1}|\dot{q}_{\tau_1S_2}|^2+m_{\tau_2}|\dot{q}_{\tau_2S_2}|^2=m_{\tau_1}\langle \dot{\xi}+\frac{m_{\tau_2}}{m_{\tau_1}+m_{\tau_2}}\dot{r},\dot{\xi}+\frac{m_{\tau_2}}{m_{\tau_1}+m_{\tau_2}}\dot{r} \rangle + $$
$$m_{\tau_2}\langle \dot{\xi}-\frac{m_{\tau_1}}{m_{\tau_1}+m_{\tau_2}}\dot{r},\dot{\xi}-\frac{m_{\tau_1}}{m_{\tau_1}+m_{\tau_2}}\dot{r} \rangle
=(m_{\tau_1}+m_{\tau_2})|\dot{\xi}|^2+\frac{m_{\tau_1}m_{\tau_2}}{m_{\tau_1}+m_{\tau_2}}|\dot{r}|^2.$$
Then
$$A_{2in}-A_{3in}=\int_{0}^{\epsilon}\frac{m_{\tau_1}m_{\tau_2}}{2(m_{\tau_1}+m_{\tau_2})}(|\dot{r}|^2-|\dot{r}_\delta|^2 ) +m_{\tau_1}m_{\tau_2}\left(\frac{1}{|r|}-\frac{1}{|r_\delta|}\right)dt,$$
$$A_{3out}-A_{1out}=\int_{0}^{\epsilon} \sum_{i\notin\mathcal{C}; \tau_j\in\mathcal{C}}\left(\frac{m_im_{\tau_j}}{|q_i-q_{{\tau_j}S_3}|}-\frac{m_im_{\tau_j}}{|q_i-q_{{\tau_j}}|}\right)dt.
$$
Now we estimate the bounds for $A_{out}$. Consider the motion of the mass $m_j$ between the arbitrary successive instants $t_1$ and $t_2$. Because the minimum of the integral $\int_{t_1}^{t_2} \frac{m_j |\dot{q}_j|^2}{2}dt$ between given positions $q_j(t_1)$ and $q_j(t_2)$ is obtained for a constant velocity vector, we can always write $\frac{m_j|q_j(t_2)-q_j(t_1)|^2}{2(t_2-t_1)}\leq \int_{t_1}^{t_2} \frac{m_j |\dot{q}_j|^2}{2}dt \leq \mathcal{A}(q)\leq K< \infty$. So if $0\leq t_1\leq t_2\leq T$, $|q_j(t_2)-q_j(t_1)|\leq \left(\frac{2K(t_2-t_1)}{m_j}\right)^{1/2}$. Pick up $\epsilon >0$ small such that the two bodies $m_{\tau_1}$ and $m_{\tau_2}$ will remain at less than twice that distance from the collision point $q_{\tau_1}(0)$ all along the time interval $[0, \epsilon]$, i.e. $|q_{\tau_1}-q_{\tau_2}|\leq J\sqrt{\epsilon}$, where $J=2(2K)^{1/2}$. $m_j, j\notin\mathcal{C}$  will remain outside of the circle centered at the collision point with radius $D$ and $J\sqrt{\epsilon}\leq J\sqrt{\epsilon_0}\ll D$ for a fixed $\epsilon_0$. So during the time interval $[0, \epsilon]$, the   bodies $m_j, j\notin\mathcal{C}$ are outside of the circle with radius $D$ and center $q_{\tau_1}(0)$, while the bodies $m_{\tau_1}$ and $m_{\tau_2}$ are inside the much smaller circle of the same center and radius $J\sqrt{\epsilon}$.
$$|A_{3out}-A_{1out}|\leq \int_{0}^{\epsilon} \sum_{i\notin\mathcal{C}; \tau_j\in\mathcal{C}} m_im_{\tau_j}\left|\left(\frac{|q_i-q_{{\tau_j}S_3}|-|q_i-q_{{\tau_j} }| }{|q_i-q_{{\tau_j} }| |q_i-q_{{\tau_j}S_3}|} \right)\right|dt.
$$
$$  \leq \int_{0}^{\epsilon} \sum_{i\notin\mathcal{C}; j\in\mathcal{C}}m_im_{\tau_j}\left(\frac{|q_{{\tau_j}S_3}-q_{{\tau_j} }| }{|q_i-q_{{\tau_j} }| |q_i-q_{{\tau_j}S_3}|} \right)dt.
$$
\begin{equation}\label{Uout} \leq \int_{0}^{\epsilon} \sum_{i\notin\mathcal{C}; \tau_j\in\mathcal{C}} m_im_{\tau_j} \left(\frac{J\sqrt{\epsilon} }{(D-J\sqrt{\epsilon_0})^2} \right)dt= \frac{4J}{(D-J\sqrt{\epsilon_0})^2}\epsilon^{\frac{3}{2}}=U_{out}\epsilon^{\frac{3}{2}}.
\end{equation}

Let us compute $A_{2in}-A_{3in}$. By choosing appropriate direction of $\vec{s}$
 such that  $\langle r,\vec{s}\rangle\geq 0$,
$$\int_{0}^{\epsilon}\frac{m_{\tau_1}m_{\tau_2}}{2(m_{\tau_1}+m_{\tau_2})}(|\dot{r}|^2-|\dot{r}_\delta|^2 )dt=-\int_{0}^{\epsilon} \frac{m_{\tau_1}m_{\tau_2}}{2(m_{\tau_1}+m_{\tau_2})}(2\delta\dot{\phi}\langle r,\vec{s}\rangle +(\delta\dot{\phi})^2)dt$$
$$\geq -\int_{\delta}^{\delta+\tilde{N}\delta}\frac{m_{\tau_1}m_{\tau_2}}{2(m_{\tau_1}+m_{\tau_2})} (\delta\dot{\phi})^2dt =-\int_{\delta}^{\delta+\tilde{N}\delta} \frac{m_{\tau_1}m_{\tau_2}}{2(m_{\tau_1}+m_{\tau_2})}( -\frac{1}{\tilde{N}})^2dt$$
\begin{equation}\label{Kin}
\geq-\frac{m_{\tau_1}m_{\tau_2}}{2(m_{\tau_1}+m_{\tau_2})}\frac{\delta}{\tilde{N}}=-K_{in}\frac{\delta}{\tilde{N}}.
\end{equation}

$$\int_{0}^{\epsilon}\left(\frac{1}{|r|}-\frac{1}{|r_\delta|}\right)dt= \int_{0}^{\epsilon}\left(\frac{1}{|r|}-\frac{1}{(|r|^2+2\delta\phi\langle r,\vec{s}\rangle+(\delta \phi)^2)^{1/2}}\right)dt
$$
$$=\int_{0}^{\epsilon}\left(\frac{2\delta\phi\langle r,\vec{s}\rangle+(\delta \phi)^2}{|r|(|r|^2+2\delta\phi\langle r,\vec{s}\rangle+(\delta \phi)^2)^{1/2} (|r|+ (|r|^2+2\delta\phi\langle r,\vec{s}\rangle+(\delta \phi)^2)^{1/2})}\right)dt
$$
$$\geq  \int_{0}^{\delta}\left(\frac{(\delta )^2}{|r|(|r|+\delta) (2|r|+ \delta )}\right)dt \geq  \int_{0}^{\delta}\left(\frac{1}{\gamma(\gamma+\delta^{1/3}) (2\gamma+ \delta^{1/3} )}\right)dt
$$
\begin{equation}\label{Uin}
\geq \left(\frac{1}{\gamma(\gamma+1) (2\gamma+ 1)}\right)\delta=U_{in}\delta,
 \end{equation}
where we use the fact $|r|\leq \gamma\delta^{2/3}$ in $[0,\delta]$ and $\delta<1$. \\
So $A_{2in}-A_{3in}>\left(-K_{in}\frac{\delta}{\tilde{N}}+ U_{in}\delta\right)=\left(-\frac{K_{in}}{\tilde{N}}+U_{in}\right)\frac{\epsilon}{N}>U_{out}\epsilon^{\frac{3}{2}}\geq A_{3out}-A_{1out}$
for small $\epsilon$, which implies $A_1>A_3$. \\
 The action of $S_3$ is smaller than the action of $S_1$ which contradicts the fact that $S_1$ is a minimizer.
The contradiction completes the proof that the vector $\vec{a}_0$ with  binary collision   is not a minimizer of $\mathcal{A}$ on $\Gamma$. \\

 %%%%%%%%%%%%%%%%%%%%%%%%%%%%%%%%%%%%
{\bf CASE TWO: Collisions with three or more bodies.}\\
We can give a unify proof as we did in paper \cite{OuXie}. But for the sake of clarity and simplicity, we only prove the triple collision case (5) $a_{10}=0,a_{20}=-a_{30}\not=0$, where $m_1$, $m_2$, and $m_4$ collide while $m_3$ is away.
We will build the two similar solutions $S_2$ (Kepler ejection orbits at the starting point) and $S_3$ (deformation of $S_2$) as we built for binary collisions.
First, consider the ejection orbits in the starting time interval $[0, \epsilon]$ in $S_2$. By \cite{Sa,SH}, the configuration of the colliding bodies $m_1, m_2, m_4$ is approaching the set of central configurations. Let $\bar{\omega}=(\bar\omega_1, \bar\omega_2,\bar\omega_4)$ be the translation of the central configuration $(q_1(\epsilon),q_2(\epsilon), q_4(\epsilon))$ by shifting center of mass at origin. By Sundman and Sperling's estimates near collisions \cite{SH,SK}, the homothetic collision-ejection orbit is given by $\omega_i(t)=\bar{\omega}_i t^{2/3}$, $i=1,2,4$ and $t\in [0,\epsilon]$.  Let $\xi(t)=\frac{m_1q_1(t)+m_2q_2(t)+m_4q_4(t)}{m_1+m_2+m_4}$ be the center of mass of the colliding bodies.
$$ q_{1S_2}(t)=\xi(t)+\omega_1(t), q_{2S_2}(t)=\xi(t)+\omega_2(t),$$
$$ q_{3S_2}(t)=q_3(t), q_{4S_2}(t)=\xi(t)+\omega_4(t).$$

We consider the deformation of $\omega_i(t)$ as \begin{equation}
\omega_{1\delta}(t)=\omega_1(t)-\frac{1}{m_1}\delta \phi(t)\vec{s}_1,
\end{equation}
\begin{equation}
\omega_{2\delta}(t)=\omega_2(t)+\frac{1}{2m_2}\delta \phi(t)\vec{s}_1+\frac{1}{2m_2}\delta \phi(t)\vec{s}_2,
\end{equation}
\begin{equation}
\omega_{4\delta}(t)=\omega_4(t)+\frac{1}{2m_4}\delta \phi(t)\vec{s}_1-\frac{1}{2m_4}\delta \phi(t)\vec{s}_2,
\end{equation}
where $\vec{s}_1, \vec{s}_2$ are two appropriate unit vectors of $(0,\pm 1)R(\theta)$ or $(\pm 1, 0)R(\theta)$, $\delta =\frac{\epsilon}{N}$ with $N\geq 2\max\{ K_{in}/U_{in}, 4\}$, and
$$\phi(t)=\left\{ \begin{array}{ll} 1, & 0\leq t\leq \delta,\\
\frac{\delta+\tilde{N}\delta-t}{\tilde{N}\delta}, & \delta<t\leq \delta+\tilde{N}\delta,\\
0, & \delta+\tilde{N}\delta<t\leq \epsilon.\end{array}\right.$$
where $K_{in}/U_{in}<\tilde{N}<N-1$. The positive $K_{in}$ and $U_{in}$ are constants given in the equations \eqref{Kin} and \eqref{Uin} respectively. The deformation $S_3$ is denoted by
$$ q_{1S_3}(t)=\xi(t)+\omega_{1\delta}(t), q_{2S_3}(t)=\xi(t)+\omega_{2\delta}(t),$$
$$ q_{3S_3}(t)=q_3(t), q_{4S_3}(t)=\xi(t)+\omega_{4\delta}(t).$$
At $t=0$, $q_{1S_3}(0)=(0,-a_{30}\pm\delta/m_1)R(\theta), q_{2S_3}(0)=(\pm \delta/m_2,-a_{30}\mp\delta/(2m_2))R(\theta),$ $q_{3S_3}(0)=(0,(2\mu+1)a_{30})R(\theta),$ $q_{4S_3}(0)=(\mp \delta/m_4,-a_{30}\mp\delta/(2m_4))R(\theta)$. So the initial conditions of $S_3$ satisfy the SPBC.\\
Now consider the expression of the actions for each path in the time interval $[0,\epsilon]$. They will be decomposed into two parts: the first part $A_{in}$ is to compute the action of the relative motion of the colliding bodies $m_1$, $m_2$ and $m_4$; the second part $A_{out}$ is to compute the action of the remainder.

It is easy to know that $A_{1in}\geq A_{2in}$   since the homothetic collision-ejection orbit is a minimizer. We only need to prove $A_{2in}-A_{3in}>A_{3out}-A_{1out}$   in order to prove $A_{1}>A_{3}$ in $[0, \epsilon]$.

First of all, we estimate $A_{3out}-A_{1out}$. Consider the motion of the mass $m_j$ between the arbitrary successive instants $t_1$ and $t_2$. Because the minimum of the integral $\int_{t_1}^{t_2} \frac{m_j |\dot{q}_j|^2}{2}dt$ between given positions $q_j(t_1)$ and $q_j(t_2)$ is obtained for a constant velocity vector, we can always write $\frac{m_j|q_j(t_2)-q_j(t_1)|^2}{2(t_2-t_1)}\leq \int_{t_1}^{t_2} \frac{m_j |\dot{q}_j|^2}{2}dt \leq \mathcal{A}(q)\leq K< \infty$. So if $0\leq t_1\leq t_2\leq T$, $|q_j(t_2)-q_j(t_1)|\leq \left(\frac{2K(t_2-t_1)}{m_j}\right)^{1/2}$. Pick up $\epsilon >0$ small such that the three bodies $m_1$, $m_2$ and $m_4$ will remain in a circle with radius $J\sqrt{\epsilon}$ from the collision point $(0,-a_{30})$ all along the time interval $[0, \epsilon]$, i.e. where $J=(2K/\min\{m_i\})^{1/2}$. Then $|q_i-q_j|\leq 2J\sqrt{\epsilon}$ for $i,j=1,2,4$.   $m_3$ will remain outside of the circle centered at the collision point with radius $D$ and $J\sqrt{\epsilon}\leq J\sqrt{\epsilon_0}\ll D$. So during the time interval $[0, \epsilon]$, the body $m_3$ is outside of the circle with radius $D$ and center $(0,-a_{30})$, while the bodies $m_1$, $m_2$ and $m_4$ are inside the much smaller circle with the same center and radius $J\sqrt{\epsilon}$. %The bodies in solution $S_2$ are also within the circles. By picking up an appropriate large $N$, the bodies in solution $S_3$ are also within the circles.

$$|A_{3out}-A_{1out}|=\left|\int_{0}^{\epsilon} \frac{1}{2}m_3|\dot{q}_{3S_3}(t)|^2 -\frac{1}{2}m_3|\dot{q}_{3}(t)|^2 + \sum_{i=1,2,4}\frac{m_im_3}{|q_{iS_3}-q_{3S_3}|} -\frac{m_im_3}{|q_i-q_{3}|}dt\right|
$$
$$=\left|
\int_{0}^{\epsilon} \sum_{i=1,2,4}\frac{m_im_3(|q_i-q_{3}|-|q_{iS_3}-q_{3S_3}|)}{|q_{iS_3}-q_{3S_3}||q_i-q_{3}|}dt\right| \leq \int_{0}^{\epsilon} \sum_{i=1,2,4}\frac{m_im_3(|q_i-q_{iS_3}|)}{|q_{iS_3}-q_{3S_3}||q_i-q_{3}|}dt$$
\begin{equation}\label{Uout}
\leq \int_{0}^{\epsilon} \frac{2Mm_3J\sqrt{\epsilon}}{(D-J\sqrt{\epsilon})^2}dt\leq \frac{2Mm_3J{\epsilon}^{3/2}}{(D-J\sqrt{\epsilon_0})^2}=U_{out}{\epsilon}^{3/2}.
\end{equation}

Now we compute $$A_{2in}-A_{3in}=\int_{0}^{\epsilon}\sum_{i=1,2,4}\frac{1}{2}m_i(|\dot{q}_{iS_2}|^2-|\dot{q}_{iS_3}|^2)+\sum_{i<j=1,2,4}\frac{m_i m_j}{|q_{iS_2}-q_{jS_2}|}-\frac{m_i m_j}{|q_{iS_3}-q_{jS_3}|}dt.$$
Because $m_2=m_4$ and $\omega_i(t)-\omega_j(t)=(\bar{\omega}_i-\bar{\omega}_j)t^{2/3}$ in $[0,\epsilon]$, we are able to pick up the appropriate direction vector $\vec{s}_1$ and $\vec{s}_2$ such that the inner product $$\langle\omega_2(t)-\omega_4(t),\left( \frac{1}{2m_2}+\frac{1}{2m_4}\right)\delta \phi(t)\vec{s}_2 \rangle\geq 0$$ and
$$\langle\omega_2(t)-\omega_1(t),\left( \frac{1}{2m_2}+\frac{1}{m_1}\right)\delta \phi(t)\vec{s}_1 + \left( \frac{1}{2m_2}\right)\delta \phi(t)\vec{s}_2 \rangle\geq 0.$$
Because $\bar\omega_1$, $\bar\omega_2,$ and $\bar\omega_4$ is a central configuration, they are either collinear or equilateral triangle. So  $$\langle\omega_4(t)-\omega_1(t),\left( \frac{1}{2m_4}+\frac{1}{m_1}\right)\delta \phi(t)\vec{s}_1 - \left( \frac{1}{2m_4}\right)\delta \phi(t)\vec{s}_2 \rangle\geq 0.$$
Since both the centers  of mass of central configurations $\omega(t)$ and $\omega_\delta(t)$ are at origin, we have the kinetic energy
$$\sum_{i=1,2,4}\frac{1}{2}m_i|\dot{q}_{iS_2}|^2=\sum_{i=1,2,4}\frac{1}{2}m_i(|\dot{\xi}|^2+|\dot{\omega}_i(t)|^2) \hbox{ and } \sum_{i=1,2,4}\frac{1}{2}m_i|\dot{q}_{iS_3}|^2=\sum_{i=1,2,4}\frac{1}{2}m_i(|\dot{\xi}|^2+|\dot{\omega}_{i\delta}(t)|^2).$$
$$\int_{0}^{\epsilon}\sum_{i=1,2,4}\frac{1}{2}m_i(|\dot{q}_{iS_2}|^2-|\dot{q}_{iS_3}|^2)dt=\int_{0}^{\epsilon} \sum_{i=1,2,4}\frac{1}{2}m_i (|\dot{\omega}_i(t)|^2-|\dot{\omega}_{i\delta}(t)|^2)dt$$
$$\geq -\int_{\delta}^{\delta+\tilde{N}\delta} \frac{1}{2}m_1\left|\frac{1}{m_1}\delta \dot\phi(t)\vec{s}_1\right|^2+\frac{1}{2}m_2\left|\frac{1}{2m_2} \delta \dot{\phi}(t)\vec{s}_1+\frac{1}{2m_2} \delta \dot{\phi}(t)\vec{s}_2\right|^2 $$
 $$+ \frac{1}{2}m_4 \left|\frac{1}{2m_4} \delta \dot{\phi}(t)\vec{s}_1-\frac{1}{2m_4} \delta \dot{\phi}(t)\vec{s}_2\right|^2 dt$$
\begin{equation}\label{Kin}
\geq- \left( \frac{1}{2 m_1} +\frac{1}{4 m_2} +\frac{1}{4 m_4} \right)\left(\frac{\delta}{\tilde{N}}\right)\equiv -K_{in}\frac{\delta}{\tilde{N} }.
\end{equation}
Let $m_{12}=\left(\left( \frac{1}{m_1}+\frac{1}{2m_2}\right)^2 +\frac{1}{4m_2^2}\right)^{\frac{1}{2}}$. For potential energy, we have
$$\int_{0}^{\epsilon}\sum_{i<j=1,2,4}\frac{m_i m_j}{|q_{iS_2}-q_{jS_2}|}-\frac{m_i m_j}{|q_{iS_3}-q_{jS_3}|}dt =  $$
$$=\int_{0}^{\epsilon}\sum_{i<j=1,2,4} m_im_j \frac{|\omega_{i\delta}(t)-\omega_{j\delta}(t)|-|\omega_i(t)-\omega_j(t)|}{|\omega_i(t)
-\omega_j(t)||\omega_{i\delta}(t)-\omega_{j\delta}(t)|}$$
$$\geq \int_{0}^{\epsilon}\sum_{i<j=1,2,4} \frac{m_im_j |\omega_{i\delta}-\omega_i+ \omega_j-\omega_{j\delta} |^2}{|\omega_i(t)
-\omega_j(t)||\omega_{i\delta}(t)-\omega_{j\delta}(t)| \left(|\omega_{i\delta}(t)-\omega_{j\delta}(t)|+|\omega_i(t)-\omega_j(t)|\right)}dt$$
$$\geq \int_{0}^{\epsilon}\frac{m_1m_2 |\omega_{1\delta}-\omega_1+ \omega_2-\omega_{2\delta} |^2}{|\omega_1(t)
-\omega_2(t)||\omega_{1\delta}(t)-\omega_{2\delta}(t)| \left(|\omega_{1\delta}(t)-\omega_{2\delta}(t)|+|\omega_1(t)-\omega_2(t)|\right)}dt$$
$$\geq
\int_{0}^{\delta}\frac{m_1m_2m_{12}^2\delta^2}{|\omega_1(t)
-\omega_2(t)||\omega_{1\delta}(t)-\omega_{2\delta}(t)| \left(|\omega_{1\delta}(t)-\omega_{2\delta}(t)|+|\omega_1(t)-\omega_2(t)|\right)}dt$$
$$\geq \int_{0}^{\delta} \frac{m_1m_2m_{12}^2 }{|\bar{\omega}_1-\bar{\omega}_2|\left(|\bar{\omega}_1-\bar{\omega}_2|+m_{12}\delta^{1/3}\right) \left(2|\bar{\omega}_1-\bar{\omega}_2|+m_{12}\delta^{1/3}\right)} dt $$
\begin{equation}\label{Uin}
\geq \frac{m_1m_2m_{12}^2 }{|\bar{\omega}_1-\bar{\omega}_2|\left(|\bar{\omega}_1-\bar{\omega}_2|+m_{12}\right) \left(2|\bar{\omega}_1-\bar{\omega}_2|+m_{12}\right)}\delta \equiv U_{in}\delta.
\end{equation}
From the above estimations \eqref{Uout},\eqref{Kin} and \eqref{Uin}, by picking small enough $\epsilon$, we have
$$A_{2in}-A_{3in}\geq \left(U_{in}-\frac{K_{in}}{\tilde{N}}\right)\delta =\left(U_{in}-\frac{K_{in}}{\tilde{N}}\right)\frac{\epsilon}{N} \geq U_{out}{\epsilon}^{3/2}>A_{3out}-A_{1out}.$$
This completes the proof that the minimizer can not have the triple collision.

\end{proof}

Let $\mathbf{A}$ and $\mathbf{B}$ be two proper linear subspaces of $(\mathbf{R}^2)^4$ which are given as $$\mathbf{A}=\left\{ \left. \left( \begin{array}{cc} 0 & -a_3\\
-a_1 & a_2\\ 0 & \frac{-m_2a_2-m_4a_2+m_1a_3}{m_3}\\
a_1 & a_2 \end{array} \right) R(\theta) \in (\mathbf{R}^2)^4\right| (a_1,a_2,a_3)\in \mathbf{R}^3 \right\}$$ and $$\mathbf{B}=\left\{ \left.  \left( \begin{array}{cc} a_4 & a_5\\
0 & -a_6\\ -a_4 & a_5\\
0 & \frac{-m_1a_5-m_3a_5+m_2 a_6}{m_4} \end{array} \right) \in (\mathbf{R}^2)^4\right| (a_4,a_5,a_6)\in \mathbf{R}^3 \right\}.$$
Let us consider the action functional $\mathcal{A}$ defined in \eqref{Action} over the function space $$\mathcal{P}(\mathbf{A},\mathbf{B}):=\{q\in H^1([0,T],(\mathbf{R}^2)^4) | q(0)\in \mathbf{A}, q(T)\in \mathbf{B}\}.$$
It is easy to prove the theorem of equivalence below.
\begin{theorem}[Equivalence]\label{Thm:EQ}
$\vec{a}_0\in \Gamma$ with corresponding path $q^*\in S(\vec{a}_0)$ satisfying $q^*(0)=Qstart$ and $q^*(T)=Qend$ is a minimizer of $\tilde{\mathcal{A}}(\vec{a})$ over the space $\Gamma$, if and only if, $q^*$ is a minimizer of $\mathcal{A}$ over the function space $\mathcal{P}(\mathbf{A},\mathbf{B})$ with $q^*(0)=Qstart\in \mathbf{A}$ and $q^*(T)=Qend \in \mathbf{B}$.
\end{theorem}

Now it is ready to prove theorem \ref{main1}.
\begin{proof}[Proof of Theorem \ref{main1}.]

By theorem \ref{Thm:NC}, any path $q^*(t)$ corresponding to a local minimizer $\vec{a}_0$ is a classic solution in the interval $[0, T]$. We only need to prove that it can be extended to a classical solution by the extension formula \eqref{qet}.  \\
Because $q^*(t)$ is a classic solution of Newton's equation \eqref{Newton} on $[0,T]$, it is easy to check that $q(t)$ is a classical solution in each interval $((n-1)T, nT)$ for any given positive integer $n$. To prove $q(t)$ is a classical solution for all real $t$, we need to  prove that $q(t)$ is connected very well at $t=nT$ for any integer $n$, i.e. $\lim_{t\rightarrow (nT)^-} q(t)=\lim_{t\rightarrow (nT)^+} q(t)$ and $\lim_{t\rightarrow (nT)^-} \dot{q}(t)=\lim_{t\rightarrow (nT)^+} \dot{q}(t)$. By the structure of  the extension equation \eqref{qet}, we only need prove it for $n=1$ and $n=2$. \\
By the SPBC, at $n=1$, we have $\lim_{t\rightarrow (T)^-} q(t)=\lim_{t\rightarrow (T)^+} q(t)$ because
  $$\lim_{t\rightarrow (T)^-} q(t)=\left(q_1^*(T),q_2^*(T),q_3^*(T),q_4^*(T)\right) =Q_{end}, $$
  and
  $$\lim_{t\rightarrow (T)^+} q(t)=\left(q_3^*(T),q_2^*(T),q_1^*(T),q_4^*(T)\right)B=Q_{end}.$$
    At $n=2$, we have $\lim_{t\rightarrow (2T)^-} q(t)=\lim_{t\rightarrow (2T)^+} q(t)$ because
  $$\lim_{t\rightarrow (2T)^-} q(t)  =\left(q_3^*(0),q_2^*(0),q_1^*(0),q_4^*(0)\right)B=\left( \begin{array}{cc}
0 & \frac{-m_2a_2-m_4a_2+m_1a_3}{m_3}\\ -a_1 & a_2\\ 0 & -a_3\\
a_1 & a_2 \end{array} \right) R(\theta)B,$$
  and $$\lim_{t\rightarrow (2T)^+} q(t)=\left(q_3^*(0),q_4^*(0),q_1^*(0),q_2^*(0)\right)R(-2\theta)= \left( \begin{array}{cc}0 & \frac{-m_2a_2-m_4a_2+m_1a_3}{m_3}\\
a_1 & a_2\\ 0 & -a_3\\
-a_1 & a_2 \end{array} \right) R(-\theta).$$
That $\lim_{t\rightarrow (nT)^-} \dot{q}(t)=\lim_{t\rightarrow (nT)^+} \dot{q}(t)$ at $n=1$ and $n=2$ is equivalent to the relations given by \eqref{TT} and \eqref{T0} below.
At
$t=T$, \begin{equation}\label{TT}
\dot{q}_{11}(T)=\dot{q}_{31}(T), \dot{q}_{12}(T)=-\dot{q}_{32}(T), \dot{q}_{22}(T)=\dot{q}_{42}(T)=0,
\end{equation}
and at $t=2T$,
\begin{equation}\label{T0}
 \begin{array}{ll}
\dot{q}_{1}(0)=(\dot{q}_{11}(0), -\dot{q}_{12}(0))R(2\theta), &
\dot{q}_{2}(0)=(\dot{q}_{41}(0), -\dot{q}_{42}(0))R(2\theta),\\
\dot{q}_{3}(0)=(\dot{q}_{31}(0), -\dot{q}_{32}(0))R(2\theta),&
\dot{q}_{4}(0)=(\dot{q}_{21}(0), -\dot{q}_{22}(0))R(2\theta).
\end{array}
\end{equation}

Now we prove the equalities \eqref{TT} and \eqref{T0}. Since $\vec{a}_0\in \Gamma$ is a minimizer of $ \tilde{\mathcal{A}}(\vec{a})$ over $\Gamma$, $q^*$ is a minimizer of $\mathcal{A}$ over the function space $\mathcal{P}(\mathbf{A},\mathbf{B})$ by theorem \ref{Thm:EQ}. Here we use $q$ for $q^*$ by our extension formula \eqref{qet}. Consider an admissible variation  $\xi\in \mathcal{P}(\mathbf{A},\mathbf{B})$ with $\xi(0)\in \mathbf{A}$ and $\xi(T)\in \mathbf{B}$, then the first variation $\delta_\xi\mathcal{A}(q)$ is computed as:
  $$\delta_\xi\mathcal{A}(q)=\lim_{\delta\rightarrow 0}\frac{\mathcal{A}(q+\delta \xi)-\mathcal{A}(q)}{\delta}$$
    $$=\int_{0}^{T}\frac{1}{2}\sum_{i=1}^{4}\lim_{\delta\rightarrow 0} m_i \frac{\|\dot{q}_i +\delta \dot{\xi}_i\|^2-\|\dot{q}_i \|^2}{\delta}+ \lim_{\delta\rightarrow 0}\frac{U(q+\delta \xi)-U(q)}{\delta}dt$$
$$=\int_{0}^{T} \left (\sum_{i=1}^{4} m_i <\dot{q}_i, \dot{\xi}_i>+ \sum_{i=1}^4<\frac{\partial }{\partial q_i}(U(q(t))),\xi_i>\right)dt$$
$$=\sum_{i=1}^{4}\left.m_i <\dot{q}_i, \xi_i>\right|_{t=0}^{t=T} + \int_{0}^{T} <-m_i \ddot{q}_i+\frac{\partial }{\partial q_i}(U(q(t))),\xi_i>dt. $$
Because the first variation $\delta_\xi\mathcal{A}(q)$ is zero for any $\xi$, and $q$ satisfies Newton's equation \eqref{Newton}, we have
\begin{equation}\label{vaeq}
\delta_\xi\mathcal{A}(q)=\sum_{i=1}^{4}\left(m_i <\dot{q}_i(T),\xi_i(T)> \right)-\sum_{i=1}^{4}\left(m_i<\dot{q}_i(0),\xi_i(0)>\right)=0.
\end{equation}
For $i=4,5,6,$ let $\xi^{(i)}(t)\in \mathcal{P}(\mathbf{A},\mathbf{B})$ satisfy $\xi^{(i)}(0)=0$ and
$$\xi^{(i)}(T)=\left( \begin{array}{cc} a_4 & a_5\\
0 & -a_6\\ -a_4 & a_5\\
0 & \frac{-m_1a_5-m_3a_5+m_2 a_6}{m_4} \end{array} \right), $$  where $a_i=1,$ $ a_j=0$  if $j\not=i.$ Then from equation \eqref{vaeq},
$$ \delta_{\xi^{(4)}}\mathcal{A}(q)=(m_1\dot{q}_{11}(T)-m_3\dot{q}_{31}(T))=0,$$
$$ \delta_{\xi^{(5)}}\mathcal{A}(q)=(m_1 \dot{q}_{12}(T)+m_3\dot{q}_{32}(T)-(m_1+m_3)\dot{q}_{42}(T))=0,$$
$$\delta_{\xi^{(6)}}\mathcal{A}(q)=(-m_2\dot{q}_{22}(T)+m_2\dot{q}_{42}(T))=0.$$
By the above three equalities and $\sum_{i=1}^4 m_i\dot{q}_{ij}(T)=0$ for $j=1,2$ and $m_1=m_3, m_2=m_4$, it is easy to derive that the relation  \eqref{TT} holds.

For $ i=1,2,3$, let $\xi^{(i)}(t)\in \mathcal{P}(\mathbf{A},\mathbf{B})$ satisfy
$\xi^{(i)}(T)=0$ and
$$\xi^{(i)}(0)= \left( \begin{array}{cc} 0 & -a_3\\
-a_1 & a_2 \\0 & \frac{-m_2a_2-m_4a_2+m_1a_3}{m_3}\\
a_1 & a_2\end{array} \right) R(\theta),$$
  where $a_i=1,$ $ a_j=0$  if $j\not=i.$ Then
   $$ \delta_{\xi^{(1)}}\mathcal{A}(q)=-\left( m_2<\dot{q}_{2}(0), (-1,0)R(\theta)>+m_4<\dot{q}_4(0), (1,0)R(\theta)>\right)=0,$$
   which is
\begin{equation}\label{pa1}
-\dot{q}_{21}\cos(\theta)+\dot{q}_{22}\sin{\theta}+\dot{q}_{41}\cos(\theta)-\dot{q}_{42}\sin(\theta)=0.
\end{equation}
$$ \delta_{\xi^{(2)}}\mathcal{A}(q)= -\left( m_2<\dot{q}_{2}(0), (0,1)R(\theta)>+(m_2+m_4)<\dot{q}_3(0), (0,-1)R(\theta)>\right.$$ $$\left.+ m_4<\dot{q}_4(0), (0,1)R(\theta)>\right)=0;$$
\begin{equation}\label{pa2}
m_2(\dot{q}_{21}\sin(\theta)+\dot{q}_{22}\cos{\theta})-(m_2+m_4)(\dot{q}_{31}\sin(\theta)+\dot{q}_{32}\cos{\theta})
+m_4(\dot{q}_{41}\sin(\theta)+\dot{q}_{42}\cos{\theta})=0.
\end{equation}
$$ \delta_{\xi^{(3)}}\mathcal{A}(q)= -\left( m_1 <\dot{q}_{1}(0), (0,-1)R(\theta)>+m_1 <\dot{q}_{3}(0), (0,1)R(\theta)>\right)=0;$$
\begin{equation}\label{pa3}
-(\dot{q}_{11}\sin(\theta)+\dot{q}_{12}\cos{\theta})+(\dot{q}_{31}\sin(\theta)+\dot{q}_{32}\cos{\theta})
=0.
\end{equation}
Let $$A_{i1}=\dot{q}_{i1}-\dot{q}_{i1}\cos(2\theta)+\dot{q}_{i2}\sin(2\theta),$$
$$A_{i2}=\dot{q}_{i2}+\dot{q}_{i1}\sin(2\theta)+\dot{q}_{i2}\cos(2\theta),$$
$$A_{j1}=\dot{q}_{j1}-\dot{q}_{k1}\cos(2\theta)+\dot{q}_{k2}\sin(2\theta),$$
$$A_{j2}=\dot{q}_{j2}+\dot{q}_{k1}\sin(2\theta)+\dot{q}_{k2}\cos(2\theta),$$
for $i=1,3$ and $(j,k)=(2,4)$ or $(j,k)=(4,2)$. Because $\sum_{i=1}^4 m_i\dot{q}_{ik}=0$ for $k=1,2$, we have
\begin{equation}\label{dqt}
m_1A_{11}+m_2A_{21}+m_3A_{31}+m_4A_{41}=0, \hspace{1cm} m_1A_{12}+m_2A_{22}+m_3A_{32}+m_4A_{42}=0
\end{equation}
by using the fact $m_1=m_3$ and $m_2=m_4$. By using the trigonometric identities $\cos(\theta)=\cos(2\theta)\cos(\theta)+\sin(2\theta)\sin(\theta)$ and $\sin(\theta)=\sin(2\theta)\cos(\theta)-\cos(2\theta)\sin(\theta)$, from equation \eqref{pa1}, we have
$$ -\dot{q}_{21}\cos(\theta)+\dot{q}_{22}\sin{\theta}+\dot{q}_{41}(\cos(2\theta)\cos(\theta)+\sin(2\theta)\sin(\theta))$$
$$-\dot{q}_{42}(\sin(2\theta)\cos(\theta)-\cos(2\theta)\sin(\theta))=0,$$
which is
\begin{equation}\label{pa11}
-A_{21}\cos(\theta)+A_{22}\sin{\theta}=0.
\end{equation}
Similarly from equation \eqref{pa1}, we also have
\begin{equation}\label{pa12}
A_{41}\cos(\theta)-A_{42}\sin{\theta}=0.
\end{equation}
From equation \eqref{pa2} and $\sum_{i=1}^4 m_i\dot{q}_{ik}=0$ , we have
$$-m_1(\dot{q}_{11}\sin(\theta)+\dot{q}_{12}\cos{\theta})-(m_3+m_2+m_4)(\dot{q}_{31}\sin(\theta)
+\dot{q}_{32}\cos{\theta})=0.$$
Thanks to  equation \eqref{pa3} and above equation,  we have
\begin{equation}\label{pa21}
\dot{q}_{31}\sin(\theta)+\dot{q}_{32}\cos{\theta}=0,
\end{equation}
\begin{equation}\label{pa22}
\dot{q}_{11}\sin(\theta)+\dot{q}_{12}\cos{\theta}=0,
\end{equation}
and
\begin{equation}\label{pa23}
(\dot{q}_{21}\sin(\theta)+\dot{q}_{22}\cos{\theta})
+(\dot{q}_{41}\sin(\theta)+\dot{q}_{42}\cos{\theta})=0.
\end{equation}
By adding the product of equation \eqref{pa1} and $\sin(\theta)$ to the product of equation \eqref{pa23} and $\cos(\theta),$ the sum is $A_{22}=0$. Then by equation \eqref{pa11}, $A_{21}=0$. Similarly, we have $A_{41}=A_{42}=0$. From equation \eqref{pa21},
$$\dot{q}_{31}\sin(\theta)+\dot{q}_{32}\cos{\theta} +\dot{q}_{31}(\sin(2\theta)\cos(\theta)-\cos(2\theta)\sin(\theta))
+\dot{q}_{32}(\cos(2\theta)\cos(\theta)+\sin(2\theta)\sin(\theta))=0,
$$
which is
\begin{equation}\label{q24}
A_{31}\sin(\theta)+A_{32}\cos(\theta)=0.
\end{equation}
By definition,
$A_{31}\sin(2\theta)+A_{32}\cos(2\theta)= (\dot{q}_{31}-\dot{q}_{31}\cos(2\theta)$ $+\dot{q}_{32}\sin(2\theta))\sin(2\theta) +(\dot{q}_{32}+\dot{q}_{31}\sin(2\theta)+\dot{q}_{32}\cos(2\theta))\cos(2\theta),$ which implies $A_{31}\sin(2\theta)+A_{32}\cos(2\theta)=A_{32}$. Hence,
\begin{equation}\label{q25}
A_{31}\cos(\theta)-A_{32}\sin(\theta)=0.
\end{equation}
From equations \eqref{q24} and \eqref{q25}, we have $A_{31}=A_{32}=0$.
Then the equation \eqref{dqt} imply that $A_{kj}=0$ for $k=1,2,3,4$ and $j=1,2$. Because the relations \eqref{T0} is equivalent to $A_{kj}=0$, we complete the proof that $q(t)$ connects very well at $t=2T$. This also completes the proof of Theorem \ref{main1}.  \\

\end{proof}

\section{Existence of New Solutions}

In this section, we prove theorem \ref{main2} by showing that there exists a minimizing path which is different from the homographic motion for any given $(\theta,\mu)\in \Omega$ in equation \eqref{thmu}. We show that the action of the test path $\bar{q}(t)$ with constant velocity for a given SPBC $\vec{a}$ is smaller than the minimum action of the path $q^\circ(t)$ for $\vec{a}^\circ$ which is extended to a homographic motion. Therefore, there exists a minimizer with lower action than the action of homographic solution.

First, the homographic motion can be obtained by extending the corresponding minimizing path $q^{\circ}(t)$ on $[0, T]$ of a particular SPBC $\vec{a}^\circ$ in $\Gamma$. Both the initial and ending configurations $q^{\circ}(0)$ and $q^{\circ}(T)$ are rhombus and they satisfy the SPBC. Second, we assume that the test path $\bar{q}(t)$ is formed by  connecting the straight line with constant velocity from the starting configuration $Qstart$ to the ending configuration $Qend$ for a given $\vec{a}$. Both actions $\mathcal{A}(q^{\circ}(t))$ and $\mathcal{A}(\bar{q}(t))$ on $[0,T]$ are explicit continuous functions of $(\theta,\mu)$  given by formula \eqref{Ad} and \eqref{At} respectively.  Although the corresponding actions can be calculated by hand, they are computed by a Matlab program.

\subsection{ Action of the path which is extended to a homographic  solution}\label{sec31}
The configuration $q$ is called a central configuration if $q$ satisfies the following nonlinear
algebraic equation system:
\begin{equation}\label{CC}
\lambda (q_i-c)-\sum_{j=1,j\not= i}^{n} \frac{m_j(q_i-q_j)
}{|q_i-q_j|^3}=0, \hspace{1cm} 1\leq i\leq n,
\end{equation}
for a constant $\lambda$, where $c=(\sum m_i q_i)/M$ is the center
of mass and $M=m_1+m_2+\cdots +m_n$ is the total mass. We recall the fact that coplanar central configurations always admit homographic solutions where each body executes a similar Keplerian ellipse of eccentricity $e$, $0\leq e\leq 1$. When $e=0$, the relative equilibrium solutions are consisting  of uniform circular motion for each of the masses about the common center of mass. When $e=1$, the homographic solutions degenerate to a homothetic solution which includes total collision, together with a symmetric segment of ejection. Gordon found that for fixed period $\mathcal{T}^{\circ}$, all of the homographic solutions have the same action (\cite{GW}).
 Consider the homographic solution of the four-body problem in rhomboid configuration $$q_{k}^\circ (t)=r_k(cos(\omega t+\rho_k),sin(\omega t+\rho_k)), k=1, 2,3, 4,$$
where $r_k>0$ is the radius and $\rho_k=\frac{(k-1)\pi}{2} + \alpha_0, k=1,2,3,4$. $\omega$ and $\alpha_0$ are chosen to make the boundary configurations $q^\circ(0)$ and $q^\circ(T)$ satisfy SPBC.   By the results of the central configurations with some equal masses \cite{AFS,LS,SX}, we can assume that $r_1=r_3$ and $r_2=r_4$ because $m_1=m_3$ and $m_2=m_4$. It is easy to find the following relations by Newtonian equations \eqref{Newton}:
\begin{equation}\label{CC1}
-{\omega}^{2}+\,{\frac {{ 2 m_2}}{ \left( {{ r_1}}^{2}+{{ r_2}}^{2}
 \right) ^{3/2}}}+\,{\frac {{ m_1}}{4{{ r_1}}^{3}}}=0,
\end{equation}
\begin{equation}\label{CC2}
-{\omega}^{2}+\,{\frac {{ 2 m_1}}{ \left( {{ r_1}}^{2}+{{ r_2}}^{2}
 \right) ^{3/2}}}+\,{\frac {{ m_2}}{4{{ r_2}}^{3}}}=0.
\end{equation}
 The minimum period is $\mathcal{T}^{\circ}=\frac{2\pi}{\omega}$. For any given positive $\omega,$ $m_1$ and $m_2$, $r_1$ and $r_2$ are uniquely determined by the above two equations.  The minimum value of the  action functional \eqref{Action} on $[0,T]$ could be computed as
 $$\mathcal{A}(q^{\circ}(t))=\int_0^T\sum_{k=1}^4\frac{1}{2}m_k|\dot{q}_k^{\circ}(t)|^2 +U(q^{\circ}(t))dt $$
 $$=\left({ m_1}\,{\omega}^{2}{{ r_1}}^{2}+{ m_2}\,{\omega}^{2}{{ r_2}}^{2}+4\,{
\frac {{ m_1}\,{ m_2}}{\sqrt {{{ r_1}}^{2}+{{ r_2}}^{2}}}}+
\,{\frac {{{ m_1}}^{2}}{{2 r_1}}}+\,{\frac {{{ m_2}}^{2}}{{
 2r_2}}}\right) T$$
 $$=3\,{\omega}^{2} \left( { m_1}\,{{ r_1}}^{2}+{{ r_2}}^{2}{ m_2}
 \right)T.$$

\begin{figure}
\includegraphics[height=5cm,width=.32\textwidth]{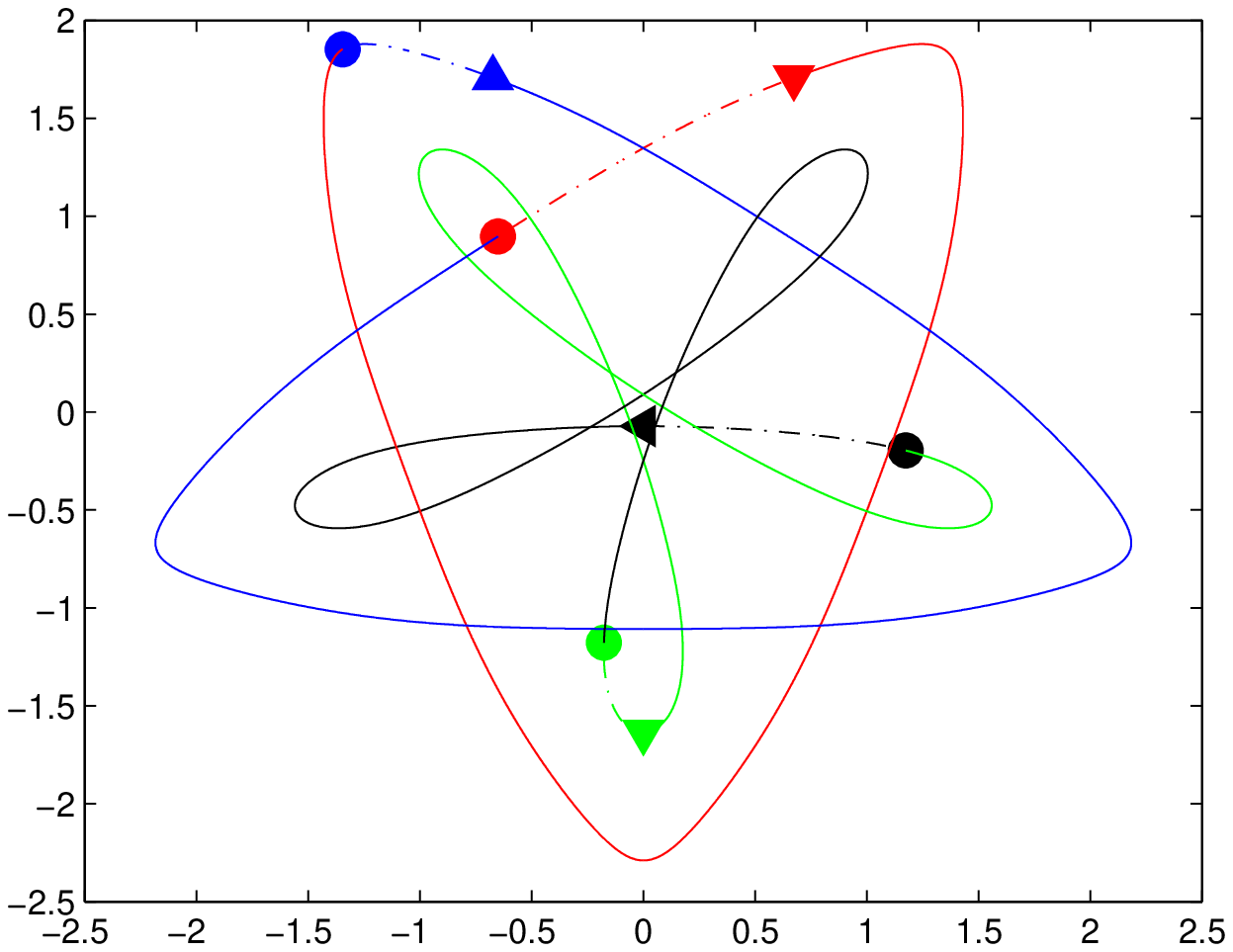}
\includegraphics[height=5cm,width=.32\textwidth]{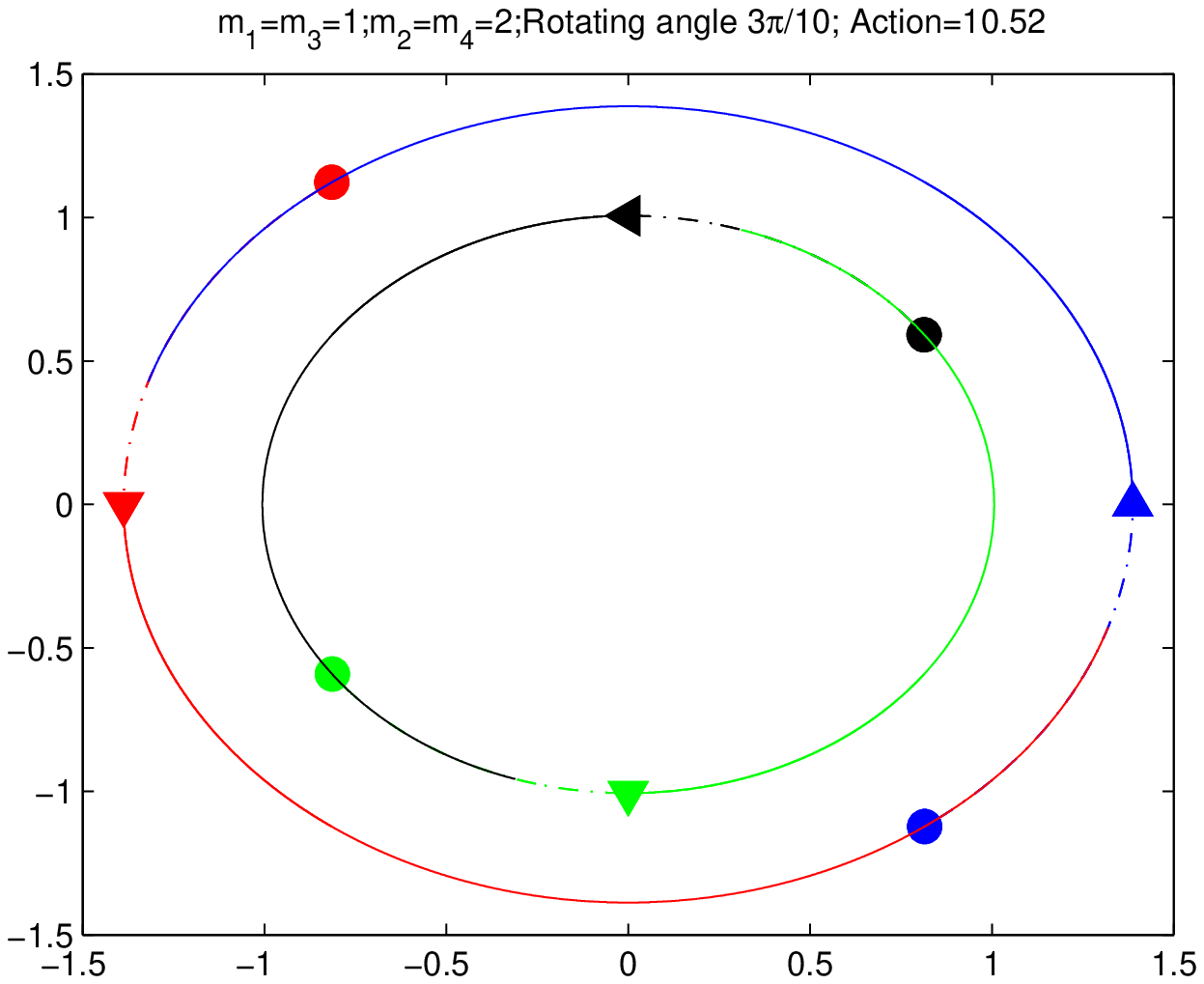}
\includegraphics[height=5cm,width=.32\textwidth]{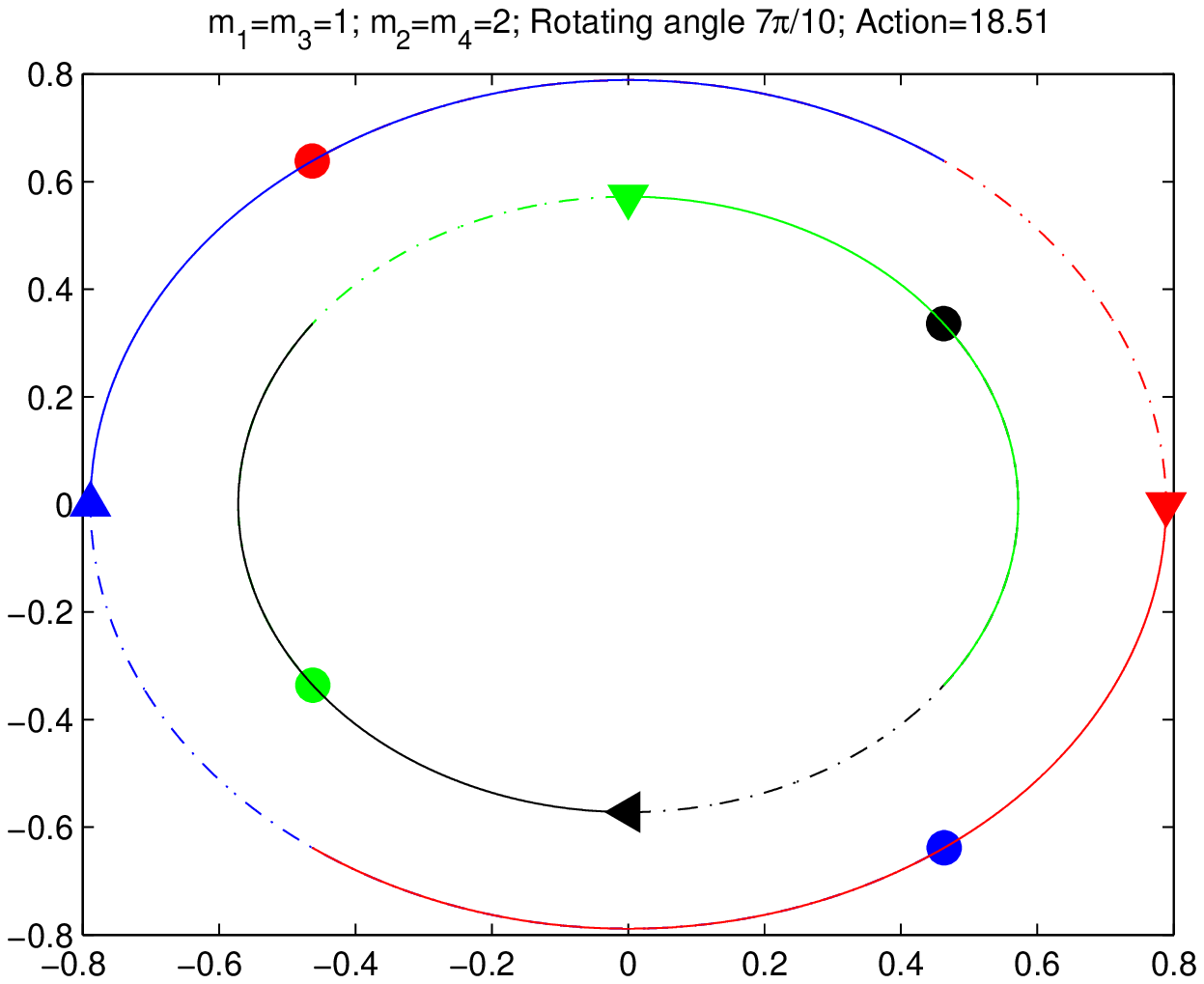}

\caption{\small $(\theta,\mu)=(\frac{4\pi}{5},2)$. Left: Minimizing path with action 9.748; Middle: Homographic solution $\omega=\frac{3\pi}{10}$ with action 10.52; Right: Homographic solution $\omega=\frac{7\pi}{10}$ with action 18.51.   }\label{fighomo}\end{figure}

For fixed $(\theta,\mu)$, there are different central configurations satisfying SPBC to generate homographic solutions. For example, when $\theta=\frac{4\pi}{5}$, $\omega=\frac{\pi}{2}-\frac{\pi}{5}=\frac{3\pi}{10}$ (see middle in Figure \ref{fighomo}) or $\omega=\frac{\pi}{2}+\frac{\pi}{5}=\frac{7\pi}{10}$ (see right in Figure \ref{fighomo}). But in order to have a minimizing action over the time interval $[0,T]$,  the bodies in homographic solutions should rotate an angle  as small as possible in $[0,T]$. By the structure of our prescribed boundary conditions, $\omega=\left|\frac{\pi}{2}-\theta\right|\frac{1}{T}$ if $0<\theta<\pi$, or $\omega=\left|\frac{3\pi}{2}-\theta\right|\frac{1}{T}$ if $\pi\leq\theta<2\pi$. So, for a given $(\theta,\mu)$, the action of the homographic solution in $(0,T)$ is given by equation \eqref{Ad}.

In particular, if $\mu=1$, i.e. $m_1=m_2=1$, then $r_1=r_2$ and $r_1^2=\left(\frac{2\sqrt{2}+1}{4}\omega^{-2}\right)^{2/3}$ by equation \eqref{CC1}.
\begin{equation}\label{Ade}
\mathcal{A}_{\diamondsuit}(\theta,1)=\left\{ \begin{array}{l} 6\left(\frac{2\sqrt{2}+1}{4}\right)^{2/3} \left(\left|\frac{\pi}{2}-\theta\right|\right)^{2/3}T^{1/3} \hbox{ if } 0<\theta<\pi, \\
6\left(\frac{2\sqrt{2}+1}{4}\right)^{2/3} \left(\left|\frac{3\pi}{2}-\theta\right|\right)^{2/3}T^{1/3}\hbox{ if } \pi\leq \theta<2\pi.
\end{array}\right.
\end{equation}
Given any $(\theta,\mu)\in (0,2\pi)\times \mathbf{R}^+$, let $k$ be a positive parameter such that $k\theta\in(0,2\pi)$. Let $\varrho= \frac{\omega(k\theta)}{\omega(\theta)}$. By the equations \eqref{CC1} and \eqref{CC2}, it is easy to derive
\begin{equation}\label{Ad1}
\mathcal{A}_{\diamondsuit}(k\theta,\mu)=\varrho^{\frac{2}{3}}\mathcal{A}_{\diamondsuit}(\theta,\mu).
\end{equation}
%$\omega=\frac{\pi}{2}-\frac{\pi}{5}=\frac{3\pi}{10}$ (see middle in Figure \ref{fig2}) or $\omega=\frac{\pi}{2}+\frac{\pi}{5}=\frac{7\pi}{10}$ (see right in Figure \ref{fig2}) in order to form a rhombus configuration. Let $T=1$, $\omega=\frac{3\pi}{10}$, and $m_1=1$. Given $m_2$, we can find the unique $r_1,r_2$ and the corresponding action of rhombus solution. For example, when $m_2=1$, $r_1=r_2=  1.025193977$ and the corresponding action is $ 5.601516212$ which is bigger than the action 5.2329087226 of test path. So there exists a minimizing path which will generate a star pentagon periodic solution.

\subsection{ Action of a test path}
For any fixed SPBC $\vec{a}_{test}$ and $T=1$,  the test path $\bar{q}(t)$ with constant velocity connecting the structural prescribed boundaries $Qstart$ and $Qend$ is given by equation \eqref{Tpath}. Then the action $\mathcal{A}_{Tpath}$
of the test path is an explicit continuous function of $(\theta,\mu)$ and it is given by equation \eqref{At}. To get lower action $\mathcal{A}_{Tpath}$ of a test path, we need to pick up an appropriate SPBC $\vec{a}_{test}$. It is better to have several SPBC for different values of $(\theta,\mu)$. For example, we use the minimizer $\vec{a}= [ 0.6676542303,$ $ 1.11499232, $ $0.5099504088,$ $ 0.6676542314,$ $ 1.11499232,$ $ 0.5099504078]$ for $\theta=\frac{4}{5}$ and $\mu=1$ as the fixed SPBC $\vec{a}_{test}$ to estimate the action of the test path for all $(\theta,\mu)$. The formula \eqref{At} becomes
$$ \mathcal{A}_{Tpath}(\theta,\mu)=0.81184+
 2.48644\,\cos \left( \theta \right) -
 0.80792\,\sin \left( \theta \right) +2.4864 \,{\mu}^{-1} +0.81184\,\mu +
$$
  $$
2.4864\,\cos \left( \theta \right) \mu- 0.80792\,\sin \left( \theta \right) \mu
+
 2.4864 \,{\mu}^{2}+\int_{0}^1 [3.0861\mu^2+ t (  1.3618\,\sin \left( \theta
 \right) {\mu}^{2}+
 $$
  $$
7.2471\,\cos \left( \theta \right) \mu- 2.9777\,\sin \left( \theta \right) \mu+6.1721\,{\mu}^{2}+
 2.8578\,\cos \left( \theta \right) {\mu}^{2} )+
{t}^{2} (  2.6985\,\mu+$$
 $$
1.3618\,\sin \left( \theta
 \right) {\mu}^{2}+7.2471\,\cos \left( \theta \right) \mu- 2.9777\,\sin \left( \theta \right) \mu+$$ $$ 3.8979\,{\mu}^{2}+
 2.8578\,\cos \left( \theta \right) {\mu}^{2}+ 4.9729 )
]^{-1/2}+\cdots dt
$$
Here we only list the kinetic energy and the one term of potential energy. The integrals only involve the form of $\int_0^1 {(a+bt+ct^2)^{-1/2}} dt $ which can be integrated explicitly by trigonometric substitution.

\subsection{The possible region $\Omega$}
Here we describe how we do the numerical computations for the region $\Omega_{\vec{a}_{test}}$ on which $\mathcal{A}_{\diamondsuit}(\theta,\mu)>\mathcal{A}_{Tpath}(\theta,\mu)$. We are going to compute  $\mathcal{A}_{Tpath}(\theta,\mu)$  for the fixed SPBC $\vec{a}_{test}= [ 0.6676542303,$ $ 1.11499232, $ $0.5099504088,$ $ 0.6676542314,$ $ 1.11499232,$ $ 0.5099504078]$ which is a minimizer of $\tilde{\mathcal{A}}$ for $(\theta,\mu)=(\frac{4\pi}{5},1)$.

For fixed $\mu=1$, the graph of $\mathcal{A}_{\diamondsuit}(\theta,1)$ can be easily obtained from equation \eqref{Ade} and the graph of $\mathcal{A}_{Tpath}(\theta,1)$ is given by figure \ref{ActionMu1}. By direct computation, $\mathcal{A}_{\diamondsuit}(0.78\pi,1)=5.3497>5.3444=\mathcal{A}_{Tpath}(0.78\pi,1)$ but $\mathcal{A}_{\diamondsuit}(0.77\pi,1)=5.2216<5.4085 =\mathcal{A}_{Tpath}(0.77\pi,1)$ ,  and  $\mathcal{A}_{\diamondsuit}(1.11\pi,1)=6.6722>=6.5124\mathcal{A}_{Tpath}(1.11\pi,1)$ but $\mathcal{A}_{\diamondsuit}(1.12\pi,1)=6.5576<6.6465=\mathcal{A}_{Tpath}(1.12\pi,1)$. So there exist $0.77\pi<\theta_0<0.78\pi$ and $1.11\pi<\theta_1<1.12\pi$ such that for any $\theta\in (\theta_0,\theta_1)$, $\mathcal{A}_{\diamondsuit}(\theta,1)>\mathcal{A}_{Tpath}(\theta,1)$.

For fixed $\theta=\frac{4\pi}{5},$ we compute $\mathcal{A}_{\diamondsuit}( \frac{4\pi}{5},\mu)$ and $\mathcal{A}_{Tpath}( \frac{4\pi}{5},\mu)$. We find that $\mathcal{A}_{\diamondsuit}( \frac{4\pi}{5},\mu)>\mathcal{A}_{Tpath}( \frac{4\pi}{5},\mu)$ for $\mu_0<  \mu< \mu_1$ where $0.8<\mu_0<0.9$ and $1.2<\mu_1<1.3$ (see figure \ref{ActionTheta45}). But the action for the minimizing path is less than the action of homographic solution for larger range of $\mu$. In fact, we can prove that $\mathcal{A}_{\diamondsuit}( \frac{4\pi}{5},\mu)>\tilde{\mathcal{A}}$ when $ 0.5\leq \mu\leq 2.1$.
\begin{figure}
\includegraphics[height=6cm,width=.80\textwidth]{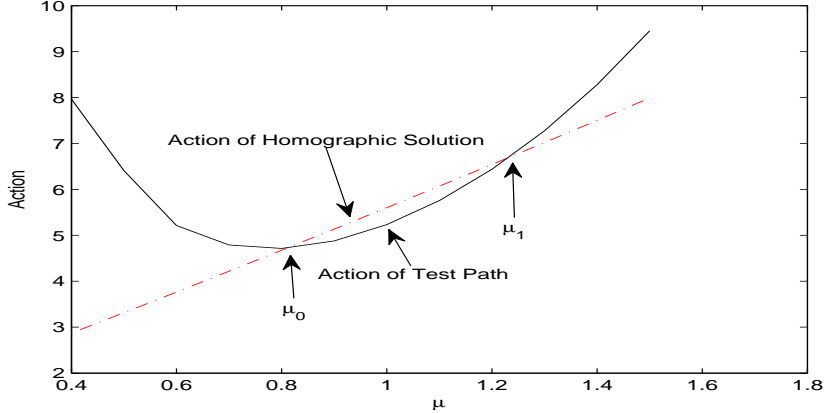}
\caption{Fixed $\theta=\frac{4\pi}{5}$. Use $\vec{a}_{test}= [ 0.6676542303,$ $ 1.11499232, $ $0.5099504088,$ $ 0.6676542314,$ $ 1.11499232,$ $ 0.5099504078]$ as fixed SPBC for the test path. Test path has lower action than homographic solution when $\mu_0<\mu<\mu_1 $.  }\label{ActionTheta45}\end{figure}

We fix $\mu$ at different value to find the intervals of $\theta$ such that $\mathcal{A}_{\diamondsuit}(\theta,\mu)>\mathcal{A}_{Tpath}(\theta,\mu)$. We use $\mathcal{A}_{\diamondsuit}(\frac{4\pi}{5},\mu)$ to generate $\mathcal{A}_{\diamondsuit}(\theta,\mu)$ by equation \eqref{Ad1}.  For the fixed SPBC $\vec{a}_{test}$, the region $\Omega_{\vec{a}_{test}}$ is presented in figure \ref{Omega1}.

  It is easy to know that the action of homographic solution $\mathcal{A}_{\diamondsuit}(\theta,\mu)$ is independent of the choice of SPBC $\vec{a}_{test}$ but the action of  test path $\mathcal{A}_{Tpath}(\theta,\mu)$ strongly depends on the choice of SPBC $\vec{a}_{test}$. The figure \ref{Omega2} shows the different $\Omega_{\vec{a}_{test}}$ when the test pathes are generated from different SPBC $\vec{a}_{test}$ which are the minimizers of $\tilde{\mathcal{A}}$ for same angle $\theta=\frac{4\pi}{5}$ and different mass ratio $\mu$. The SPBC $\vec{a}_{test}$ corresponding to the region $\Omega_{\vec{a}_{test}}$ in figure \ref{Omega2} from top to bottom are  \\ (1) $\vec{a}_{test}=[ 0.8347577868, $ $0.8492284757,$ $ 1.107411045, $ $0.6740939528, $ $1.7071110, $ $0.072136065]$ for  $(\theta,\mu)=(\frac{4\pi}{5}, 2)$;\\ (2) $\vec{a}_{test}=[ 0.6676542303,$ $ 1.11499232, $ $0.5099504088, $ $0.6676542314, $ $1.11499232,$ $ 0.5099504078]$ for  $(\theta,\mu)=$ $(\frac{4\pi}{5}, 1)$;\\ (3)  $\vec{a}_{test}=[ 0.6216336897,$ $ 1.197204657,$ $ 0.347804861, $ $0.6658203645, $ $0.9561601763, $ $ 0.6379731628]$ for  $(\theta,\mu)=(\frac{4\pi}{5}, 0.8)$;\\ (4) $\vec{a}_{test}=[ 0.5350313653, $ $1.354931439,$ $ 0.0572523078,$ $ 0.6625485,$ $ 0.674032487, $ $0.8789531194]$ for  $(\theta,\mu)=(\frac{4\pi}{5}, 0.5)$.\\
   The possible admissible set $\Omega$  contains the union of the regions in figure \ref{Omega2}  but larger range can be expected by refining the test path. %$(0.78\pi,1.1\pi)\times (0.01, 10)$ from our numerical computations

\begin{figure}
\includegraphics[height=6cm,width=.80\textwidth]{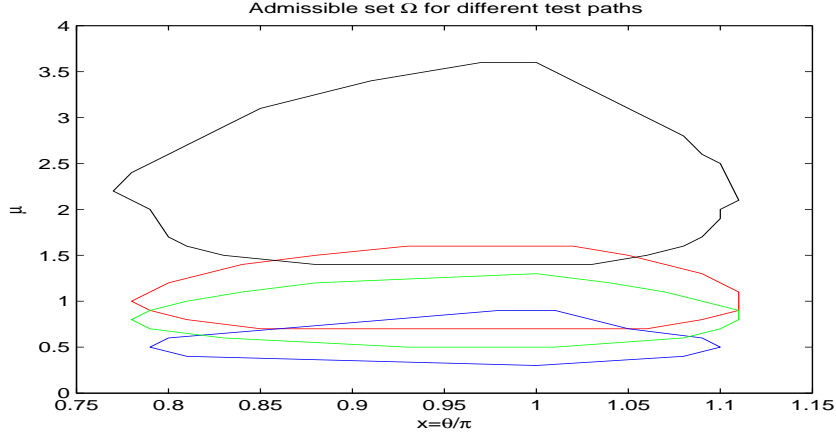}
\caption{Different sets $\Omega_{\vec{a}_{test}}$ for different test path. The SPBC $\vec{a}_{test}$ from top to bottom are minimizers for same rotation angle $\theta=\frac{4\pi}{5}$ and different mass ratio $\mu=2,1,0.8, 0.5$.   }\label{Omega2}\end{figure}

\subsection{Classification of solutions according to  $(\theta,\mu)$}
Now it is ready to prove theorem \ref{main2} and to classify the minimizing solutions based on the rotation angle $\theta$ and mass ratio $\mu$. Since the action of test path is lower than the action of homographic solutions when $(\theta,\mu)\in \Omega$, there exists a minimizer of $\tilde{\mathcal{A}}$ by theorem \ref{Thm:Ex} and theorem \ref{main1}. Its corresponding path $q^*(t)$ in $[0,T]$ can be extended to a classical non-collision non-homographic solution of Newtonian equation \eqref{Newton} for time $t\in [0,\infty)$. By the extension formula \eqref{qet}, $$q(t)=\sigma^{k}(q(t-2kT))R(-2k\theta) \hbox{ for } t\in (2kT,(2k+2)T]
  \hbox{ and } k\in \mathbf{Z}^+,
$$
and $\sigma^2=[1,2,3,4]$, the orbit of the solution is assembled out the sides $q_i(t), t\in[0,4T]$  by rotation only. For any $(\theta,\mu)\in \Omega$, the non-homographic solution $q(t)$ can be classified as follows.\\
(1) By the extension formula, it is easy to show that $q(t)$ is a quasi-periodic solution if $\theta$ is not commensurable with $\pi$ by the fact that the rotation matrix $R(-2k\theta)$ can not be an identity matrix for any integer $k$. Some quasi-periodic solutions are illustrated in figure \ref{qua1}. \\
(2) If $\theta$ is commensurable with $\pi$ and $\theta=\frac{P}{Q}\pi$ where the positive integers $P$ and $Q$ are relatively prime, then $q(t)=\sigma^{2Q}(q(t-4QT))R(-4P\pi)=q(t-4QT)$ which implies that $q(t)$ is a periodic solution.
The trajectory sets $\{q_i(t)| t\in(2k,(2k+2)T)\}$  of the given $i$th-body are all different for $0\leq k<Q$, since the rotation matrix $R(-2k\frac{P}{Q}\pi)$ is not identity matrix. The four trajectories on which the four body travel  in $t\in(0,2QT)$ are all different. So the minimum period is between $2QT$ and $4QT$.

  \begin{itemize}
 \item When $Q$ is even, then $q(t)=\sigma^{Q}(q(t-2QT))R(-2P\pi)=q(t-2QT)$ because $\sigma^2=[1,2,3,4]$. So the four different trajectories in $t\in(0,2QT)$  are closed on their own at $t=2QT$, i.e. $q(2QT)=q(0)$. $q_i(k_1T)\not=q_j(k_2T)$ if $i\not=j$ and $0\leq k_1, k_2\leq Q$, i.e. trajectories of $i$-th body and $j$-th body do not meet at the end of any piece orbit.  The periodic solution  is non-choreographic and the minimum period is $\mathcal{T}=2QT$. Each closed curve has $\frac{2QT}{4T}=\frac{Q}{2}$ sides. In particular, when $Q=4$, each closed curve is ellipse-like (two sides); when $Q=6$, each closed curve is triangle-like (three sides); when $Q=8$, each closed curve is diamond-like (four sides); and so on. See figure \ref{P1}.

 \item When $Q$ is odd, there are four cases. \\
 $\sigma^Q=\sigma$ and $q(t)=\sigma^{Q}(q(t-2QT))R(-2P\pi)=\sigma(q(t-2QT))$. So $q_1(t+2QT)=q_3(t)$, $q_3(t+2QT)=q_1(t)$, and $q_2(t+2QT)=q_4(t)$, $q_4(t+2QT)=q_2(t)$, which means that $q_1$ chases $q_3$ and $q_2$ chases $q_4$. We have different cases on whether the closed orbit  for $q_2$ and $q_4$ is the same as the closed orbit for $q_1$ and $q_3$. \\

 Case 1: If $\mu\not=1$, the ratio between the distance of the center of mass for $m_1$ and $m_3$ to the origin and the distance of the center of mass for $m_2$ and $m_4$ to origin is $\mu$   at  both $t=0$ and $t=T$. By the extension formula \eqref{qet}, bodies $q_1$ and $q_3$ rotate around their center of mass, and  bodies $q_2$ and $q_4$ rotate around their center of mass respectively. They can not have the same orbits. Otherwise the orbit of the center of mass for $m_1$ and $m_3$ would be the same as the orbit of the center of mass for $m_2$ and $m_4$ since $m_1=m_3$ and $m_2=m_4$. This contradicts to the ratio $\mu\not=1$.  Therefore, if $\mu\not=1$, the periodic solution $q(t;\theta,\mu,\vec{a}_0)$ is a double-choreographic solution. The minimum period is $\mathcal{T}=4QT$. Each closed curve has $Q$ sides. Body $q_1$ chases body $q_3$ on a closed curve and body $q_2$ chases body $q_4$ on another closed curve. $q_1(t+2QT)=q_3(t)$ and $q_3(t+2QT)=q_1(t).$ $q_4(t+2QT)=q_2(t)$ and $q_2(t+2QT)=q_4(t).$ See figure \ref{P21}.\\

 Case 2:  If $\mu=1$ and $P$ is odd, by the extension formula \eqref{qet} it is easy to check that the intersection is empty between the sets of $\{q_1(2kT), q_3(2kT)|\hbox{k is a nonnegative integer}\}$   and   $\{q_2((2k+1)T), q_4((2k+1)T)|\hbox{k is a nonnegative integer}\}$. So the orbit of $q_1$ and $q_3$ can not be the same as the orbit of $q_2$ and $q_4$. The periodic solution $q(t;\theta,\mu,\vec{a}_0)$ is a double-choreographic solution with minimum period $\mathcal{T}=4QT$. Body $q_1$ chases body $q_3$ on a closed curve and body $q_2$ chases body $q_4$ on another closed curve. See figure \ref{P22}.\\

 Case 3: If $\mu=1$, $P$ is even and the initial configuration $q(0)$ is geometrically same to the ending configuration $q(T)$, i.e. $a_{i0}=a_{(i+3)0}$ for $i=1,2,3$, then the set $\{ q_1(kT), q_3(kT)\}$ is equal to the set $\{q_2(kT), q_4(kT)\}$ by the extension formula \eqref{qet}. The orbit of $q_1$ and $q_3$ is  the same as the orbit of $q_2$ and $q_4$. Then the periodic solution is a choreographic solution. The minimum period is $\mathcal{T}=4QT$. The closed curve has $Q$ sides.\\
  (A) If $\frac{Q-1}{2}$ is odd, then the four bodies chase each other on the closed curve in the order of $q_1, q_2, q_3, q_4,$ and then $q_1$, i.e. $q_1(t+QT)=q_2(t),$ $q_2(t+QT)=q_3(t),$ $q_3(t+QT)=q_4(t),$ and $q_4(t+QT)=q_1(t).$ See figure \ref{P23}.\\
  (B) If $\frac{Q-1}{2}$ is even, then the four bodies chase each other on the closed curve in the order of $q_1, q_4, q_3, q_2,$ and then $q_1$, i.e. $q_1(t+QT)=q_4(t),$ $q_4(t+QT)=q_3(t),$ $q_3(t+QT)=q_2(t),$ and $q_2(t+QT)=q_1(t).$ See figure \ref{P24}. \\

  {\bf Case 4: } If $\mu=1$, $P$ is even and the initial configuration $q(0)$ is not geometrically same to the ending configuration $q(T)$, i.e. $(a_{10}, a_{20}, a_{30})\not= (a_{40}, a_{50}, a_{60})$,  then $q_1(k_1T)$ and $q_3(k_1T)$ can not  match with $q_2(k_2T)$ and $q_4(k_2T)$ for any nonnegative integer $k_1$ and $k_2$. So the periodic solution is a double choreographic solution. Each closed curve has $Q$ sides. The minimum period is $\mathcal{T}=4QT$. Body $q_1$ chases body $q_3$ on a closed curve and body $q_2$ chases body $q_4$ on another closed curve. $q_1(t+2QT)=q_3(t)$ and $q_3(t+2QT)=q_1(t).$ $q_4(t+2QT)=q_2(t)$ and $q_2(t+2QT)=q_4(t).$ See figure \ref{case4}.\\

\end{itemize}

\begin{figure}
\includegraphics[height=5cm,width=.32\textwidth]{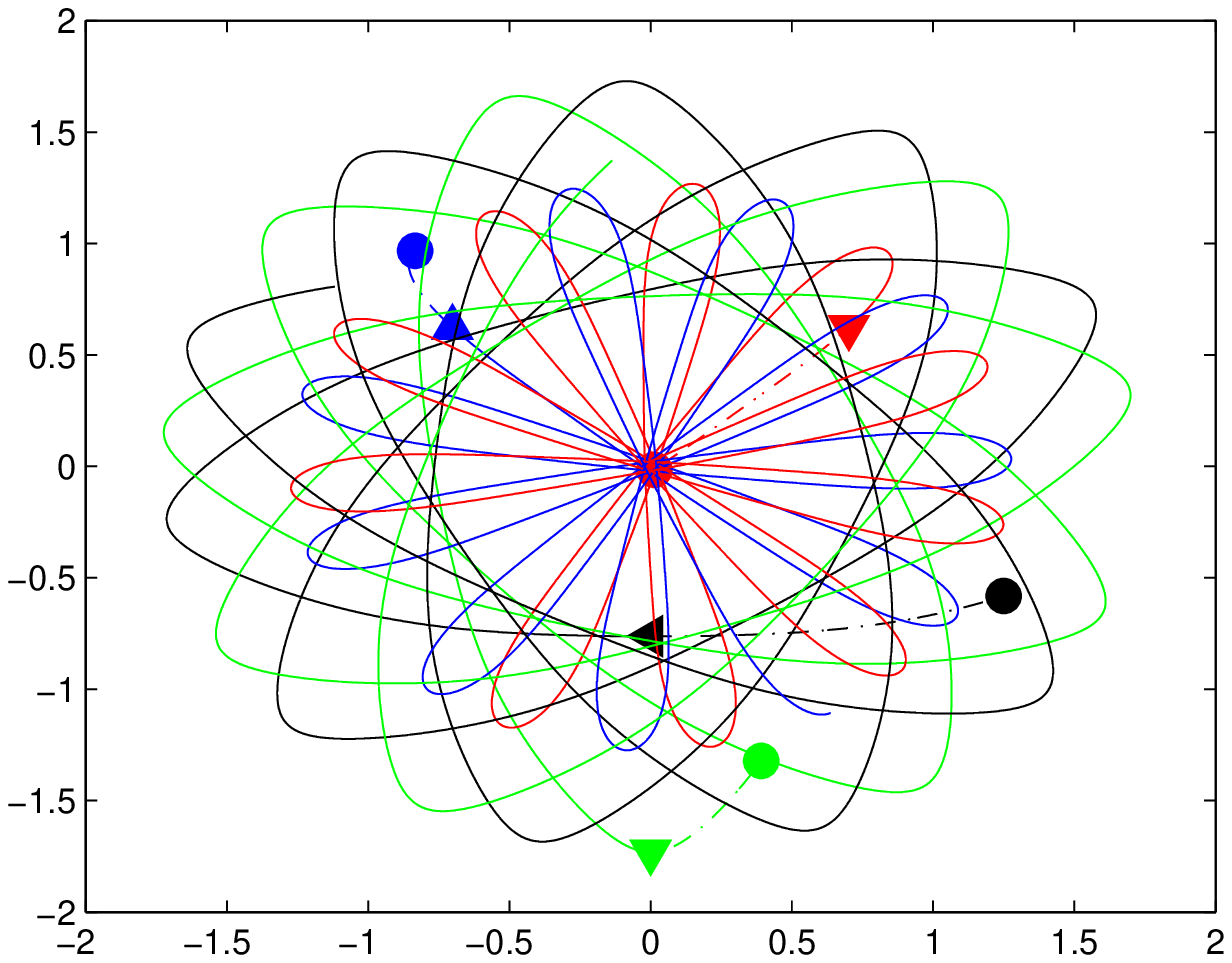}
\includegraphics[height=5cm,width=.32\textwidth]{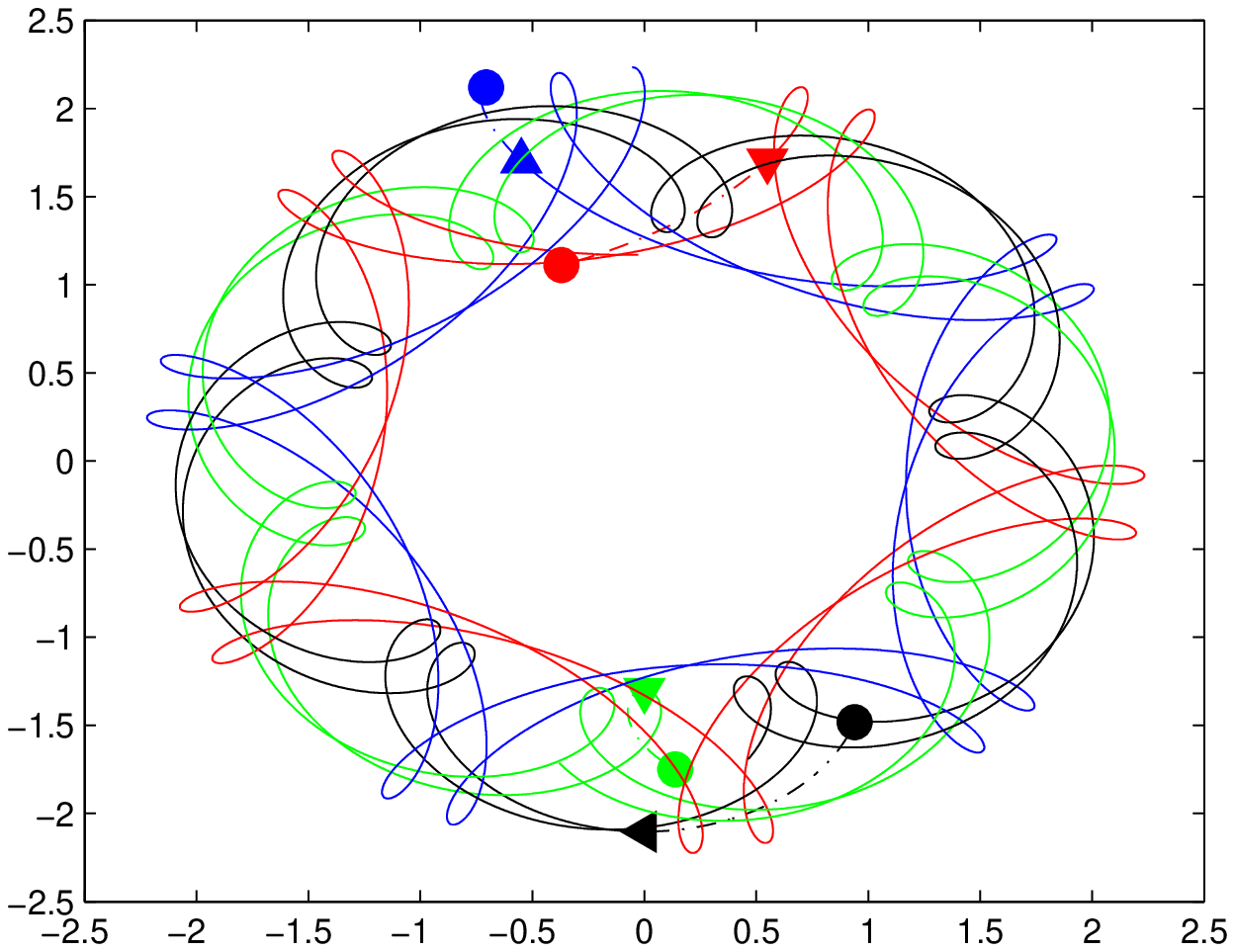}
\includegraphics[height=5cm,width=.32\textwidth]{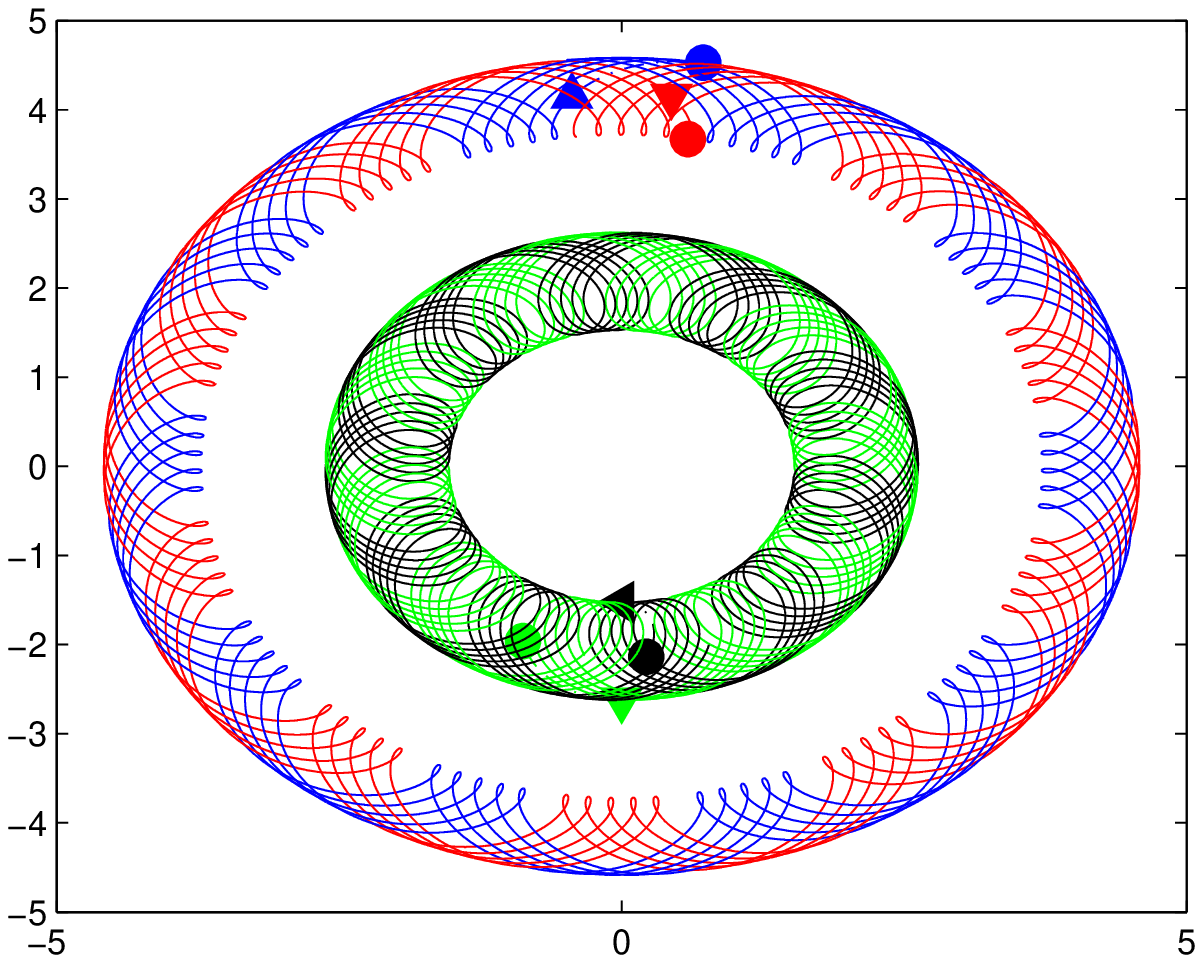}
\caption{\small Quasi-Periodic Solutions. From left to right: $(\theta,\mu)=(2.43,0.5)$ on $[0,40T]$; $(2.82, 1)$ on $[0,40T]$; $(3.3, 2)$ on $[0,220T]$ }
\label{qua1}\end{figure}

\begin{figure}
\includegraphics[height=5cm,width=.32\textwidth]{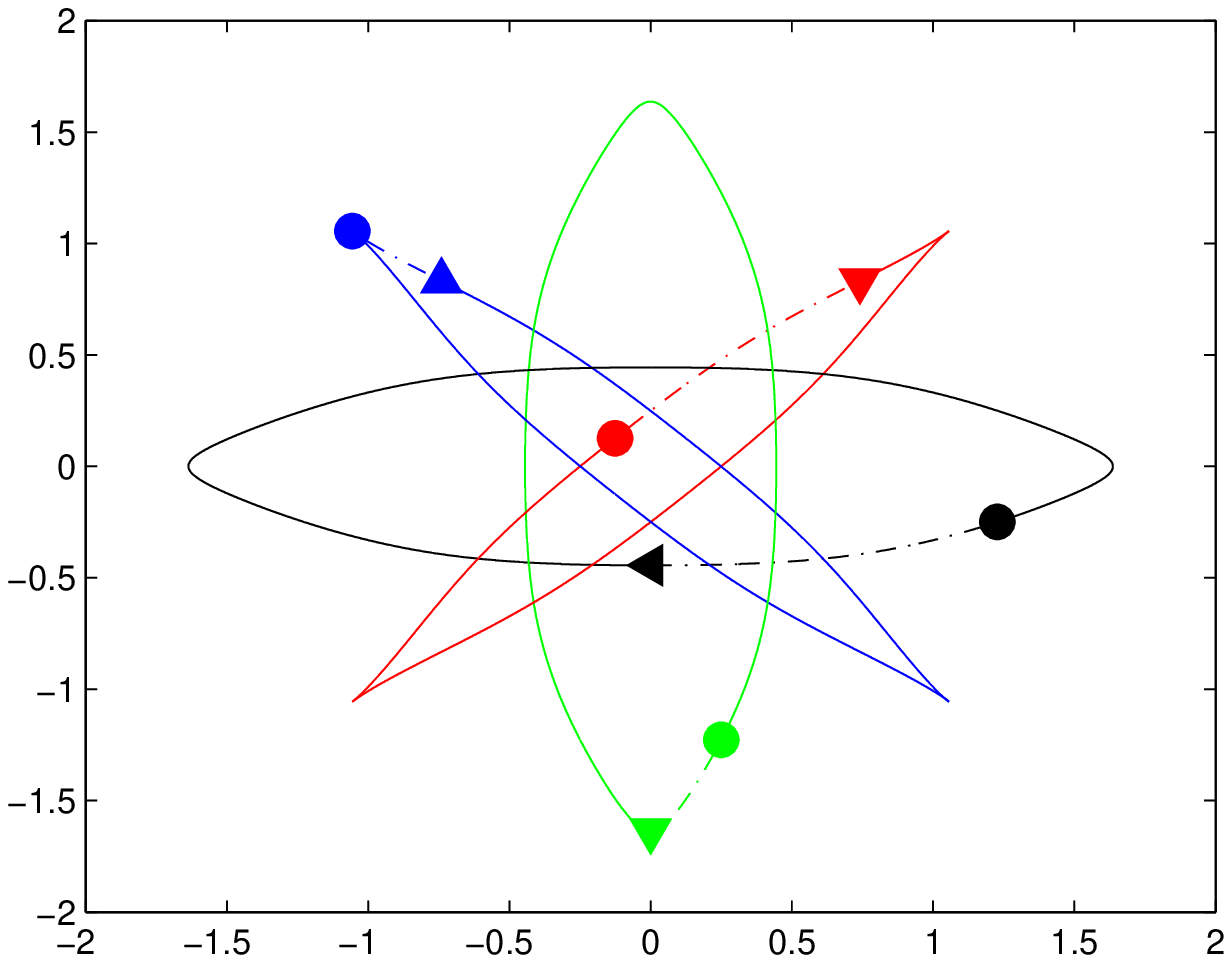}
\includegraphics[height=5cm,width=.32\textwidth]{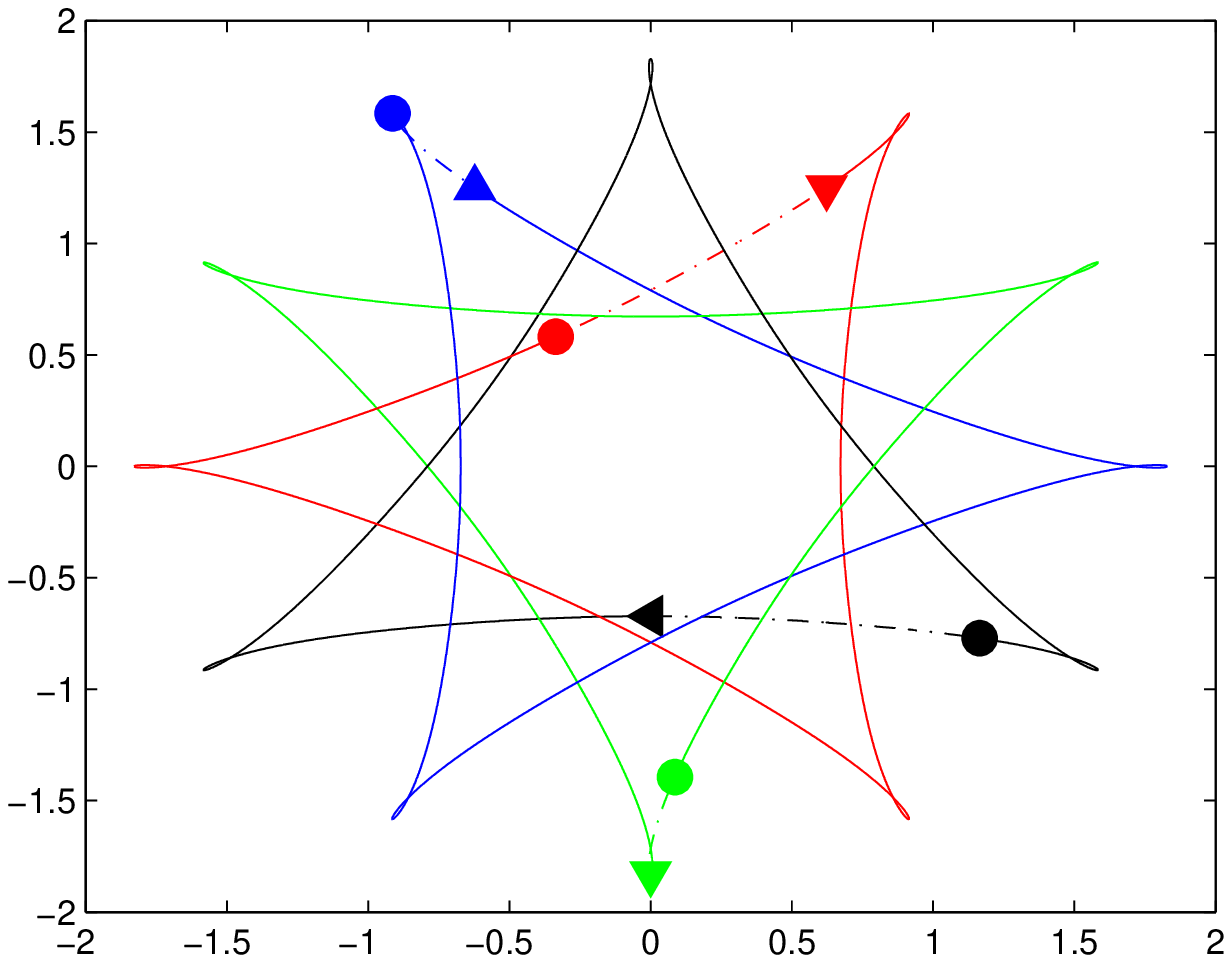}
\includegraphics[height=5cm,width=.32\textwidth]{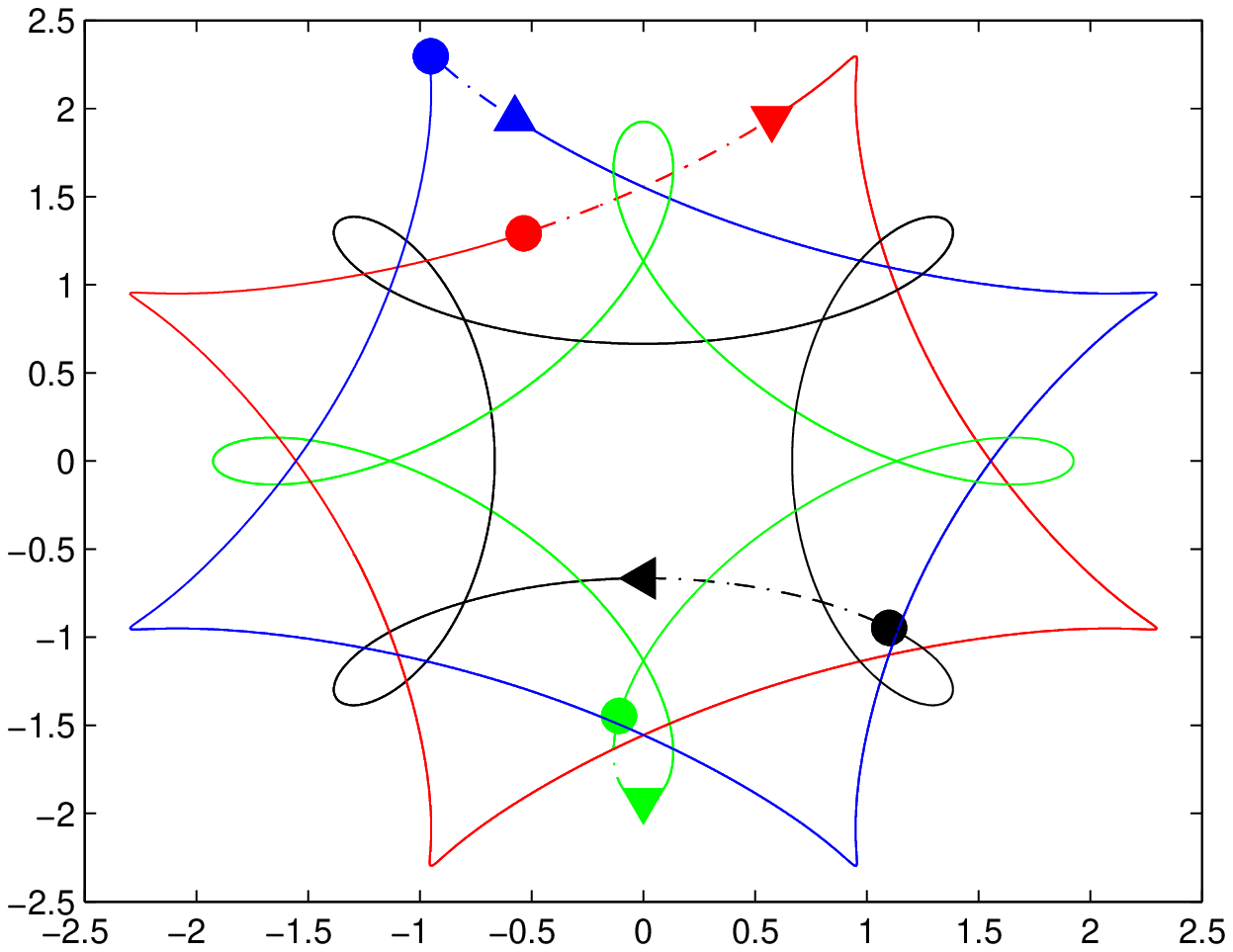}
\caption{\small Non-choreographic Periodic Solutions when  $\theta = \frac{P}{Q}\pi$ with $Q$ even. From left to right: $(\theta,\mu)=(\frac{3\pi}{4},0.8)$;$(\theta,\mu)=(\frac{5\pi}{6},1)$; $(\theta,\mu)=(\frac{7\pi}{8},1.5)$. }
\label{P1}\end{figure}

\begin{figure}
\includegraphics[height=5cm,width=.32\textwidth]{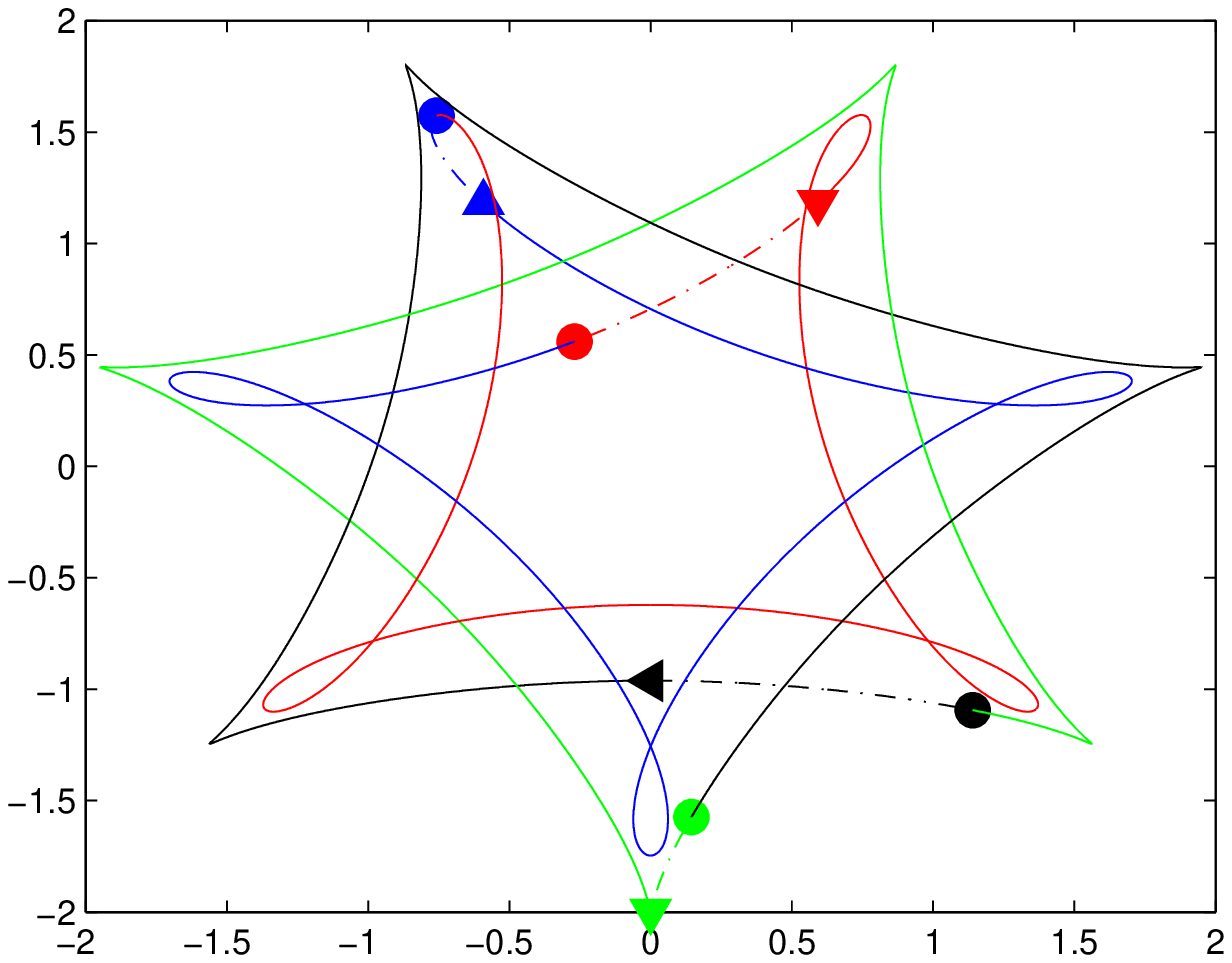}
\includegraphics[height=5cm,width=.32\textwidth]{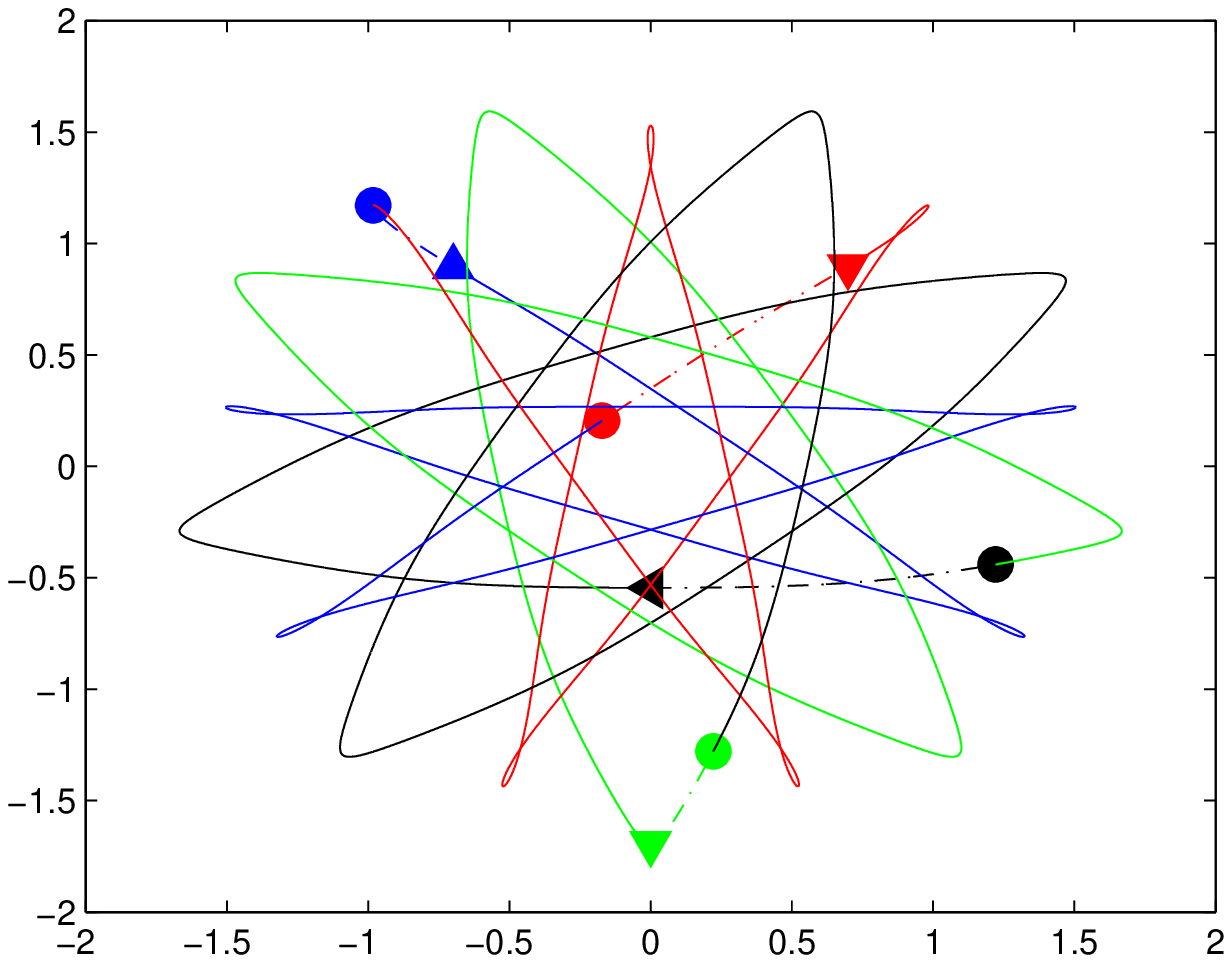}%{Vm09A14f17.eps}
\includegraphics[height=5cm,width=.32\textwidth]{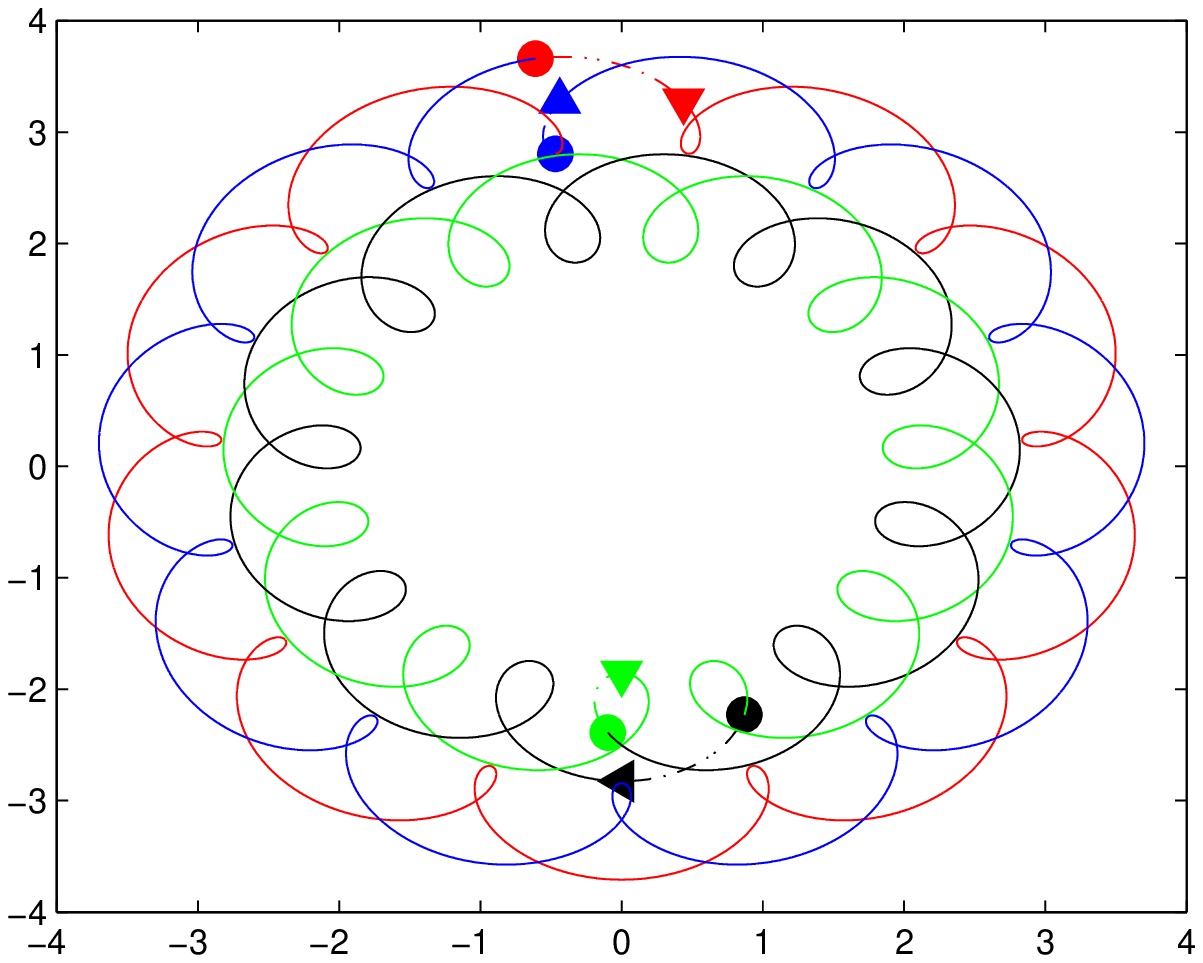}
\caption{\small Double Choreographic Solutions when $\mu\not=1$ and  $\theta = \frac{P}{Q}\pi$ with $Q$ odd. From left to right: $(\theta,\mu)=(\frac{6\pi}{7},0.8)$;$(\theta,\mu)=(\frac{7\pi}{9},0.8)$; $(\theta,\mu)=(\frac{18\pi}{19},1.4)$.}
\label{P21}\end{figure}

\begin{figure}
\includegraphics[height=5cm,width=.32\textwidth]{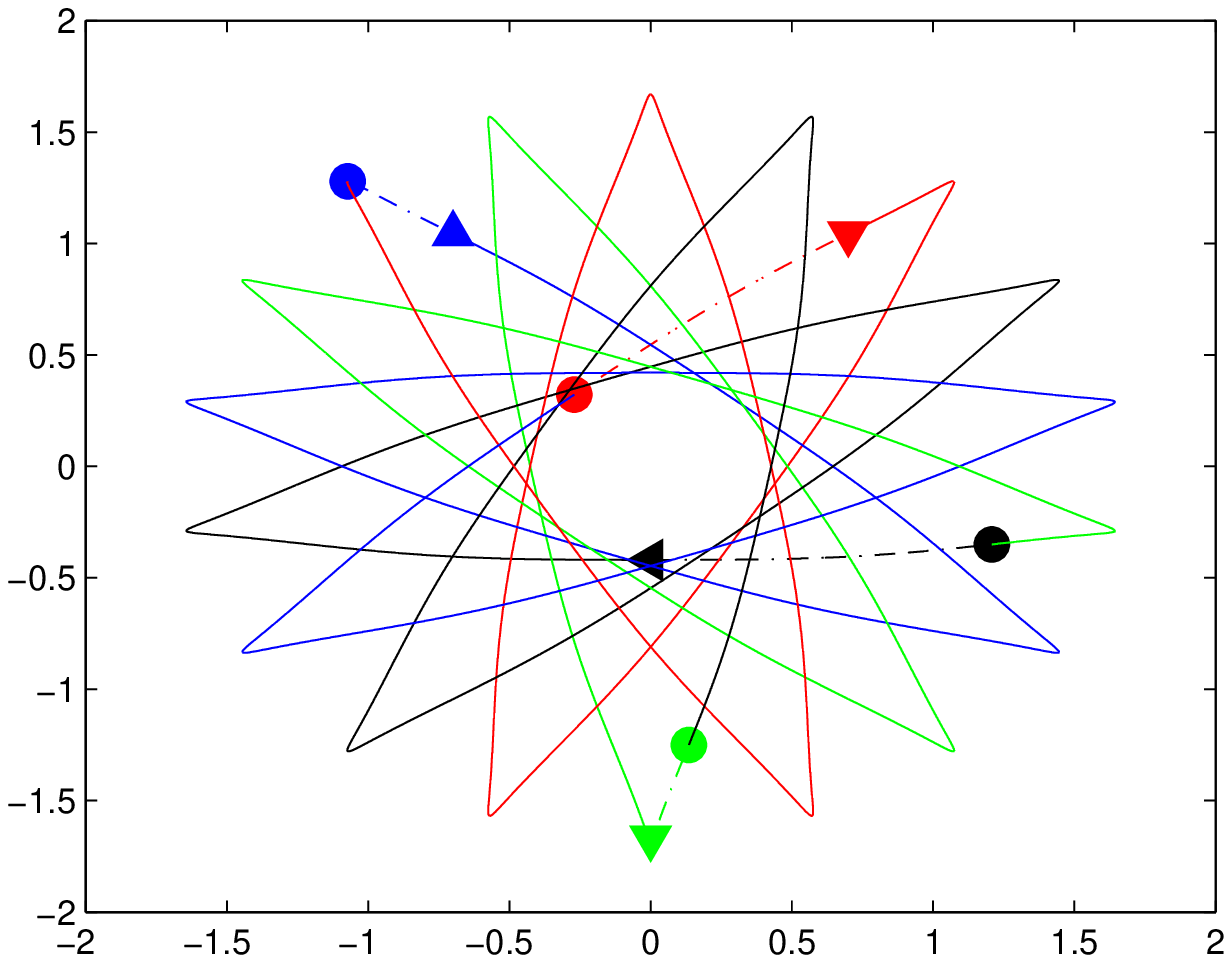}
\includegraphics[height=5cm,width=.32\textwidth]{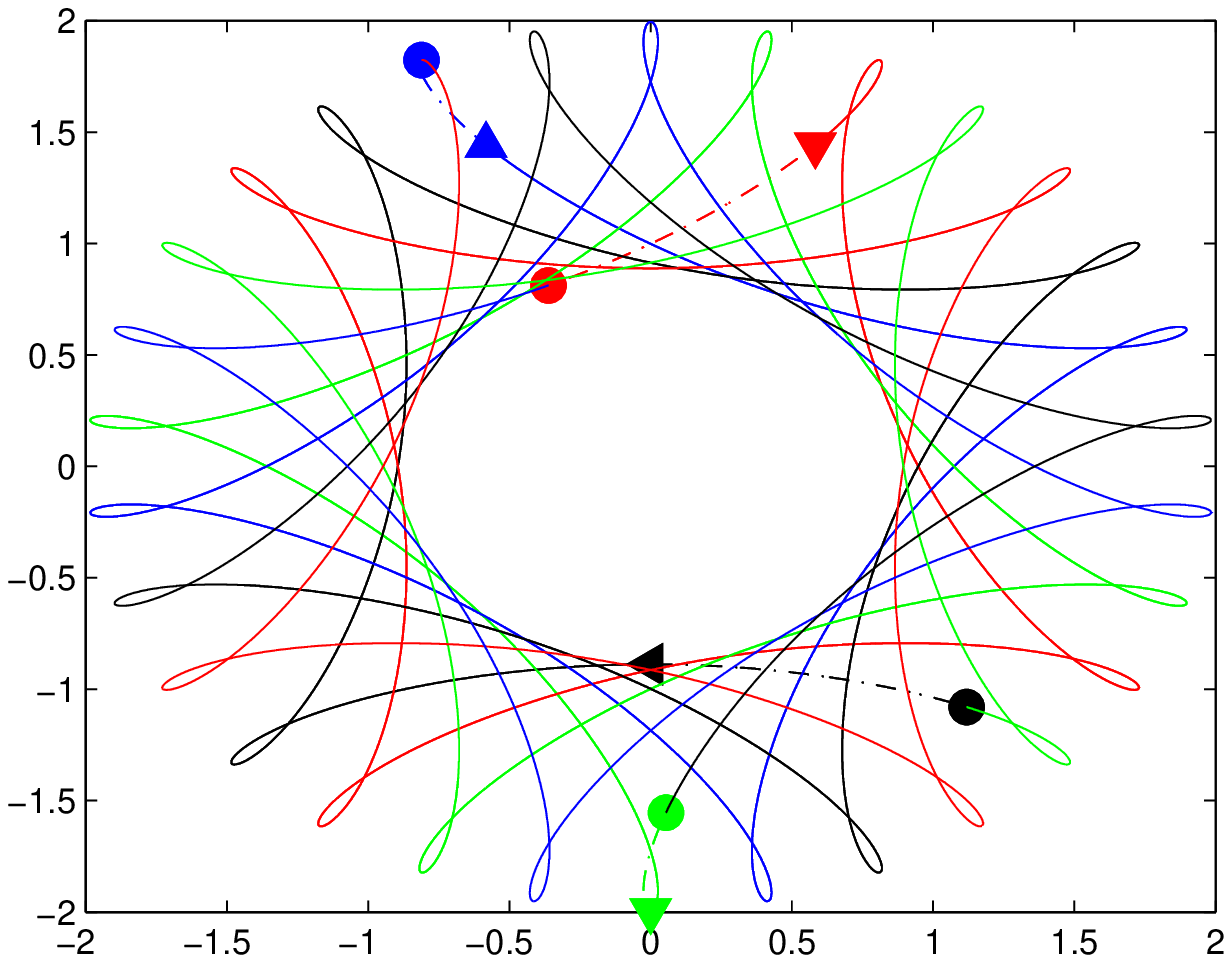}
\includegraphics[height=5cm,width=.32\textwidth]{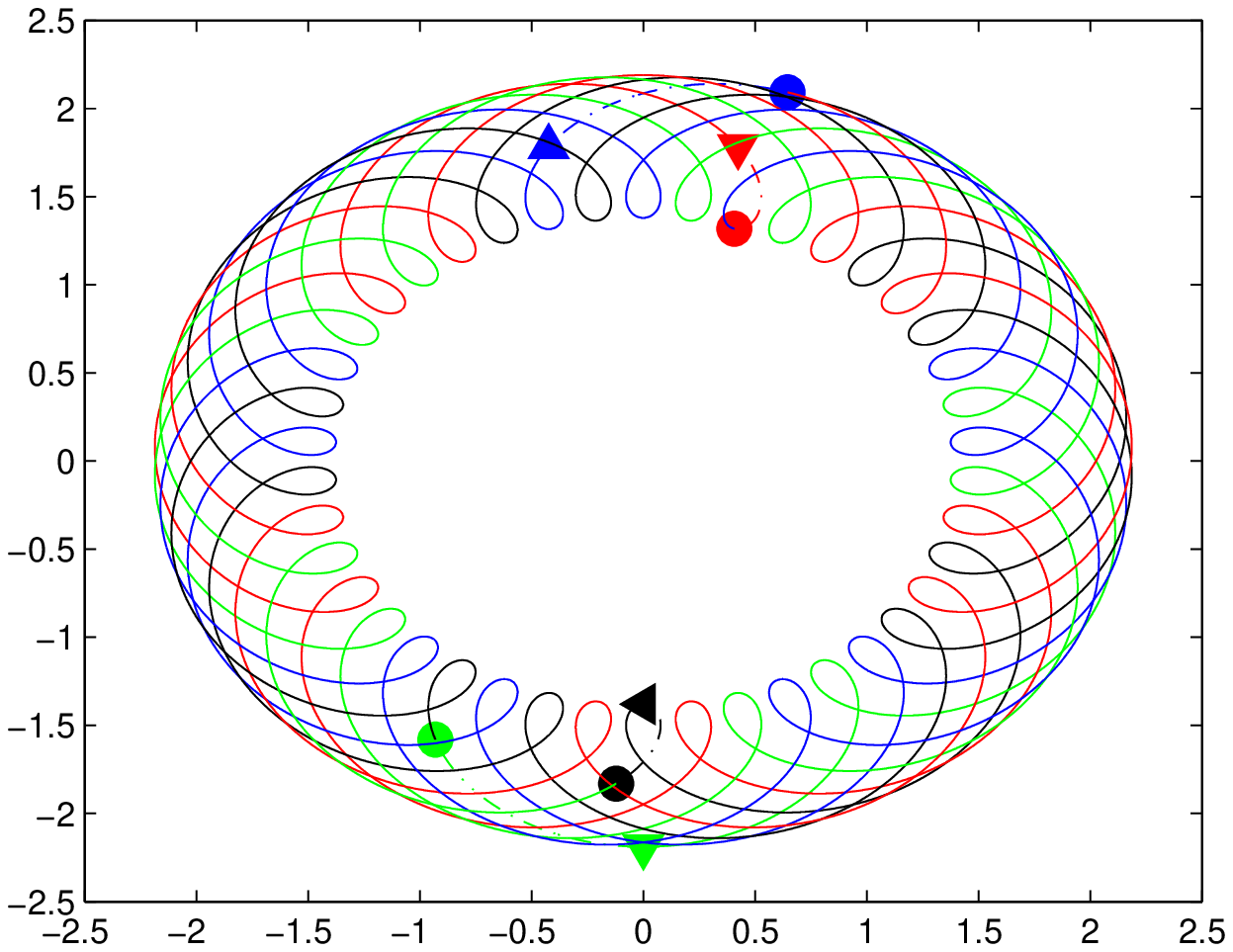}
\caption{\small  Double Periodic Solutions when $\mu=1$ and  $\theta = \frac{P}{Q}\pi$ with both $P$ and $Q$ odd. From left to right: $(\theta,\mu)=(\frac{7\pi}{9},1)$;$(\theta,\mu)=(\frac{13\pi}{15},1)$; $(\theta,\mu)=(\frac{23\pi}{21},1)$. }
\label{P22}\end{figure}

\begin{figure}
\includegraphics[height=5cm,width=.32\textwidth]{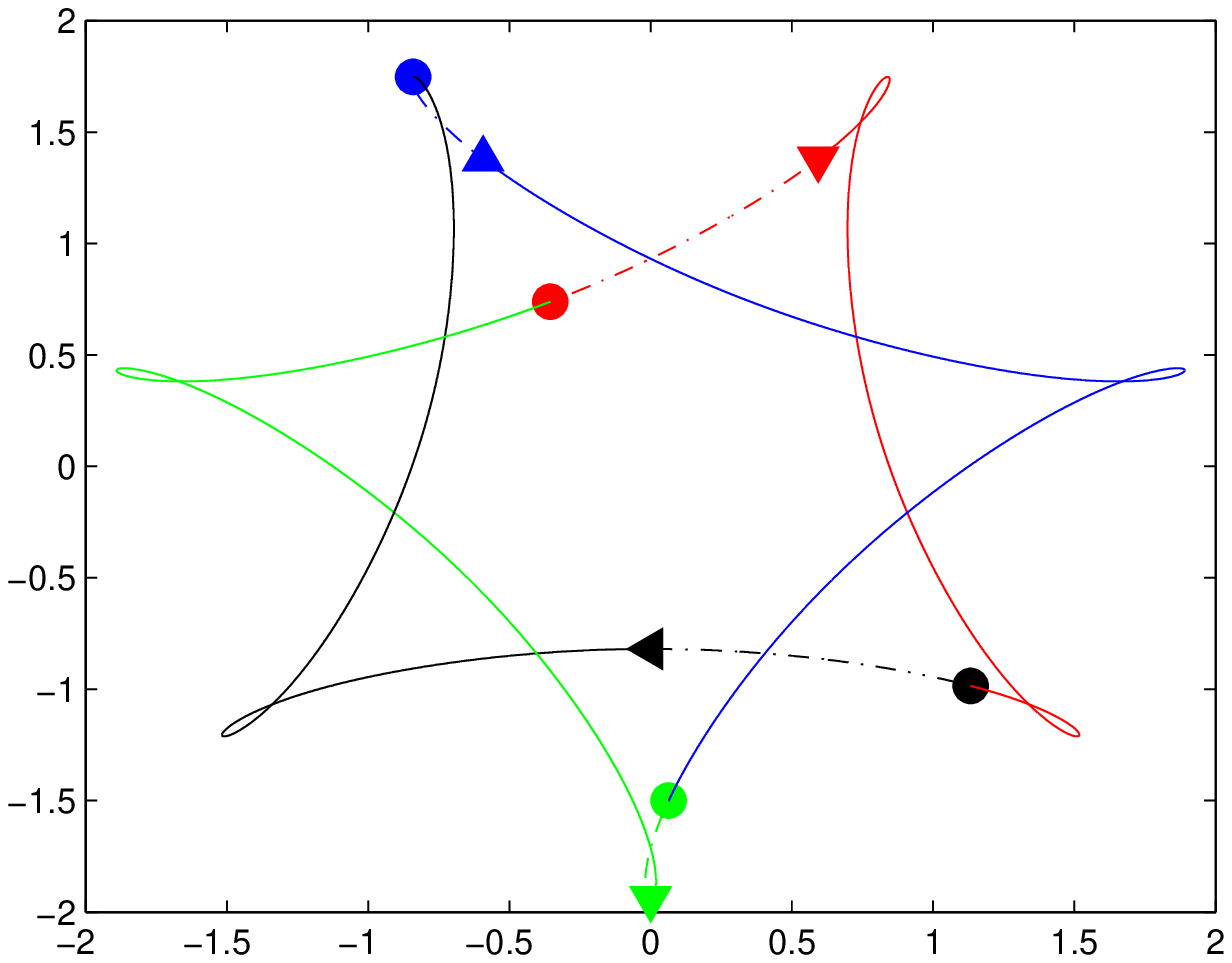}
\includegraphics[height=5cm,width=.32\textwidth]{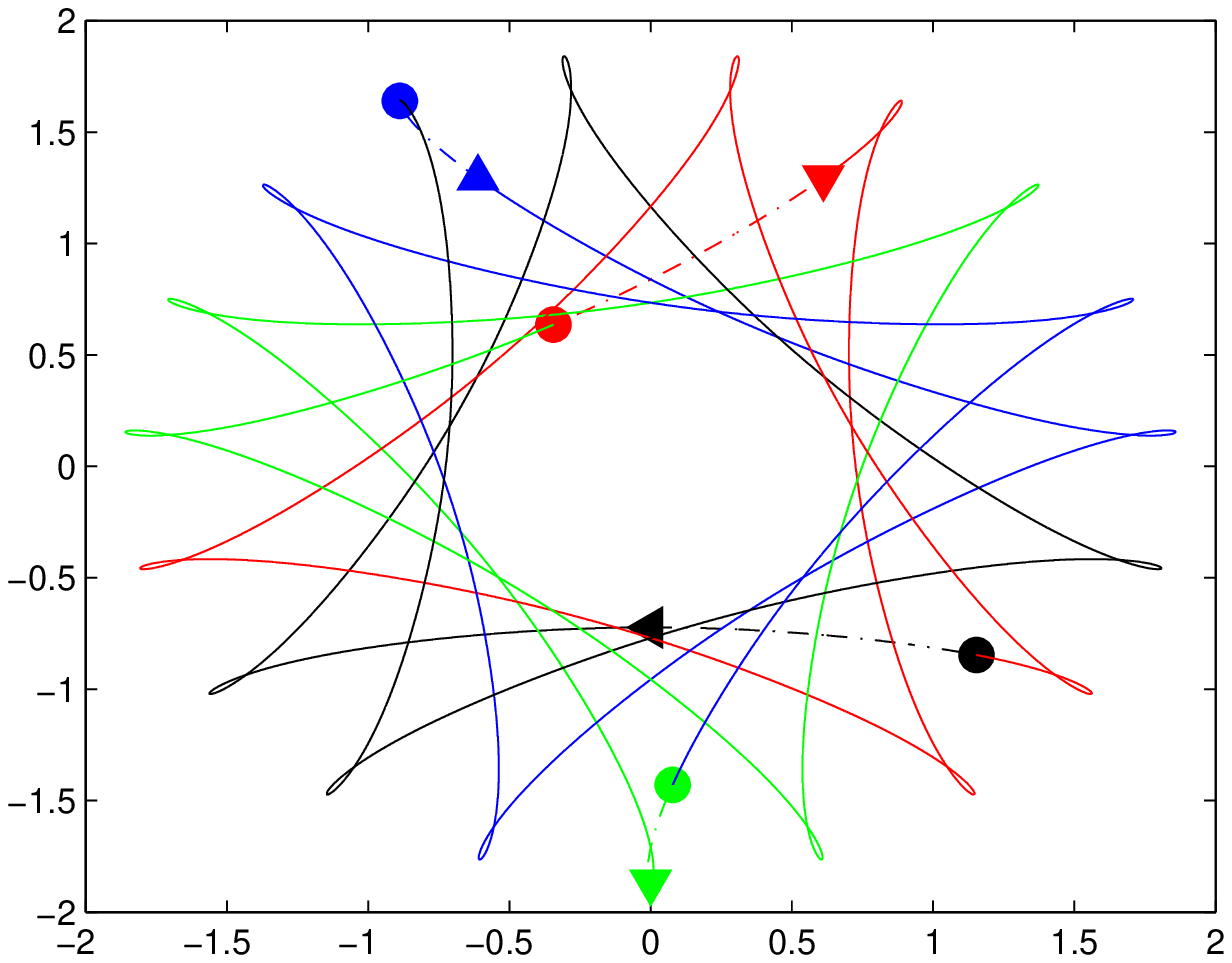}
\includegraphics[height=5cm,width=.32\textwidth]{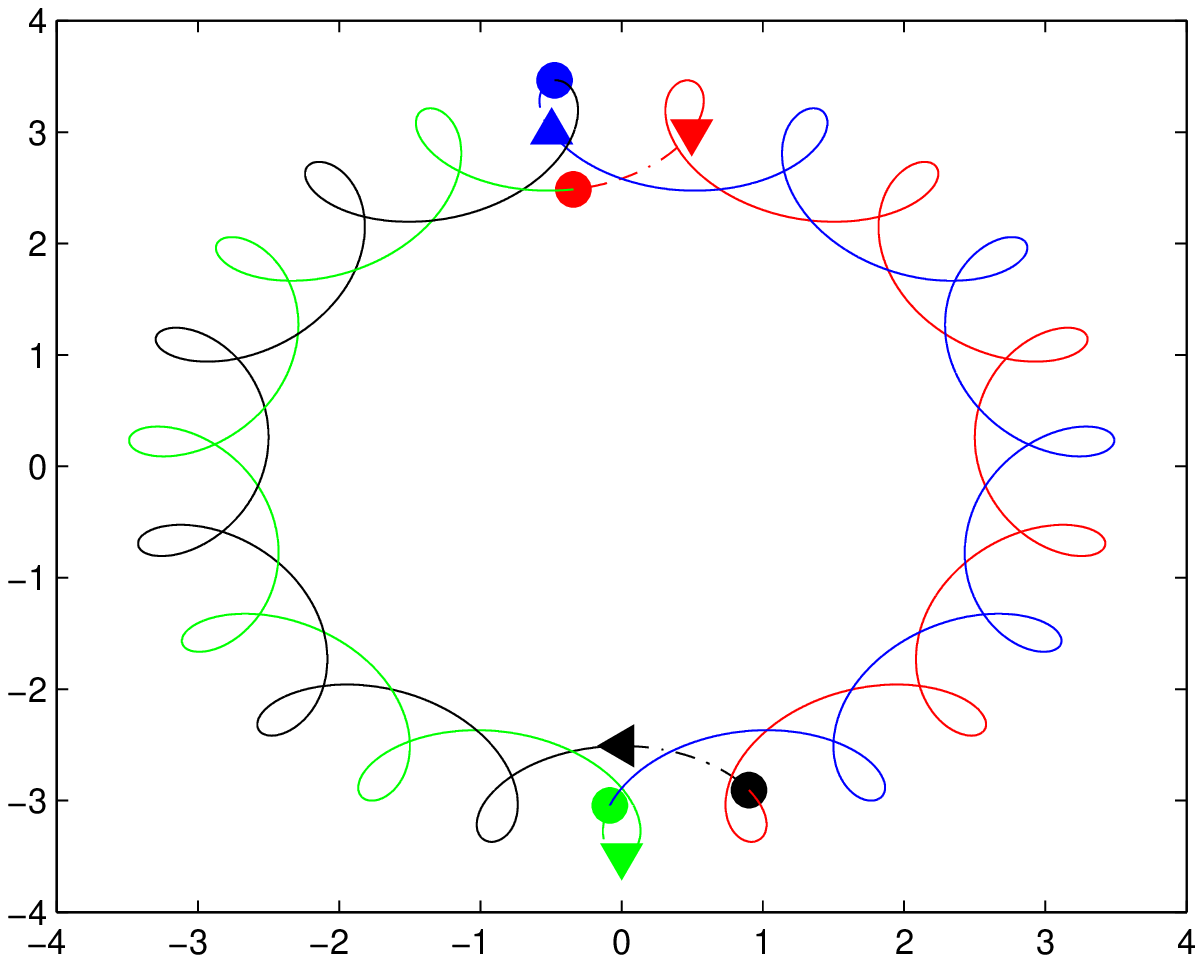}
\caption{\small \small Simple Choreographic Solutions when $\mu=1$ and $\theta = \frac{P}{Q}\pi$ with $P$ even and $\frac{Q-1}{2}$ odd. Bodies chase each other in the order $q_1 (red) $ $\rightarrow q_2 (black)$ $ \rightarrow q_3 (blue)$ $ \rightarrow q_4 (green)$ $ \rightarrow q_1 $. From left to right $\theta= \frac{6\pi}{7},$ $\frac{16\pi}{19},$ and  $\frac{22\pi}{23}$.   }
\label{P23}\end{figure}

\begin{figure}
\includegraphics[height=5cm,width=.32\textwidth]{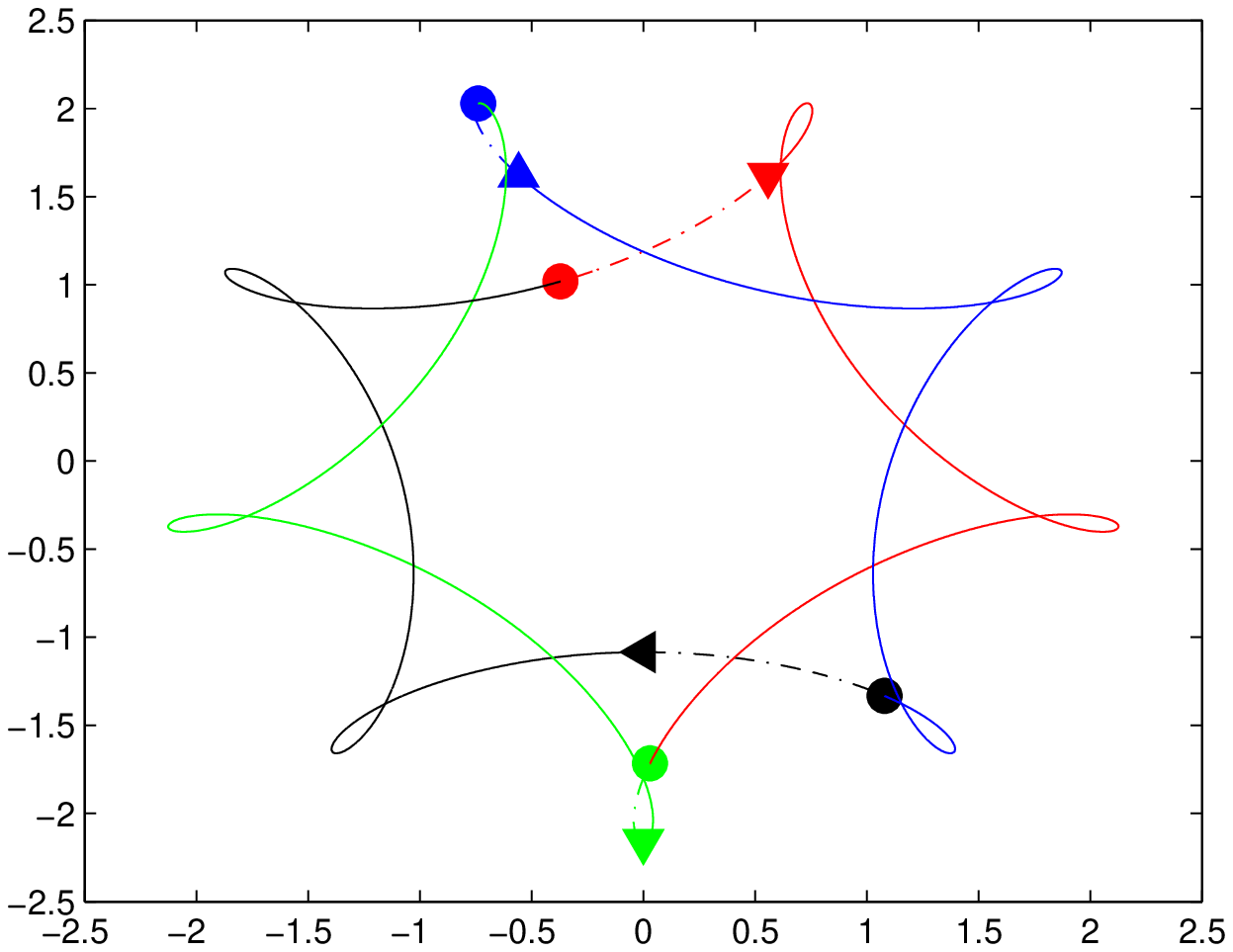}
\includegraphics[height=5cm,width=.32\textwidth]{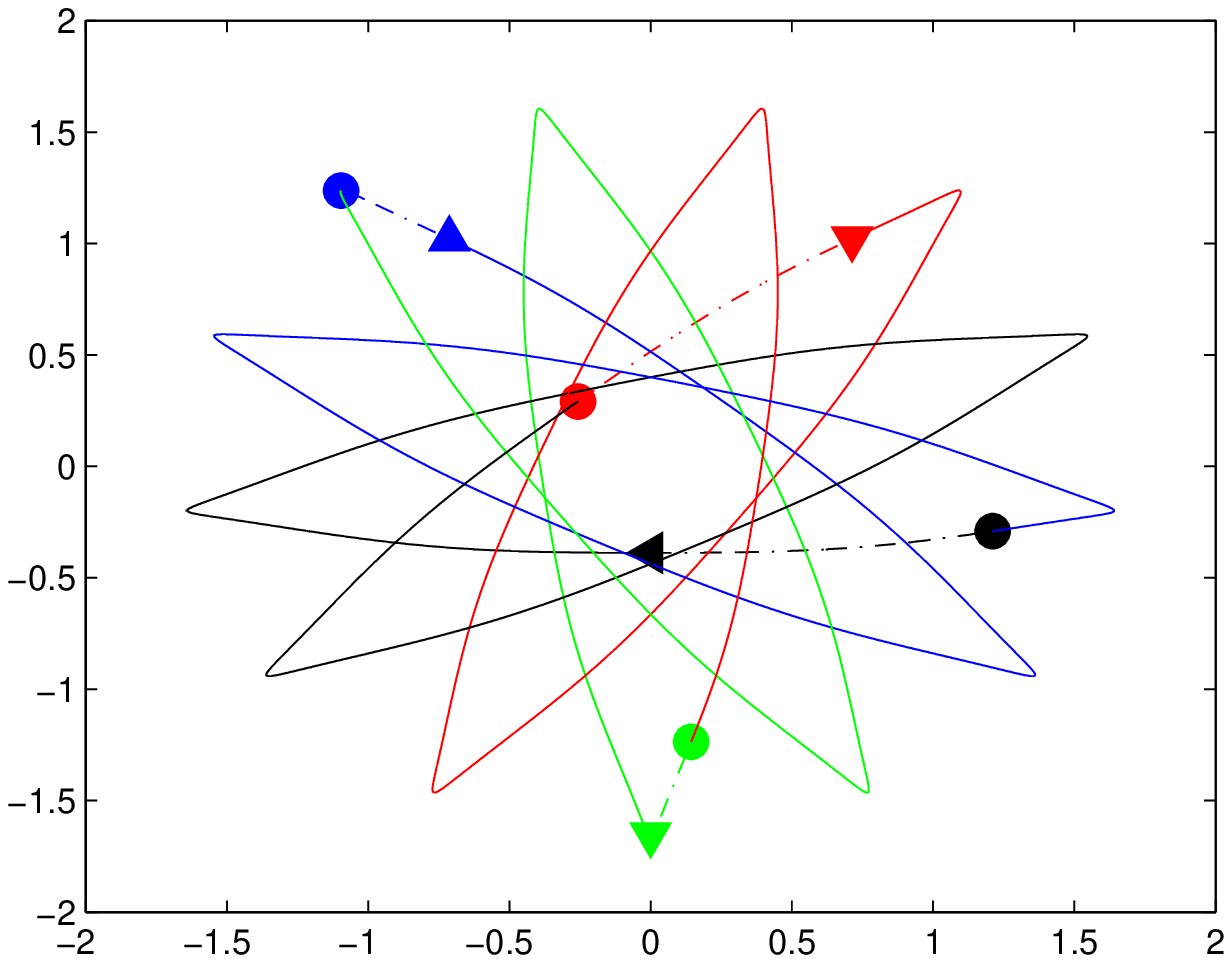}
\includegraphics[height=5cm,width=.32\textwidth]{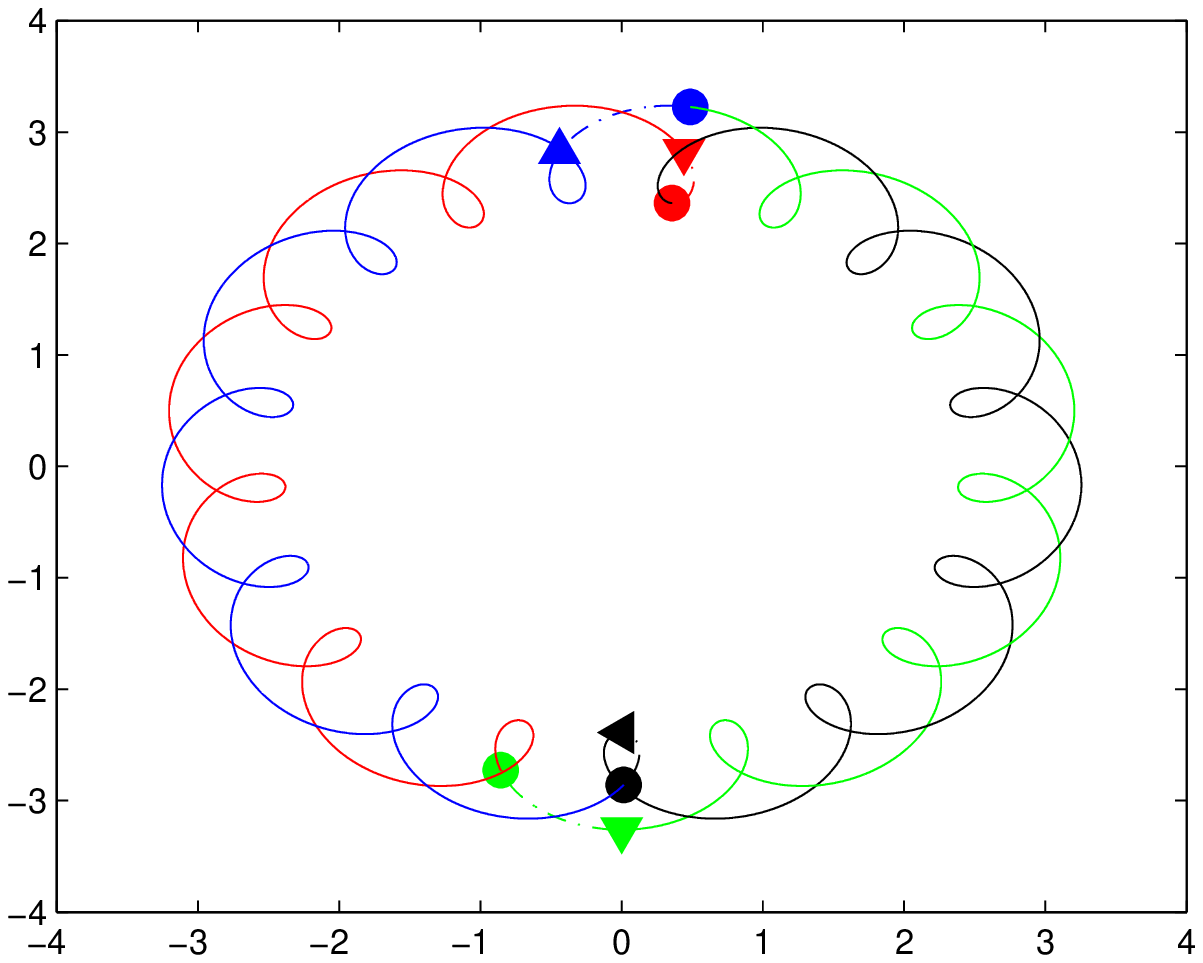}
\caption{\small Simple Choreographic Solutions when $\mu=1$ and  $\theta = \frac{P}{Q}\pi$ with $P$ even and $\frac{Q-1}{2}$ even.  Bodies chase each other in the order $q_1 (red) \rightarrow q_4 (green) \rightarrow q_3 (blue) \rightarrow q_2 (black)\rightarrow q_1$. From left to right $\theta= \frac{8\pi}{9},$ $\frac{10\pi}{13},$ and  $\frac{22\pi}{21}$. }
\label{P24}\end{figure}

\begin{figure}
\includegraphics[height=5cm,width=.32\textwidth]{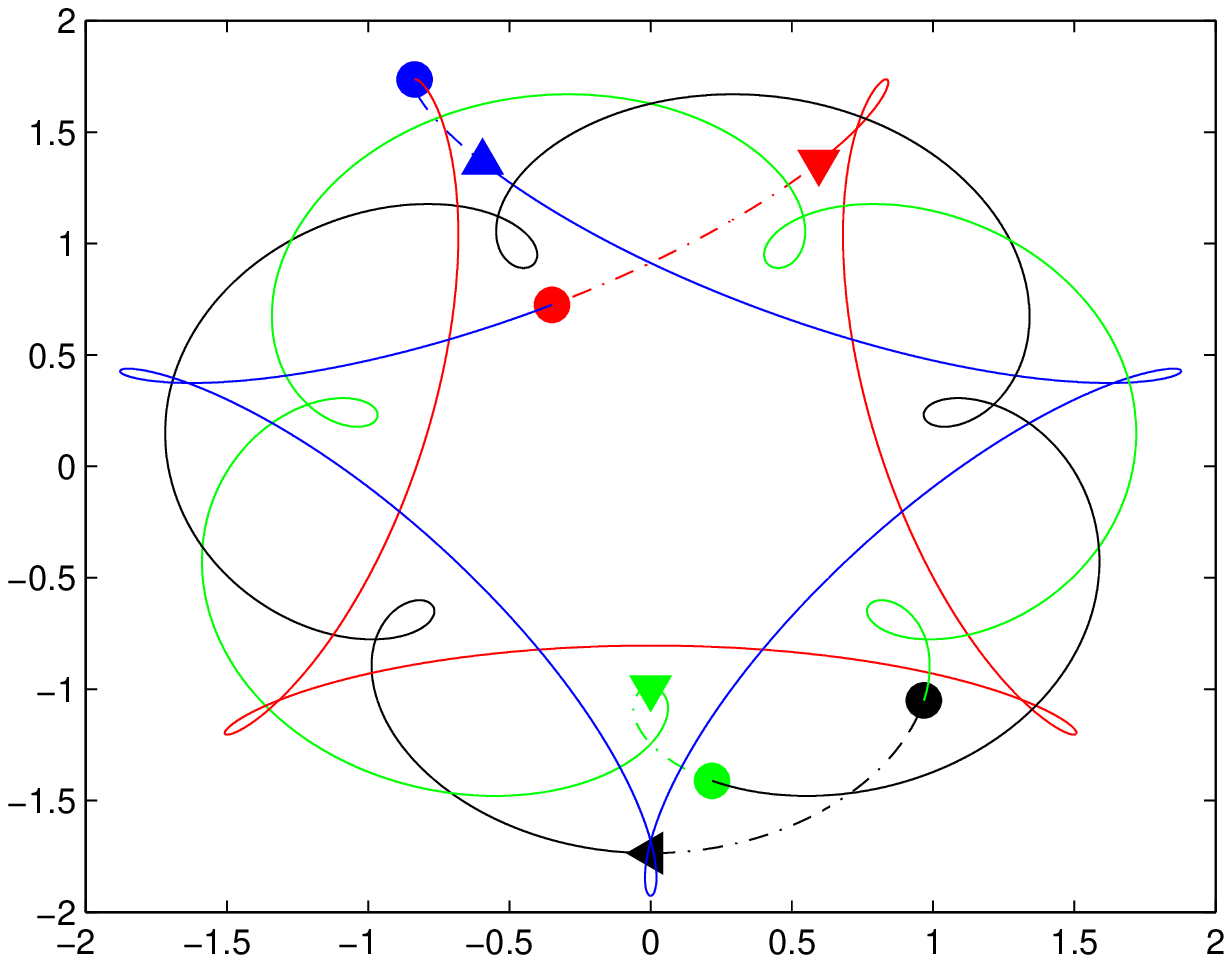}
\includegraphics[height=5cm,width=.32\textwidth]{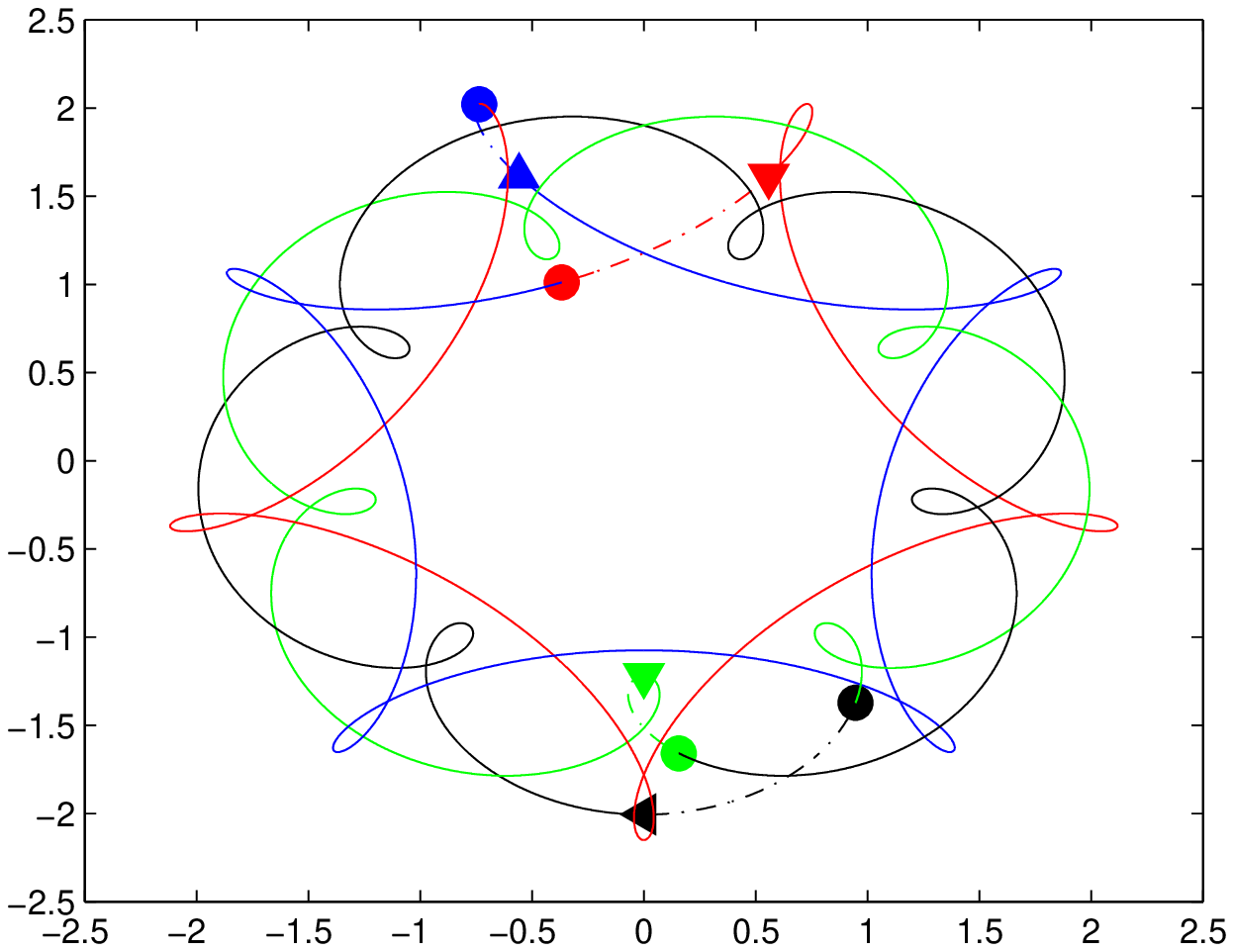}
\includegraphics[height=5cm,width=.32\textwidth]{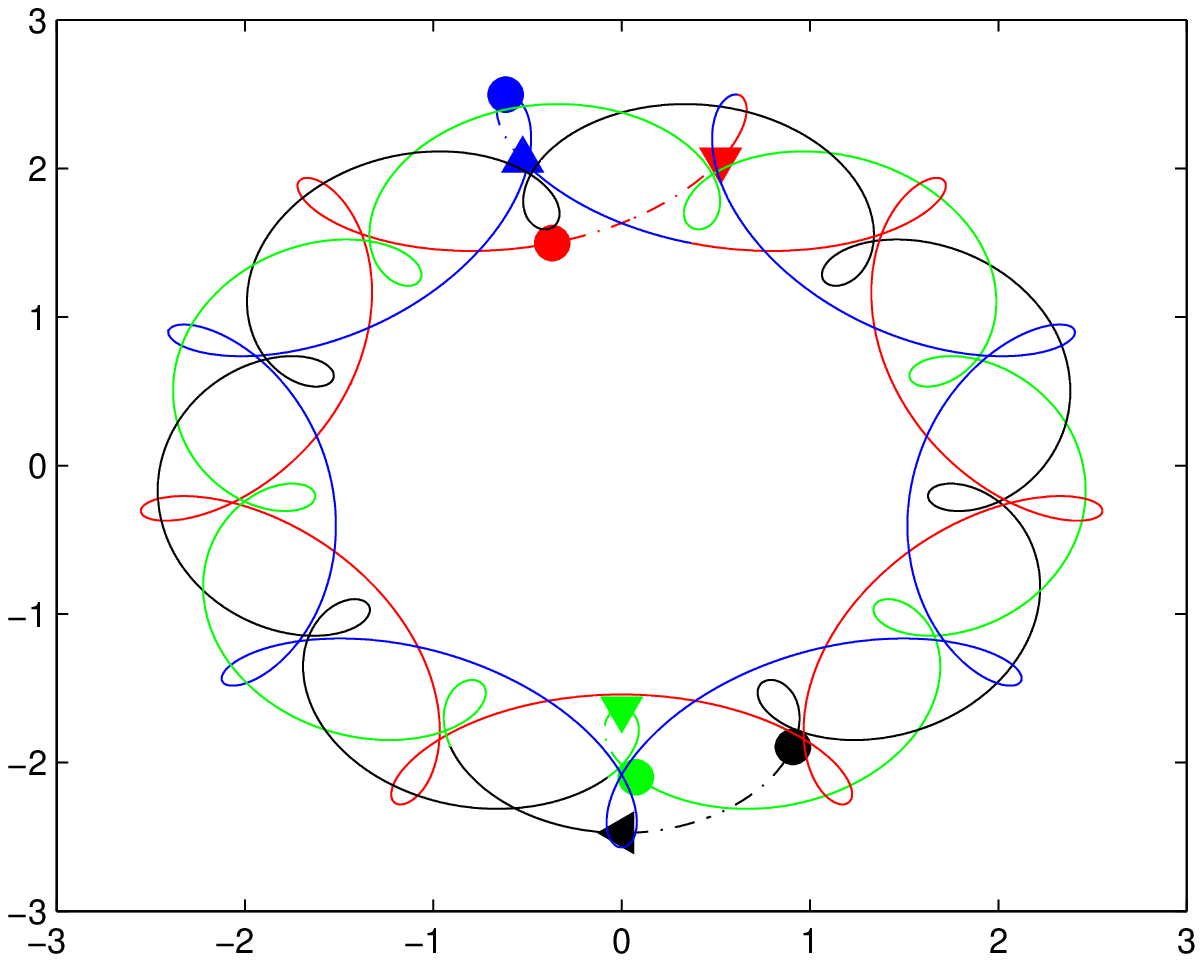} %very stable.
\caption{\small  Double Periodic Solutions when $\mu=1$ and  $\theta = \frac{P}{Q}\pi$ with  $P$ even and $Q$ odd. From left to right: $(\theta,\mu)=(\frac{6\pi}{7},1)$;$(\theta,\mu)=(\frac{8\pi}{9},1)$; $(\theta,\mu)=(\frac{12\pi}{13},1)$. }
\label{case4}\end{figure}

\begin{figure}
\includegraphics[height=5cm,width=.32\textwidth]{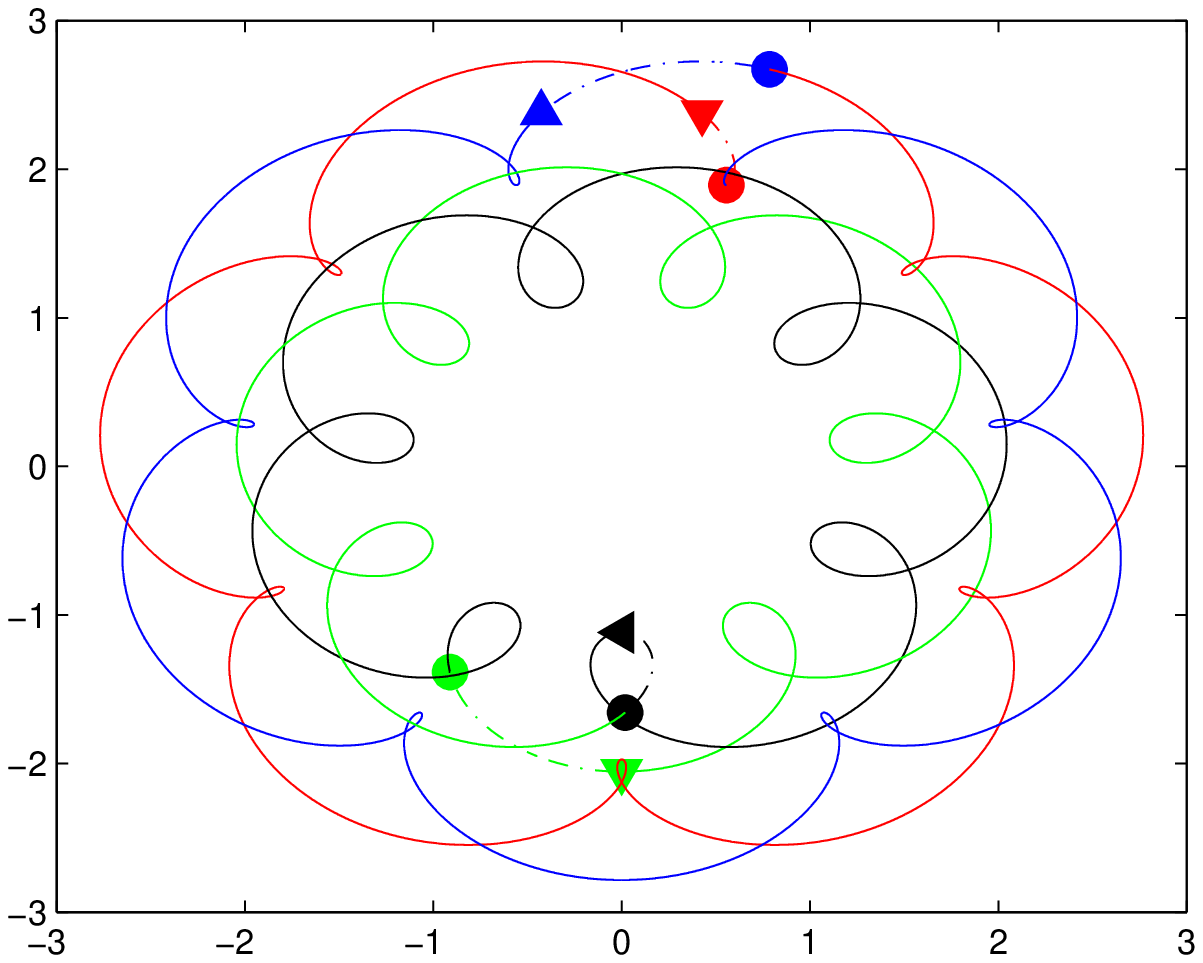}
\includegraphics[height=5cm,width=.32\textwidth]{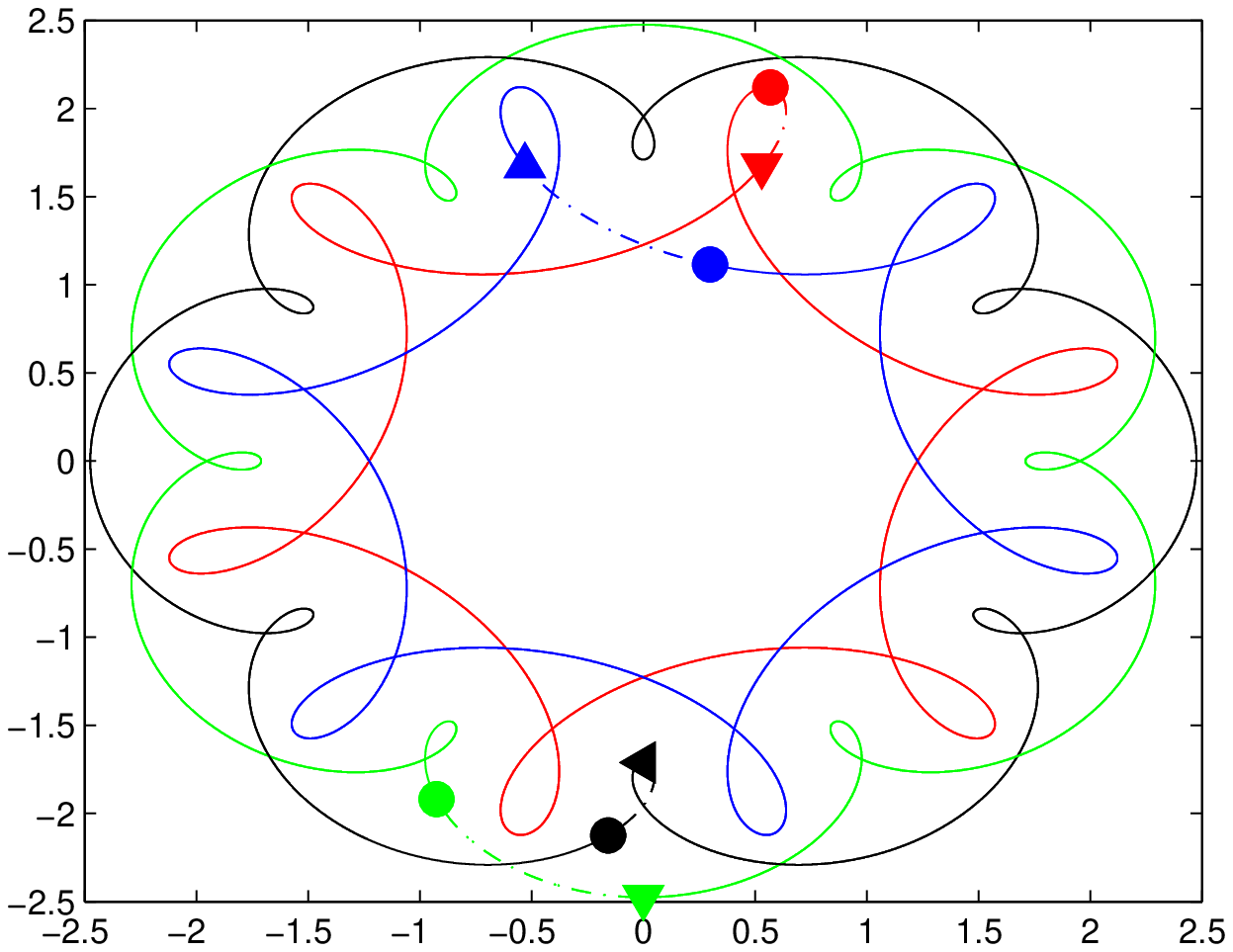}
\includegraphics[height=5cm,width=.32\textwidth]{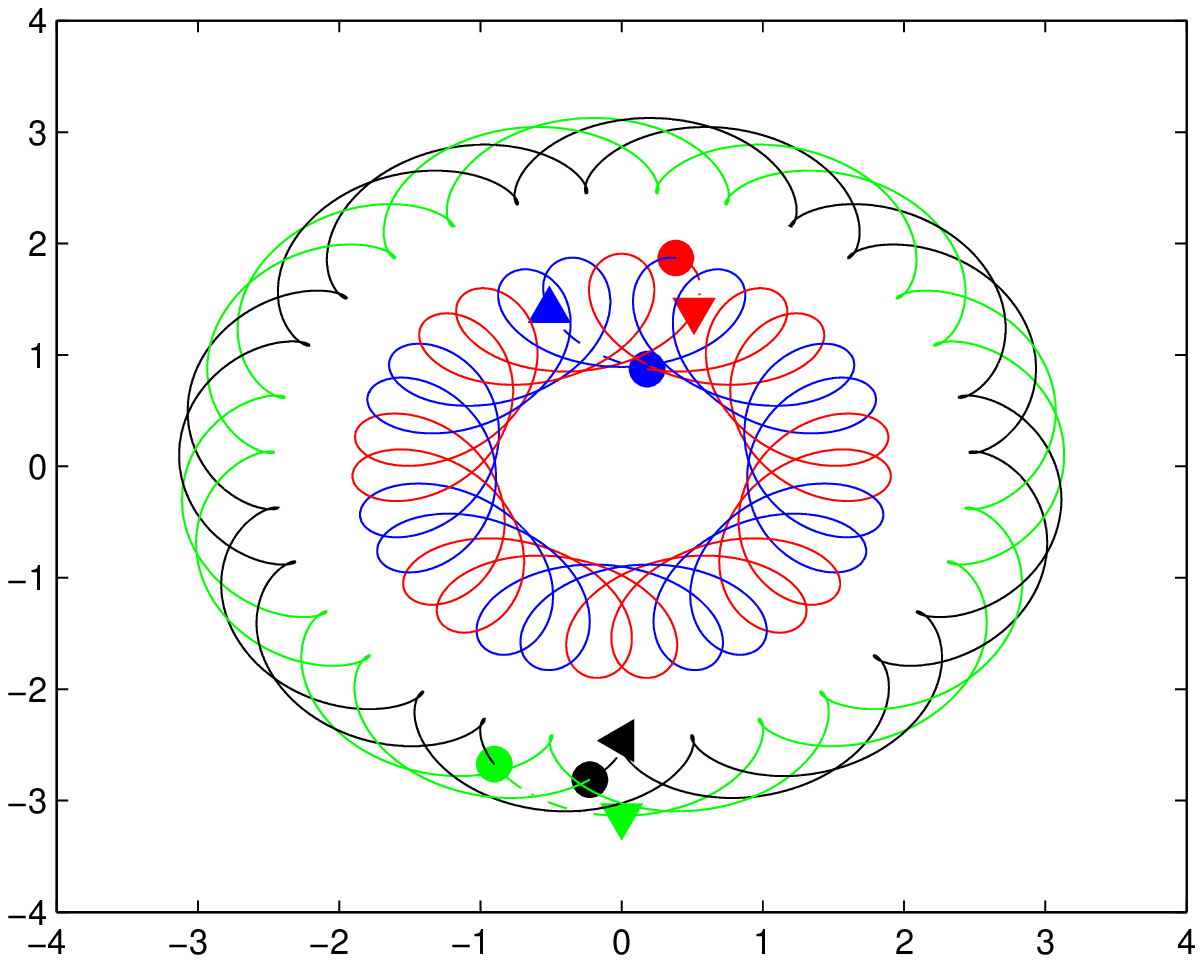}
\caption{\small Periodic Solutions for $\theta>\pi$. From left to right $(\theta, \mu)= (\frac{12\pi}{11}, 1.5)$ ,$(\theta, \mu)= (\frac{13\pi}{12}, 0.8)$,$(\theta, \mu)= (\frac{33\pi}{31}, 0.5)$. }
\label{figA}\end{figure}

\section{Linear Stability of the Periodic Solutions }\label{sec4}
%\begin{proof}[The proof of theorem \ref{Thm:LS}]
%\end{proof}
In this section, we provide a rigorous computation to study the linear stability of the periodic solutions. The periodic solutions generated by the local minimizers will be regarded as a $\mathcal{T}$-periodic solutions to a Hamiltonian system. The linear stability of the periodic solutions will be determined by the eigenvalues of their corresponding monodromy matrix.
%In theory all calculations could be done by paper and pencil, but in practice the number of operations, even if they are very basic and trivial, exceeds human resources.
Suppose that $\gamma(t)$ is a $\mathcal{T}$-periodic solution to the Hamiltonian system $\dot{\gamma}=J\nabla H(\gamma)$, where $J=\left[\begin{array}{ll} 0 & I\\ -I & 0 \end{array}\right]$ is the standard symplectic matrix and $I$ is the appropriately sized identity matrix.  Let $X(t)$ be the fundamental matrix solution to
$$\dot{\xi}=JD^2H(\gamma(t))\xi, \hspace{0.2cm} \xi(0)=I.$$
$X(t)$ is symplectic and satisfies $X(t+\mathcal{T})=X(t)X(\mathcal{T})$ for all $t$. The matrix $X(\mathcal{T})$ is called the monodromy matrix whose eigenvalues, the characteristic multipliers, determine the linear stability of the periodic solution. Since every integral in the $n$-body problem yields a multiplier of $+1$, there are eight $+1$ multipliers for a periodic orbit in the planar problem. It is natural to define the linear stability of a periodic solution by examining stabiltiy on the reduced quotient space.
\begin{definition}
A periodic solution of the planar $n$-body problem has eight trivial
characteristic multipliers of $+1$. The solution is spectrally stable if the remaining
multipliers lie on the unit circle and linearly stable if, in addition, the monodromy matrix
$X(T)$ restricted to the reduced space is diagonalizable.
\end{definition}

Here we apply standard symplectic transforms to reduce Hamiltonian system to a 10 dimension Hamiltonian system. Such reductions have been constructed in \cite{OuXie} and we include it here for the sake of completeness. The monodromy matrix of the periodic solution $\gamma(t)$ in the reduced system has a pair of $+1$ eigenvalues and the remaining eight eigvalues must be on the unit circle if the solution is linearly stable.

To eliminate the trivial $+1$ multipliers of a periodic solution, we use Jacobi coordinates and symplectic polar coordinates (see chapter 7 in \cite{MHO2}). Denote $p_i=m_i\dot{q}_i$ as the momentum coordinates and let $\mu_i=\sum_{j=1}^i m_j$ and $M_i=\frac{m_i\mu_{i-1}}{\mu_i}$. Then let
$$\begin{array}{ll}
g_4={\frac {{ m_4}\,{ q_4}+{ m_3}\,{ q_3}+{ m_2}\,{ q_2}+{
m_1}\,{ q_1}}{{ m_1}+{ m_2}+{ m_3}+{ m_4}}}, & G_4={ p_4}+{ p_3}+{ p_2}+{ p_1};\\
u_2=q_2 - q_1, & v_2={\frac {{ \mu_1}\,{ p_2}}{{ \mu_2}}}-{\frac {{ m_2}\,{ p_1}}{{
 \mu_2}}};\\
u_3={ q_3}-{\frac {{ m_2}\,{ q_2}+{ m_1}\,{ q_1}}{{ m_1}+{
m_2}}}, & v_3={\frac {{ \mu_2}\,{ p_3}}{{ \mu_3}}}-{\frac {{ m_3}\, \left( {
 p_2}+{ p_1} \right) }{{ \mu_3}}};\\
u_4={ q_4}-{\frac {{ m_3}\,{ q_3}+{ m_2}\,{ q_2}+{ m_1}\,{
q_1}}{{ m_1}+{ m_2}+{ m_3}}}, & v_4={\frac {{ \mu_3}\,{ p_4}}{{ \mu_4}}}-{\frac {{ m_4}\, \left( {
 p_3}+{ p_2}+{ p_1} \right) }{{ \mu_4}}}.
\end{array}$$
The new Hamiltonian is $$H_2(u_2,u_3,u_4,v_2,v_3,v_4)=\,{\frac {{{ v_2}}^{2}}{{2 M_2}}}+\,{\frac {{{ v_3}}^{2}}{{
2 M_3}}}+\,{\frac {{{ v_4}}^{2}}{{ 2 M_4}}}-U_2.$$
%where
% $$ U_2={\frac {{ m_1}\,{ m_2}}{ \left| { u_2} \right| }}+{ m_1}\,{ m_3} \left(  \left| {\frac {{ m_2}\,{ u_2}}{{ \mu_2}}}+{ u_3} \right|  \right) ^{-1}+{ m_1}\,{ m_4} \left(  \left| { u_4}+{\frac {{ m_2}\,{ u_2}}{{ \mu_2}}}+{\frac {{ m_3}\,{ u_3}}{{ \mu_3}}} \right|  \right) ^{-1}+$$ $${ m_2}\,{ m_3} \left(  \left| -{ u_3}+{\frac {{ m_1}\,{ u_2}}{{ \mu_2}}} \right|  \right) ^{-1}+{ m_2}\,{ m_4} \left(  \left| -{ u_4}+{\frac {{ m_1}\,{ u_2}}{{ \mu_2}}}-{\frac {{ m_3}\,{ u_3}}{{ \mu_3}}} \right| \right) ^{-1}+$$ $${ m_3}\,{ m_4} \left(  \left| -{ u_4}+{\frac {{ \mu_2}\,{ u_3}}{{ \mu_3}}} \right|  \right) ^{-1}.$$
$U_2$ is the corresponding potential energy in the new coordinates and similarly $U_3, U_4$ in the below are the potential energy in the different cooordinates.The new Hamiltonian is independent of $g_4$ and $G_4$, the center of mass and total linear momentum respectively. This reduces the dimension by four from 16 to 12.\\
Next we change to symplectic polar coordinates to eliminate the integrals due to the angular momentum and rotational symmetry. Set
$$u_i=(r_i\cos(\theta_i), r_i\sin(\theta_i))$$ %\hspace{0.3cm} \hbox{ and } \hspace{0.3cm}
$$v_i=(R_i\cos(\theta_i)-\frac{\Theta_i}{r_i}\sin(\theta_i), R_i\sin(\theta_i)+\frac{\Theta_i}{r_i}\cos(\theta_i))$$
for $i=2,3,4$. Then the new Hamiltonian  becomes
$$H_3=\,{\frac {{{ R_2}}^{2}{{ r_2}}^{2}+{{ \Theta_2}}^{2}}{{2 M_2}
\,{{ r_2}}^{2}}}+\,{\frac {{{ R_3}}^{2}{{ r_3}}^{2}+{{
\Theta_3}}^{2}}{{ 2M_3}\,{{ r_3}}^{2}}}+\,{\frac {{{ R_4}}^{2}{{
 r_4}}^{2}+{{ \Theta_4}}^{2}}{{ 2M_4}\,{{ r_4}}^{2}}}-U_3.$$
Note that the Hamiltonian $H_3$ has only terms of difference angles. This suggests making a final symplectic change of coordinates by leaving the radial variables alone. Use the generating function $ S={ \Theta_2}\,{ x_2}+{ \Theta_3}\, \left( {x_3}+{ x_2} \right) +{ \Theta_4}\, \left( { x_4}+{ x_3}+{ x_2} \right)$, and so
 $$\theta_2 = x_2, \theta_3 = x_3+x_2, \theta_4 = x_4+x_3+x_2;$$
 $$ \Theta_2 = X_2-X_3, \Theta_3 = X_3-X_4; \Theta_4 = X_4.$$
 The new Hamiltonian will be independent of $x_2$ which means that $X_2=\Theta_2+\Theta_3+\Theta_4$ (total angular momentum) is an integral, and $x_2$ is an ignorable variable. Setting $X_2=c$ and plugging into the Hamiltonian $H_3$ yields
 $$H_4= \,{\frac {{{ R_2}}^{2}{{ r_2}}^{2}+ \left( {c}-{ X_3}
 \right) ^{2}}{{ 2M_2}\,{{ r_2}}^{2}}}+\,{\frac {{{ R_3}}^{2}{
{ r_3}}^{2}+ \left( { X_3}-{ X_4} \right) ^{2}}{{ 2M_3}\,{{
r_3}}^{2}}}+\,{\frac {{{ R_4}}^{2}{{ r_4}}^{2}+{{ X_4}}^{2}}{{
 2M_4}\,{{ r_4}}^{2}}}-U_4.$$
 This reduces the system to 10 dimensions, with the variables $z=(r_2, r_3, r_4, x_3, x_4, R_2, R_3, R_4,$ $ X_3, X_4)$. \\
 Because $H_4$ is a Hamiltonian system, the monodromy matrix $X(\mathcal{T})$ is symplectic. Its periodic solution $\gamma(t)$ will generate an eigenvector of $X(\mathcal{T})$. In fact, $\gamma(t)$ is a solution of $\dot{z}=J\nabla H_4(z)$ with initial condition $z(0)=\gamma(0)$. Then $\ddot{\gamma}(t)=JD^2H_4(\gamma(t))\dot{\gamma}(t)$. This implies that $\dot{\gamma}(t)$ satisfies the associated linear system
$$\dot{\xi}=JD^2H_4(\gamma(t))\xi,\hspace{0.2cm} \xi(0)=\dot{\gamma}(0).$$
Since $X(t)$ is the fundamental solution of the above linear system, $\dot{\gamma}(t)=X(t)\dot{\gamma}(0)$, which implies $X(\mathcal{T})\dot{\gamma}(0)=\dot{\gamma}(\mathcal{T})=\dot{\gamma}(0)$. Because $X(\mathcal{T})$ is symplectic, $J^{-1}X(\mathcal{T})J= X(\mathcal{T})$. Then $X(\mathcal{T})J\dot{\gamma}(0)=J\dot{\gamma}(0)$. So the Monodromy matrix has two $+1$ multipliers, leaving the remaining eight eigenvalues to determine the linear stability of the periodic solution. Because the eigenvalues of a symplectic matrix occur in quadruples $(\lambda,$ $ \lambda^{-1}, $ $\bar{\lambda}$, $\bar{\lambda}^{-1})$, we have the following lemma. %So we have proved the following results \\ %{\bf \it In the polar coordinates, the angle $\theta$ is periodic in the quotient space $R/[-\pi/2,3\pi/2]$. In time period $\mathcal{T}$, there are jumping values. } \\
%\begin{lemma}
%Suppose that $\gamma(t)$ is a $\mathcal{T}$-periodic solution of the Hamiltonian system of the 4-body problem. Let $X(\mathcal{T})$ is the monodromy matrix of the periodic solution $\gamma(t)$  in the reduced Hamiltonian system $H_4$. Then the periodic solution $\gamma(t)$ is linearly stable if the remaining eight eigenvalues of the 10 dimensional symplectic matrix $X(\mathcal{T})$ are on unit circle.
%\end{lemma}
%\begin{lemma}
%Let $v_1=\dot{\gamma}(0)/|\dot{\gamma}(0)|$. Let $v_1, v_2,v_3,v_4,v_5$ be an orthonormal set of vectors such that $Y_0=[Jv_1,Jv_2,\cdots,Jv_5,v_1,v_2,\cdots, v_5]$ is an orthogonal, symplectic matrix. Then $X(\mathcal{T})=Y(\mathcal{T})Y_0^{-1}$ is similar to $W=Y_0^{-1}Y(\mathcal{T})$. $Y_0^{-1}v_1=e_6$ and $We_6=e_6$. $Y_0^{-1}Jv_1=e_1$ and $We_1=e_1$. Let $D$ be the matrix obtained from $W$ by deleting the first and sixth rows and columns. Then the linear stability of periodic orbits for planar-four body problem is determined by the eight eigenvalues of $D$.
%\end{lemma}
\begin{lemma}
Let $X$ be a symplectic matrix and $W=\frac{1}{2}(X+X^{-1})$. Then the eigenvalues of $X$ are all on the unit circle if and only if all of the eigenvalues of $W$ are real and in $[-1,1]$.
\end{lemma}
\begin{proof}
The lemma and its proof are similar to Lemma 4.1 in Roberts' paper \cite{RG}. We prove it here for the sake of completeness. Suppose that $\vec{v}$ is an eigenvector of the symplectic matrix $X$ with eigenvalue $\lambda$, i.e. $X\vec{v}=\lambda \vec{v}$. Then $X^{-1}\vec{v}=\lambda^{-1}\vec{v}$.  $W\vec{v}=\frac{1}{2}(X+X^{-1})\vec{v}=\frac{1}{2}(\lambda+\lambda^{-1})\vec{v}$ from which it follows that $\frac{1}{2}(\lambda+\lambda^{-1})$ is an eigenvalue of $W$. The map $f : \mathcal{C}\mapsto\mathcal{C}$ given by $f(\lambda) = \frac{1}{2}(\lambda+\lambda^{-1})$ takes the unit circle onto the
real interval $[-1, 1]$ while mapping the exterior of the unit disk homeomorphically onto
$\mathcal{C}\backslash[-1, 1]$. The lemma follows this assertion immediately.
\end{proof}
Because the eigenvalue pairs $\lambda$ and $\lambda^{-1}$ of $X$ are mapped to the same eigenvalue $\frac{1}{2}(\lambda+\lambda^{-1})$ of $W$, the multiplicity of eigenvalues of $W$ must be at least two. The two $+1$ multipliers is still mapped to $+1$ with multiplicity two.  The remaining eight non-one eigenvalues on the unit circle of $X$ for linear stable periodic solution have been mapped to four pairs of real eigenvalues in $(-1,1)$.

%\subsection{Numerical Calculations}
Numerically, a MATLAB program was written using a Runge-Kutta-Fehlberg method with local truncation error of
order four to compute the monodromy matrix $X(\mathcal{T})$ of the  reduced linearized Hamiltonian $H_4$ for a periodic solution of planar 4-body problem.  Then we compute $W=\frac{1}{2}(X+X^{-1})$ and  its eigenvalues. In order to conclude the stability, we first need to improve the estimates of our SPBC $\vec{a}_0$ and initial conditions of a periodic solutions. There are two steps in searching a solution satisfying SPBC. The first step is to find a solution for a fixed boundary and the second step is to vary the boundary to find a minimizer. In this way, we can easily get a good approximation of the initial conditions of the star pentagon and other solutions with an  absolute error tolerance of $10^{-12}$. To check whether the global error is within the expected accuracy, we also compute the monodromy matrix and its eigenvalues with several different step sizes for each case as in \cite{OuXie}. By our computation, the four pairs of eigenvalues are  all real and distinct  in $(-1,1)$. Returning to the full monodromy matrix, the corresponding eigenvalues are distinct and on the unit circle. Therefore, the corresponding periodic solutions are all linearly stable.

%Numerically, a MATLAB program was written using a Runge-Kutta-Fehlberg method to compute the monodromy matrix $X(\mathcal{T})$ of the  reduced linearized Hamiltonian $H_4$ for the star pentagon choreographic solution presented in figure \ref{fig1}. Then we compute $W=\frac{1}{2}(X+X^{-1})$ and  its eigenvalues.\\

%Based on the initial conditions and  $\mathcal{T}=20$, the four pairs of eigenvalues other than the multipliers $+1$ are $[0.75737,$ $0.75737,$ $0.237095,$ $0.237095,$ $- 0.298735,$ $ - 0.298735,$ $- 0.458522,$ $- 0.458522].$
%$[ 0.761537,$  $  0.761537, $  $ 0.235841,$  $ 0.235841,$  $  - 0.299445, $  $ - 0.299445, $  $ - 0.456736, $  $- 0.456736].$ This paper [ 1.05241, 1.05241, 0.75737 - 0.0000000000000875288*i, 0.0000000000000875288*i + 0.75737, 0.237095 - 0.00000000000100314*i, 0.00000000000100314*i + 0.237095, - 0.00000000000342911*i - 0.298735, 0.00000000000342911*i - 0.298735, - 0.000000000000104006*i - 0.458522, 0.000000000000104006*i - 0.458522]
 %The four pairs of eigenvalues are real and distinct  in $(-1,1)$. Returning to the full monodromy matrix, the corresponding eigenvalues are distinct and on the unit circle. Therefore, the star pentagon choreographic solution is linearly stable.

 Here we only list the initial conditions for some stable orbits with different rotation angle $\theta$ and mass ratio $\mu$ in theorem \ref{main2}. \\% Eigenvalues of the monodromy matrix can be computed from the initial conditions.  All the orbits are extended from the initial four pieces connecting from $q(0)$ to $q(T)$ by extension formula \eqref{qet}. We use $T=1$ and $m_1=m_3=1$ in our calculation. As we point out in Remark \ref{remark1}, non-circular minimizers exist for $(\theta,\mu)$ out of the region $\Omega$. Most of the figures can be generated in any Newtonian $n$-body simulation program by using these initial data. \\ %However, some examples are highly unstable and it is hard to produce satisfactory numerical figures.
(1) $(\theta,\mu)=(\frac{4\pi}{5},1)$, $\mathcal{T}=20$.
$$q_1(0)= [-0.2997475302,  0.4125670813], \dot{q}_1(0)=[1.114760563,    0.8099231855];$$
$$q_2(0)=[1.195555973, -0.5096218631], \dot{q}_2(0)=[-0.8600513847, -0.003559696213];$$
$$q_3(0)= [-1.011040523,   1.391577897],\dot{q}_3(0)=[0.01444607736,   0.01049661818];$$
$$q_4(0)= [0.1152320804,  -1.294523115],\dot{q}_4(0)=[ -0.2691552559,   -0.8168601074].$$  %variation 5.133063e+000, [ 0.6676768845, 1.115022917, 0.5099609578, 0.6676768751, 1.115022937, 0.5099609569]
(2) $(\theta,\mu)=(\frac{4\pi}{5},0.5)$, $\mathcal{T}=20$.
 $$q_1(0)= [-0.03365216432, 0.04631823056], \dot{q}_1(0)=[0.8791868243,  0.6387681838];$$
 $$q_2(0)=[1.229294751, -0.7817016926],  \dot{q}_2(0)=[-0.904358996, -0.2059939437];$$
 $$q_3(0)= [-0.762779971,   1.049876561], \dot{q}_3(0)=[-0.1893205096, -0.1375492335];$$
 $$q_4(0)= [0.3635695192,  -1.410687891],\dot{q}_4(0)=[-0.4753736332, -0.7964439574].$$ %variation is 3.070172e+000, [ 0.5350476184, 1.354971279, 0.05725248156, 0.6625690539, 0.674050662, 0.8789715597] max(eig)=1.0021
(3) $(\theta,\mu)=(\frac{4\pi}{5},1.5)$, $\mathcal{T}=20$.
 $$q_1(0)= [-0.4954623785,  0.6819454601], \dot{q}_1(0)=[1.286530639,    0.93472253];$$
  $$q_2(0)=[ 1.17942819, -0.3292799005], \dot{q}_2(0)=[-0.8390962366,  0.1359441271];$$
 $$q_3(0)= [-1.196730568,   1.647158317], \dot{q}_3(0)=[0.1671195441,   0.121421328];$$
 $$q_4(0)= [ -0.05129955936,  -1.223455951], \dot{q}_4(0)=[-0.1300038857, -0.8400400324].$$ %variation is 7.354403e+000variation is 7.354403e+000
(4) $(\theta,\mu)=(\frac{7\pi}{8},1)$, $\mathcal{T}=16$.
   $$q_1(0)= [ -0.3657149699, 0.8829140403], \dot{q}_1(0)=[1.186543081,   0.4914819077];$$
   $$q_2(0)=[1.104628492, -1.169044107], \dot{q}_2(0)=[-0.7831641034,   0.3494036031];$$
     $$q_3(0)= [-0.7844622398,  1.893859379],\dot{q}_3(0)=[-0.09666570501, -0.04003861407];$$
      $$q_4(0)= [ 0.04554871816, -1.607729312],\dot{q}_4(0)=[-0.3067132724,  -0.8008468968].$$%Phi=7pi/8, variation is 4.705669e+000, [ 0.5731698431, 1.502778945, 0.9556592707, 0.5731698911, 1.502778966, 0.9556592488]
(5) $(\theta,\mu)=(\frac{7\pi}{8},1.5)$, $\mathcal{T}=16$.
  $$q_1(0)= [-0.5350773026,   1.291790881], \dot{q}_1(0)=[1.350250694,  0.5592959079];$$
   $$q_2(0)=[1.099754106, -0.9458393211], \dot{q}_2(0)=[-0.7454903418,  0.4733604227];$$
     $$q_3(0)= [-0.9513025743,   2.296647577],\dot{q}_3(0)=[0.05662296548, 0.02345549633];$$
      $$q_4(0)= [ -0.1088341884,  -1.446452984],\dot{q}_4(0)=[-0.1924254315, -0.8618613589].$$ %Phi=7pi/8, variation is 6.792303e+000, [ 0.6540832718, 1.294699266, 1.398224374, 0.5749989621, 1.943136851, 0.6652771972]
 (6) $(\theta,\mu)=(\frac{7\pi}{9},1)$, $\mathcal{T}=36$.
  $$q_1(0)= [-0.27004813,  0.3218308291], \dot{q}_1(0)=[1.071180019,  0.8988296083];$$
  $$q_2(0)=[1.207641964, -0.3501521227], \dot{q}_2(0)=[-0.8488458756, -0.1158622427];$$
   $$q_3(0)= [-1.072721533,   1.278419741], \dot{q}_3(0)=[0.03916771092, 0.03286635046];$$
   $$q_4(0)= [ 0.1351276988,  -1.250098447], \dot{q}_4(0)=[-0.2615018538,  -0.815833716].$$ %Phi=7pi/9, variation is 5.243419e+000, [ 0.7000339698, 1.044489379, 0.4201203102, 0.7000339725, 1.044489378, 0.4201203071]

% $(\theta,\mu)=(\frac{5\pi}{6},1)$, $\mathcal{T}=$.    $$q_1(0)= [], \dot{q}_1(0)=[];$$   $$q_2(0)=[], \dot{q}_2(0)=[];$$     $$q_3(0)= [],\dot{q}_3(0)=[];$$      $$q_4(0)= [],\dot{q}_4(0)=[].$$
\begin{remark}
Without the symplectic reduction, the original Hamiltonian system of the planar four-body problem has 16 dimension. To check the stability of a periodic solution, one can directly compute the eigenvalues of its Monodromy matrix. Thanks a Matlab Program by Professor Robert Vanderbei, the largest absolute values of the eigenvalues for above examples are all 1.0000.
\end{remark}

\section{Other solutions from different SPBC in the planar 4-body problem}
We decide to close our paper by presenting a different SPBC in the planar 4-body problem. They produce some interesting orbits including triple-choreographic solutions where orbits consist of  three closed curves. The configurations formed by four-body with some symmetries can be isosceles triangle with one on the axis of symmetry of the triangle, rectangle, square, diamond, kite, and collinear.   We numerically found lots of periodic solutions by the different combinations of these symmetrical configurations. Some of known planar 4-body periodic orbits can be found by this method.   \\

  \begin{example}\label{ex1} %from program 40.
   The SPBC is  given by two appropriate configuration subspaces $\mathbf{A} \subset (\mathbf{R}^2)^4$ and $\mathbf{B}\subset (\mathbf{R}^2)^4$ as follows.
$$\mathbf{A}=\left\{ \left.  \left( \begin{array}{cc} 0 & a_1\\
0 & a_1-|a_2|\\ 0 & a_1-|a_2|-|a_3|\\
0 & \frac{-(m_1a_1+m_2(a_1-|a_2|)+m_3(a_1-|a_2|-|a_3|))}{m_4} \end{array} \right) \in (\mathbf{R}^2)^4\right| (a_1,a_2,a_3)\in \mathbf{R}^3 \right\},$$
and
$$\mathbf{B}=\left\{ \left. \left( \begin{array}{cc} 0 & -a_6\\
0 & (-m_3a_5-m_4a_5+m_1a_6)/m_2\\ -a_4 & a_5\\
m_3a_4/m_4 & a_5 \end{array} \right)R(\theta)\in (\mathbf{R}^2)^4\right| (a_4,a_5,a_6)\in \mathbf{R}^3 \right\}.$$
Geometrically, four bodies start from a collinear configuration $q(0)\in \mathbf{A}$ (circular spots in figures) and end at a triangle configuration $q(T)\in \mathbf{B}$ (triangular spots in figures).
\end{example}
(1) $\theta=\frac{4\pi}{5}$, $m_1=m_2=m_3=m_4=1$ (see left graph in  Figure \ref{f1ex1}).
 $$q_1(0)= [0,    1.4415], \dot{q}_1(0)=[0.1842, 0]; q_2(0)=[0,    0.7814], \dot{q}_2(0)=[-1.6060,  0];$$
 $$q_3(0)= [0,   -0.5094],\dot{q}_3(0)=[1.4062,  0]; q_4(0)= [ 0,   -1.7134],\dot{q}_4(0)=[0.0156, 0].$$ %program 40 case 2. Phi=4pi/5, variation is 6.164185e+000
(2) $\theta=\frac{4\pi}{5}$, $m_1=m_2= 1.5, m_3=m_4=1$ (see middle graph in  Figure \ref{f1ex1}).
$$q_1(0)= [0,    1.3354], \dot{q}_1(0)=[0.4047, 0]; q_2(0)=[0,    0.5743], \dot{q}_2(0)=[-1.6241,  0];$$
 $$q_3(0)= [0,   -0.8395],\dot{q}_3(0)=[1.6258,  0]; q_4(0)= [ 0,   -2.0250],\dot{q}_4(0)=[0.2033, 0].$$ %program 40 case 2.%Phi=4pi/5, variation is 9.389144e+000, [ 1.335413942, 0.7611442486, 1.41378492, 0.6709714733, 1.418208218, 0.5594732173]
(3) $\theta=\frac{4\pi}{5}$, $m_1=m_2= 0.9, m_3=m_4=1$ (see right graph in  Figure \ref{f1ex1}).
$$q_1(0)= [ 0,    1.4691 ], \dot{q}_1(0)=[0.1300, 0]; q_2(0)=[0,    0.8331], \dot{q}_2(0)=[-1.6031,  0];$$
 $$q_3(0)= [0,   -0.4315],\dot{q}_3(0)=[1.3541,  0]; q_4(0)= [ 0,   -1.6405],\dot{q}_4(0)=[ -0.0283, 0].$$ %program 40 case 2.
\begin{figure}
\includegraphics[height=4.5cm,width=.33\textwidth]{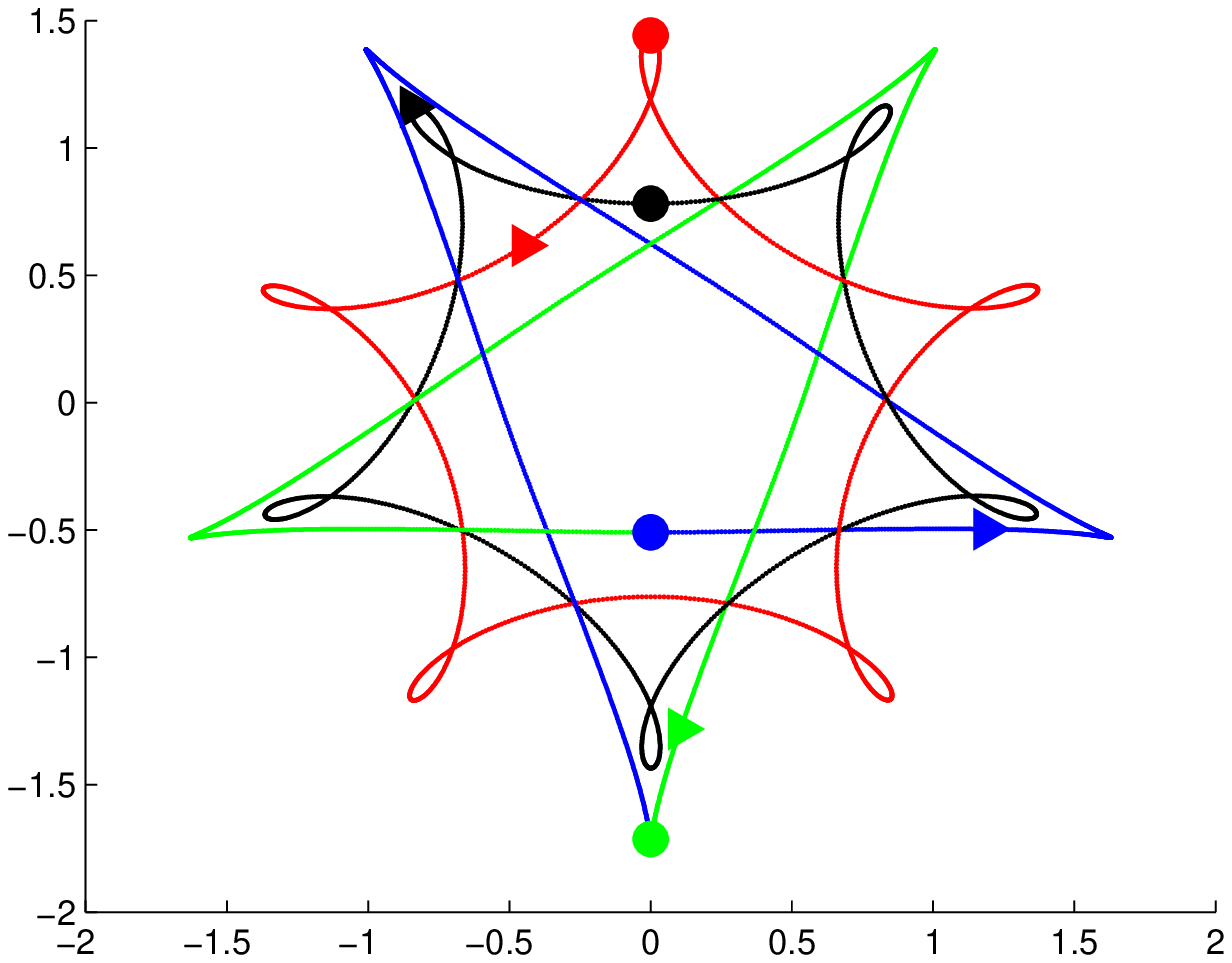}
\includegraphics[height=5cm,width=.32\textwidth]{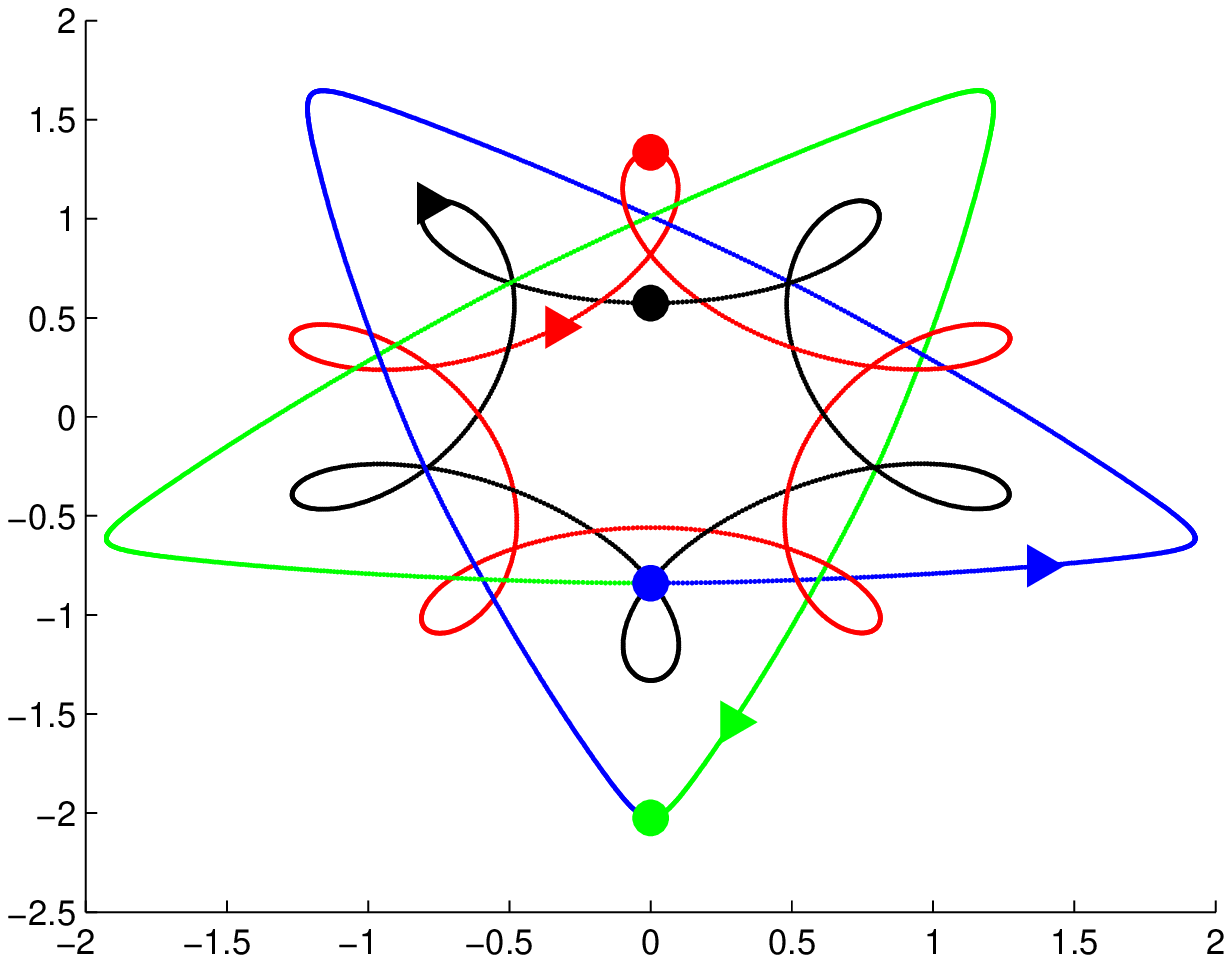}
\includegraphics[height=4.5cm,width=.33\textwidth]{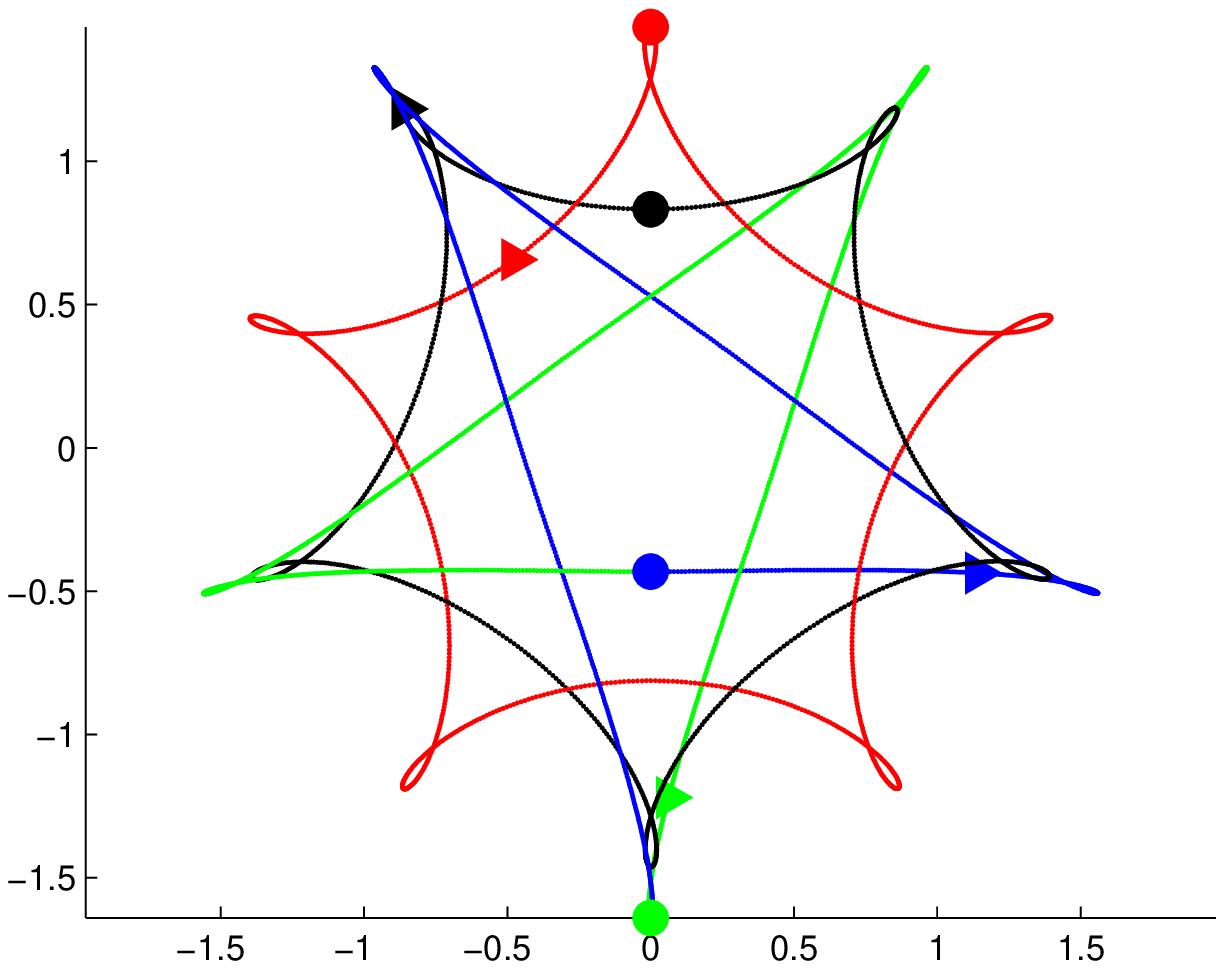}
\caption{\small Triple-choreographic Periodic Solutions for $\theta=\frac{4\pi}{5}$ with different masses in Example \ref{ex1}. }
\label{f1ex1}\end{figure}
\begin{figure}
\includegraphics[height=4.3cm,width=.32\textwidth]{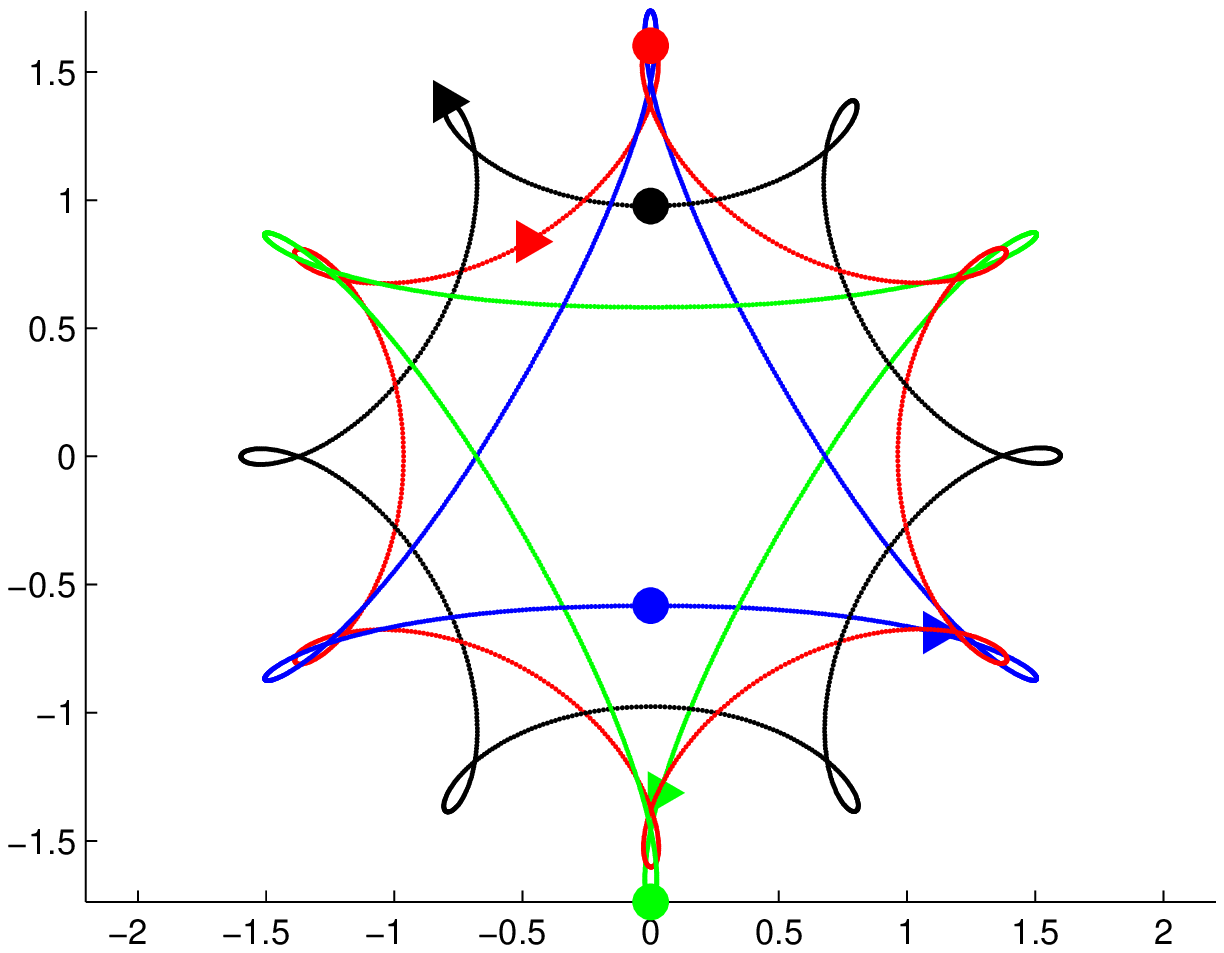}
\includegraphics[height=4.3cm,width=.32\textwidth]{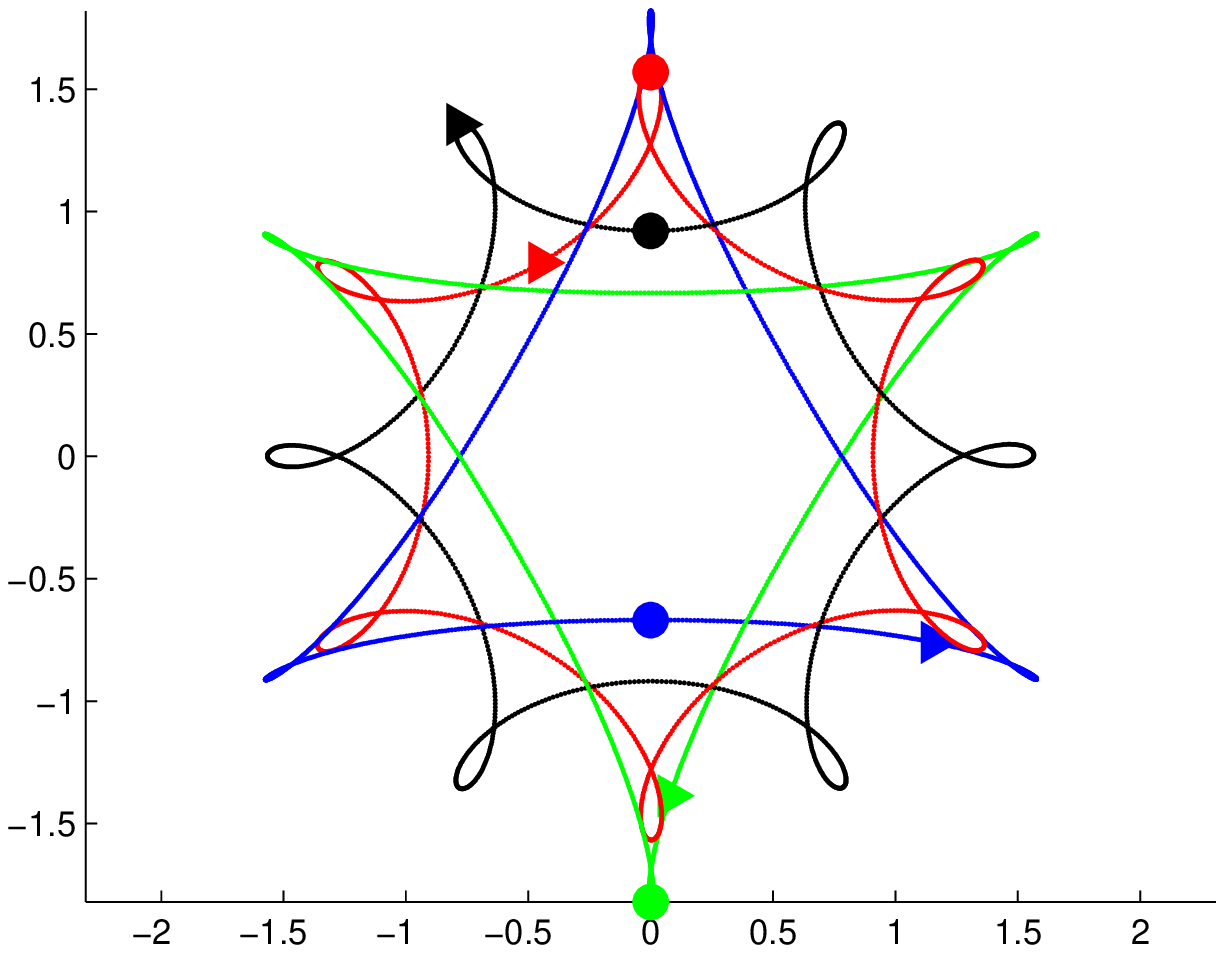}
\includegraphics[height=4.3cm,width=.32\textwidth]{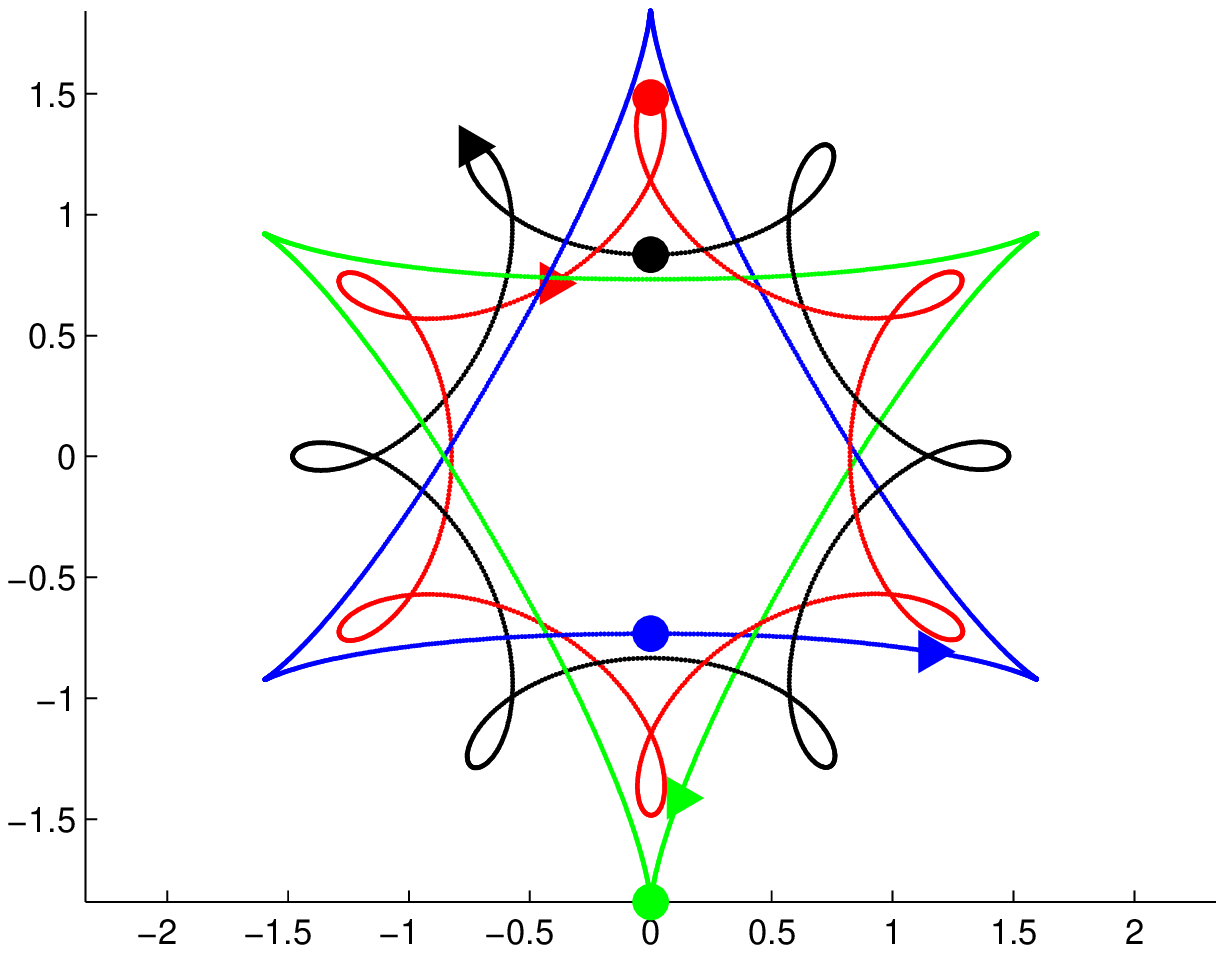}
\caption{\small Non-choreographic Periodic Solutions for $\theta=\frac{5\pi}{6}$ with different masses in Example \ref{ex1}. }
\label{f2ex1}\end{figure}
%E1P3A5f6; Phi=5pi/6, variation is 5.565893e+000 [ 1.484158895, 0.6498662248, 1.5674762, 0.6017079122, 1.280731091, 0.8239182188]
 %[ 0,   1.484158895, 0,   0.2788451039,   0.00001128617652, 0, 1.0]
%[ 0,  0.8342926707, 0,   -1.506304892, -0.000005616866318, 0, 1.0]
%[ 0, -0.7331835298, 0,    1.355128941, -0.000003945132658, 0, 0.9]
%[ 0,  -1.842873766, 0, 0.008715268865, -0.000002355485734, 0, 0.9]
%E1P2A5f6; Phi=5pi/6, variation is 5.400682e+000   [ 1.602294054, 0.6258044347, 1.558897173, 0.621823344, 1.153389197, 0.9643712756]
%[ 0,   1.602294054, 0,   0.1818228921, -0.000001360966039, 0, 0.9]
%[ 0,  0.9764896194, 0,    -1.54781934, -0.000006390259134, 0, 0.9]
%[ 0, -0.5824075539, 0,    1.306124192, -0.000008962413556, 0, 1.0]
%[ 0,  -1.738497752, 0, -0.07672738936,   0.00001593867074, 0, 1.0]
%E1P1A5f6; Phi=5pi/6, variation is 5.972134e+000 [ 1.56947853, 0.6486931907, 1.589689236, 0.6223807227, 1.237825127, 0.9093622557]
  %0    1.5695         0    0.2349    0.0007         0    1.0000
 %        0    0.9208         0   -1.5537   -0.0005         0    1.0000
%         0   -0.6689         0    1.3537    0.0000         0    1.0000
 %        0   -1.8214         0   -0.0349   -0.0002         0    1.0000

%%%%%%%%%%%%%%%%%%%%%%%%%%%%%%%%%%%%%%%%%%%%%%%%%%%%%%%%%%%%%%%%%%%%%%%

{\it Acknowledgements.} This work was partially supported by a grant from the Simons Foundation ($\#$278445 to Zhifu Xie).

\end{document}